\documentclass[]{elsarticle}

\usepackage{lineno,hyperref}

\usepackage{graphicx}
\usepackage{bm}
\usepackage{multirow}
\usepackage{amsmath}
\usepackage{amsfonts}
\usepackage{amssymb}
\usepackage{amsthm}
\usepackage{secdot}
\usepackage{algpseudocode}
\usepackage{esvect}
\usepackage{subfigure}
\usepackage{enumitem}
\usepackage{xcolor}
\usepackage{pgfplots}

\journal{Computer-Aided Design}

\theoremstyle{plain}
\newtheorem{theorem}{Theorem}[section]

\newtheorem{corollary}[theorem]{Corollary}

\newtheorem{remark}[theorem]{Remark}
\newtheorem{thm}{Theorem}

\theoremstyle{definition}

\numberwithin{equation}{section}
\numberwithin{theorem}{section}

\newcommand{\wavespeed}{\mathfrak{c}}

\newcommand{\vertiii}[1]{{\left\vert\kern-0.25ex\left\vert\kern-0.25ex\left\vert #1 
    \right\vert\kern-0.25ex\right\vert\kern-0.25ex\right\vert}}

\begin{document}

\begin{frontmatter}

\title{Anisotropic Delaunay Hypervolume Meshing for Space-Time Applications: Point Insertion, Quality Heuristics, and Bistellar Flips}

\author{Jude T. Anderson \corref{mycorrespondingauthor}}

\cortext[mycorrespondingauthor]{Corresponding author}
\ead{jta29@psu.edu}


\author{David M. Williams} 

\address{Department of Mechanical Engineering, The Pennsylvania State University, University Park, Pennsylvania 16802}

\address{Laboratories for Computational Physics and Fluid Dynamics, Naval Research Laboratory, Washington, DC 20375}

\fntext[fn1]{Distribution Statement A: Approved for public release. Distribution is unlimited.}

\begin{abstract}
    This paper provides a comprehensive guide to generating unconstrained, simplicial, four-dimensional (4D), hypervolume meshes for space-time applications. While several universal procedures for constructing unconstrained, $d$-dimensional, anisotropic Delaunay meshes are already known, many of the explicit implementation details are missing from the relevant literature for cases in which $d \geq 4$.
    As a result, the purpose of this paper is to provide explicit descriptions of the key components in the 4D meshing algorithm: namely, the point-insertion process, geometric predicates, element quality heuristics, and bistellar flips. This paper represents a natural continuation of the work which was pioneered by Anderson et al.~in ``Surface and hypersurface meshing techniques for space-time finite element methods", Computer-Aided Design, 2023. In this previous paper, hypersurface meshes were generated using a novel, trajectory-tracking procedure. In the current paper, we are interested in generating coarse, 4D hypervolume meshes (boundary meshes) which are formed by sequentially inserting points from an existing hypersurface mesh. In the latter portion of this paper, we present numerical experiments which demonstrate the viability of this approach for a simple, convex domain. Although, our main focus is on the generation of hypervolume boundary meshes, the techniques described in this paper are broadly applicable to a much wider range of 4D meshing methods. We note that the more complex topics of \emph{constrained} hypervolume meshing, and boundary recovery for non-convex domains will be covered in a companion paper. 
\end{abstract}


\end{frontmatter}

\section{Introduction}
\label{sec;introduction}

The fundamental goal of volume meshing is to accurately partition a given polyhedral domain by constructing a non-degenerate constrained mesh that includes all of the domain's original vertices and bounding surfaces. A variety of methodologies have been devised to address the complexities of this problem in multiple dimensions, up to and including three dimensions (3D). Within the past few decades, researchers have taken interest in developing new strategies aimed towards four-dimensional (4D), constrained hypervolume meshing---a concept with clear applications to 3D+$t$ space-time problems. The most attractive feature of 4D space-time meshes is their ability to model arbitrary, large-scale boundary motions within a single spatial-temporal framework. This directly facilitates more efficient and accurate simulations of fluid-structure interaction (FSI), and other multi-material and multi-phase problems. For problems with significant boundary motion, one must consider whether the motion is known \emph{a priori} or \emph{a posteriori}. If the motion is known \emph{a priori}, then the 4D mesh generation process is conceptually straightforward, as the domain boundary location is known for all time. However, if the boundary motion is known \emph{a posteriori}, more sophistication is required. In this case, one may use a space-time method in conjunction with boundary-deformation or a predictor-corrector approach. This results in a space-time method with weakly-coupled boundary motion. Conversely, one may simultaneously solve for the solution and the mesh in a fully-coupled fashion. In the latter case, a 4D mesh can serve as an initial guess for the fully-coupled process. Generally speaking, in all of the cases described above, 4D meshes have considerable utility.

To our knowledge, current space-time meshing technologies are mostly limited to semi-structured and fully-structured meshes. In fact, many of these technologies bypass the well-known challenges associated with fully-unstructured hypervolume meshing by using tetrahedral extrusion and prism splitting strategies (see below). As a result, there does not appear to be an explicit method for generating fully-unstructured, boundary-conforming, hypervolume meshes for space-time applications in the literature. In what follows, we begin with a short review of current meshing technologies, and then go on to describe our unique approach to fully-unstructured, space-time hypervolume meshing.

\subsection{Background}

The most widely used methods for volume meshing include the Advancing Front method~\cite{lohner1988generation,george1994advancing}, the Delaunay-based method~\cite{hang2015tetgen,gruau20053d}, and a hybrid of both methods~\cite{mavriplis1995advancing,frey19983d}. The Advancing Front method starts with an existing surface mesh and then generates one element at a time until the associated volume mesh is completed, (i.e.~all unique element faces belong to the initial surface mesh). A key strength of this method is that the original surface mesh is maintained at all times. As a result, the Advancing Front method does not require a post-processing procedure in order to recover the boundary. Unfortunately, the method fails to maintain a valid volume mesh which is free of gaps/cavities at intermediate steps of the iterative process. Therefore, if the method fails for any reason, one is left with unclosed gaps/cavities in the mesh. In an alternative fashion, the Delaunay-based method begins with a simplicial bounding box that contains all of the points associated with the desired volume mesh. These points are inserted one-by-one using a Bower-Watson strategy~\cite{bowyer1981computing, watson1981computing} until the resulting volume mesh includes all of the points. The most attractive feature of this method is that, throughout all iterations, the intermediate volume meshes remain valid, as they do not contain any gaps/cavities. However, it is important to note that meshes generated by this method always require post-processing procedures to recover the boundary. Therefore, if the post-processing procedure fails for any reason, then part of the boundary will be missing.  Lastly, the hybrid Advancing-Front/Delaunay method usually constructs a new volume mesh by using an existing, coarse volume mesh as a starting point. This coarse volume mesh is called a \emph{boundary mesh}, and it consists solely of points from the original surface mesh.  The initial mesh is treated as an empty volume, and new elements are generated through an Advancing Front scheme, where the points on the front are inserted using a Delaunay kernel. Throughout this process, it is possible to retain a valid intermediate mesh at each iteration. Since the hybrid approach generally requires an initial volume mesh, it is more correct to describe it as a mesh \emph{refinement} method, as opposed to a mesh \emph{generation} method. Lo has given detailed summaries of all three methods in a recent book~\cite{lo2014finite}. The Delaunay-based method has proven to be the most commonly used of the three due to its inherent flexibility and well-understood mathematical properties.  An in-depth review of this method can be found in the book by Cheng et al.~\cite{cheng2013delaunay}. 

The traditional volume meshing approaches (above) are often used to construct isotropic meshes. However, anisotropy can still be introduced through a sizing function or a metric field. A sizing function is a scalar field $\phi(x_1, \ldots, x_d)$ which specifies a desirable element size at a point $(x_1, \ldots, x_d)$, where $d$ is the number of dimensions. A metric field $M(x_1, \ldots, x_d)$ is a tensor field which specifies the size \emph{and} orientation of an element at a point. The metric field is especially useful because it enables us to scale different coordinate directions in different ways. Naturally, sizing functions are essentially just metric fields which take the form $\phi \mathbb{I}$, where $\mathbb{I}$ is the identity matrix. We can construct \emph{anisotropic Delaunay meshes}, or \emph{anisotropic Advancing-Front meshes} using metric fields. For example, an anisotropic Delaunay mesh can be constructed by performing the standard Bowyer-Watson point-insertion strategy in metric space. In particular, lengths, dot products, volumes, and related quantities can be computed using the quadratic form induced by the metric field. In addition, due to the anisotropy of the metric field, edges which have a unit length in the metric space can have much larger or smaller lengths in physical space. In this way, anisotropy can be naturally introduced into the Delaunay meshing process. A more detailed overview of anisotropic Delaunay meshes appears in the introduction of~\cite{boissonnat2015anisotropic}, and in Section 2.1 of this paper.

We note that virtually all of the previous work (above) is completely general; however most practical implementations and explicit descriptions have focused on mesh generation in two or three spatial dimensions. In what follows, we turn our attention to mesh-generation techniques which were primarily developed for 4D problems.

In~\cite{behr2008simplex}, Behr developed a popular, extrusion-based method for generating 3D+$t$ space-time hypervolume meshes. The method starts by taking a 3D unstructured tetrahedral volume mesh and extruding its elements through time in order to generate tetrahedral prisms. These prisms are subsequently subdivided into pentatopes (4-simplexes), resulting in a semi-structured 4D hypervolume mesh. The split of each tetrahedral prism is performed carefully so that the resulting pentatopes satisfy the Delaunay criterion. This process is also applicable to 2D+$t$ cases. A number of researchers have applied Behr's method to their work, see~\cite{pauli2017stabilized, von2019simplex, karyofylli2019simplex, make2022spline, karyofylli2022simplex, von2023time, karabelas2015generating, lehrenfeld2015nitsche} for details.

Von Danwitz et al.~\cite{von2021four} expanded Behr's extrusion-based method by using an elastic mesh update method to incorporate time-variant topology. In short, a 4D mesh was progressively conformed to varying surface topology by gradually deforming its elements without changing its original connectivity. 
We note that the elastic-deformation strategy is very effective for small and moderate-scale deformations of the boundary. However, we do not expect the robustness of the method to be maintained in the presence of large-scale surface deformations, and highly anisotropic near-wall meshes.

The process of extruding unstructured tetrahedral meshes in order to obtain hypervolume meshes of pentatopes or tetrahedral prisms is not limited to Behr's work. A significant amount of work on this topic has been performed by Tezduyar and coworkers, (see the review in~\cite{takizawa2023space}). In addition, extrusion-based methods have often been adapted to model purely rotational motions. For example, Wang and Persson~\cite{wang2015high} used this approach in 2D+$t$ to simulate a rotating cross geometry. Their approach combines space-time extrusion-meshing techniques with a \emph{sliding-mesh} approach. Here, a tetrahedral mesh is subdivided into three different regions: a rotating region, a buffer region, and a stationary region. The mesh accommodates large-scale rotations through edge reconnections in the buffer region between the stationary and rotating regions. The same concept was extended to 3D+$t$ for simple cases~\cite{wang2015discontinuousthesis}, including a rotating ellipsoid.  A similar methodology was developed by Horváth and Rhebergen~\cite{horvath2022conforming}. We note that the viability of the sliding mesh approach only holds for purely rotational boundary motion.

The space-time mesh generation method of \emph{pitching tents}~\cite{ungor2002pitching, erickson2005building, abedi2004spacetime} is similar, in a superficial fashion, to traditional Advancing Front techniques. Here, new vertices are generated by identifying existing vertices on the advancing space-time front, and projecting these vertices along the temporal direction in accordance with carefully chosen conditions, (often based on characteristic curves). By tessellating the neighboring faces of an original vertex with its corresponding new vertex, one may construct new, higher-dimensional elements. The tent-pitching method has been applied to hyperbolic systems~\cite{gopalakrishnan2015tent,gopalakrishnan2017mapped,drake2022convergence} and the Maxwell equations~\cite{gopalakrishnan2020explicit}. 
This approach may not be well-suited for problems which are dominated by diffusive phenomena, or other isotropic processes. Fortunately, this issue has been remedied by the next generation of such methods, as discussed below.

The traditional tent-pitching approaches (above) were developed to meet the requirements of causal Space-time Discontinuous Galerkin (cSDG) methods, and related approaches~\cite{richter1994explicit,lowrie1998space,falk1999explicit}. More recently, mesh generation techniques for cSDG methods have seen significant improvement, as exemplified by the excellent work of Haber and colleagues~\cite{miller2008spacetime,palaniappan2008sub,abedi2017effect}. The improved tent-pitching-inspired approaches can be called `adaptive cSDG methods'. In these approaches, the elements of the mesh are grouped into clusters, called patches. These patches satisfy causality and progress constraints, (i.e. they ensure the locality of causation \emph{and} that the advancing space-time front reaches the end of the temporal domain). These new methods allow for a larger range of patch types, which aid in the adaptive meshing process, and enable the alignment of element interfaces with solution discontinuities.  We recommend that the interested reader consult the thesis of Howard~\cite{howard2019spacetime} for a more comprehensive description of this technology.

The topic of \emph{fully-unstructured}, 4D mesh generation has received very limited attention in the literature. The most relevant work on this topic is that of Boissonnat et al.~\cite{boissonnat2009incremental}, where they construct unconstrained, isotropic Delaunay meshes for `medium' numbers of  dimensions, ($d = 2-6$). Their focus is on maximizing efficiency while minimizing the memory footprint. They achieve very good performance through careful selection of the appropriate data-storage structure, point-sorting algorithm, and adjacency search protocol. They do not focus on the $d =4$ case, in particular, nor do they specialize their algorithm for space-time applications. More recently, Foteinos and Chrisochoides~\cite{foteinos20154d}, have developed an isotropic Delaunay-based method to generate unstructured meshes for 4D medical images. The main focus of this work is the development of a 4D Delaunay refinement technique for removing sliver elements. This work does not contain an explicit description of several key aspects of the Delaunay algorithm, including the point insertion process and the procedure for computing geometric predicates in 4D. This work also does not provide a comprehensive treatment of general space-time applications.

While higher-dimensional meshing techniques are often absent from the literature, they are contained (and maintained) in existing software. We note that CGAL (the Computational Geometry Algorithms Library~\cite{fabri2009cgal}) and Qhull~\cite{barber1996quickhull} are open-source libraries for Delaunay meshing algorithms, and related meshing methods. Both codes have the ability to generate unconstrained, isotropic Delaunay meshes in higher dimensions, $d\geq 4$. The precise details of the Delaunay algorithms for the $d=4$ case are not explicitly described in print; however, they are documented in the codes themselves. These codes are excellent resources for practitioners who are looking to develop their own higher-dimensional isotropic Delaunay mesh generators. However, they are missing some important features which may facilitate the development of anisotropic Delaunay meshes for space-time applications, (as we will discuss). 

Interestingly enough, while fully-unstructured 4D mesh generation has seen limited attention in the literature, 4D mesh refinement has been thoroughly explored in recent years. This work has been performed by researchers such as Caplan et al.~\cite{caplan2020four}, Neumüller and Steinbach~\cite{neumuller2011refinement}, and Belda-Ferrín et al.~\cite{belda2023conformal}. A concise review of the literature on this topic is contained in~\cite{caplan2019four}.

\subsection{Motivation and Paper Overview}

The existing research on 4D space-time hypervolume meshing (above) relies heavily on extrusion-based meshing approaches due to their inherent conformity to the boundary, and overall simplicity of implementation. However, these methods are limited when considering arbitrary, large-scale boundary motion and are not fully unstructured in both space and time. In addition we believe that, despite the effectiveness of the tent-pitching and adaptive cSDG strategies, other approaches are needed. In particular,  because of the complexity of space-time applications, multiple meshing options are desirable for flexibility and robustness purposes. For these reasons, we are developing a general constrained, anisotropic Delaunay, 4D hypervolume meshing technique which maintains boundary conformity with fully-unstructured, space-time hypersurface meshes. In previous work~\cite{anderson2023surface}, we developed an explicit method for generating space-time hypersurface meshes. It remains for us to develop fully-unstructured, anisotropic Delaunay-based, hypervolume meshing techniques which conform to these hypersurface meshes. 

In this work, we introduce a new, explicit description of anisotropic Delaunay, 4D, unstructured hypervolume meshing. Our methodology assumes the existence of a predefined hypersurface mesh. Thereafter, the points of the hypersurface mesh are used to construct a coarse, hypervolume mesh -- also referred to as a \emph{hypervolume boundary mesh}. 
Once the hypervolume mesh is generated from the points of the hypersurface mesh, the bounding tetrahedral facets of the hypersurface mesh can be recovered. However, due to the length and complexity of the 4D boundary recovery process, its description will be reserved for subsequent work. In the present work, we will provide novel, explicit descriptions of important aspects of the unconstrained, 4D anisotropic Delaunay mesh generation process. In addition, we will describe post-processing techniques for quality assessment and improvement purposes. The latter techniques are based on the construction of 4D bistellar flips (Pachner moves), the majority of which have never been reported in the literature. A summary of the key contributions appears below:

\begin{itemize}
    \item A rigorous justification of anisotropic Delaunay meshing techniques for space-time applications. To our knowledge, we develop the first, fully-unstructured, space-time Delaunay meshing framework.
    \item Delaunay-conforming subdivisions of the tesseract (4-cube) into pentatopes. Two of these subdivision strategies appear to be new.
    \item A comprehensive description of predicates for anisotropic Delaunay meshes in 4D. To our knowledge, these are the first explicit descriptions of metric-weighted predicates in the literature.
    \item The development of three explicit algebraic heuristics for evaluating the quality of pentatope elements. To our knowledge, these are the first explicit descriptions of such heuristics in the literature.
    \item A complete enumeration of conventional bistellar flips in 4D, along with the introduction of new, extended flips. 
    \item A rigorous grid convergence study which measures the hypervolume error of a sequence of meshes, and obtains the expected order of accuracy. 
\end{itemize}
From our perspective, the present work is clearly distinct from that of previous software, including CGAL~\cite{fabri2009cgal} and Qhull~\cite{barber1996quickhull}, as well as previous researchers, including Boissonnat et al.~\cite{boissonnat2009incremental}, and Foteinos and Chrisochoides~\cite{foteinos20154d}. In particular, this previous work provides a general overview of isotropic Delaunay meshing for higher-dimensional cases, (dimensions greater than 3); however, we focus exclusively on the 4D case, and provide explicit descriptions of the anisotropic Delaunay meshing process. In addition, we provide new contributions, in the form of a space-time Delaunay meshing framework, quality heuristics, and bistellar flips (see above).

The format of this paper is as follows: In section~\ref{sec;compatibility}, we discuss anisotropic Delaunay meshing and its excellent properties for space-time applications. Next, in section~\ref{sec;preliminaries}, we discuss more general concepts related to hypervolume meshing. In sections~\ref{sec;point_insertion} and \ref{sec;metric_predicates}, we discuss the specifics of our point-insertion algorithm, including bounding-box subdivision strategies and 4D metric-weighted  predicates. In section~\ref{sec;quality_improvement}, we introduce quality heuristics for pentatopes, and a new class of bistellar flips for the purposes of improving these heuristics. In section~\ref{sec;numerical_examples}, we present the results of numerical experiments which demonstrate the effectiveness of our mesh-generation and quality-heuristic methods. Finally, in section~\ref{sec;conclusion}, we conclude with a brief discussion of future work on 4D boundary recovery.

\section{Delaunay Meshes and Space-Time Applications} \label{sec;compatibility}

The concept of developing Delaunay meshing techniques for space-time applications is somewhat controversial. Many scientists and engineers view space and time as distinctly different regimes, which must be meshed with different techniques. They view space-time as a tensor product of the Euclidean space $\mathbb{E}^3$ and the real line $\mathbb{R}$. The idea of creating fully unstructured Delaunay meshes directly in $\mathbb{R}^4$ is met with justifiable skepticism. Usually, the reasons for this skepticism can be traced back to one or more of the following issues:

\begin{enumerate}
    \item \emph{Consistency}. The standard inner product on $\mathbb{R}^4$ is dimensionally inconsistent. For example, if we define two points $V_1 = \left[x_1, y_1, z_1, t_1\right]^T$ and $V_2 = \left[x_2, y_2, z_2, t_2\right]^{T}$, the square of the distance between them is given by
    \begin{align*}
        (V_1 -V_2)^{T} (V_1-V_2) = (x_1-x_2)^2 + (y_1-y_2)^2 + (z_1-z_2)^2 + (t_1-t_2)^2.
    \end{align*}
    The dimensions of this inner product are 
    \begin{align*}
        \left(\mathrm{length}\right)^2 + \left(\mathrm{time}\right)^2,
    \end{align*}
    which are inconsistent. 
    \item \emph{Optimality}. One of the main motivations for using Delaunay meshes is their optimality for certain, elliptic boundary-value problems in $\mathbb{E}^2$. In fact, one can obtain \emph{a priori} error estimates which demonstrate the superior effectiveness of Delaunay meshes in this context. However, it is unclear how these optimality properties can be extended to space-time problems. 
    \item  \emph{Slivers}. Delaunay meshes often contain slivers, i.e., elements which have very small volumes, and high aspect ratios. There are well-developed procedures for identifying and removing slivers from Delaunay meshes in $\mathbb{E}^3$. It is not immediately clear how one can rigorously define sliver elements for space-time applications in four dimensions, due to the inconsistency of the standard inner-product dimensions, (see point 1). In addition, a process for removing these slivers is unknown. 
    \item \emph{Causality}. A numerical method operating on a fully-unstructured mesh in $\mathbb{R}^4$ may not maintain causality. In particular, there are concerns about the fully unstructured nature of the mesh, which allows elements with vastly different time scales to interact with one another. There are concerns about the physical accuracy of this process, and in some cases, tent-pitching or similar meshing techniques which guarantee local causality are seen as the only viable option.
    \item \emph{Cost}. Delaunay meshing of an entire, four-dimensional space-time domain is seen as prohibitively expensive. High-fidelity simulations of three-dimensional problems are already computationally demanding, and it seems like adding another dimension makes the problem intractable. 
\end{enumerate}

In what follows, we will address each of these concerns, and provide an important alternative perspective.

\subsection{On Consistency} \label{consistency_section}

The dimensional inconsistency of the standard inner product in $\mathbb{R}^4$ can be addressed by scaling the temporal component of each 4-vector by a characteristic speed $\wavespeed = \wavespeed(x,y,z,t)$. In particular, we can compute the square of the distance as follows
\begin{align*}
    (V_{1}-V_{2})^{T} M (V_1-V_{2}) = (x_1 -x_2)^2 + (y_1- y_2)^2 + (z_1-z_2)^2 + \wavespeed^2(x,y,z,t) (t_1-t_2)^2.
\end{align*}
where $M$ is a symmetric, positive-definite metric tensor in $\mathbb{R}^{4 \times 4}$
\begin{align}
    M = M(x,y,z,t) = \begin{bmatrix} 1 & 0 & 0 & 0\\
    0 & 1 & 0 & 0 \\
    0 & 0 & 1 & 0 \\
    0 & 0 & 0& \wavespeed^{2}(x,y,z,t) \end{bmatrix}.
\end{align} \label{metric_space_time}
The dimensions of this modified inner product are
\begin{align*}
    \left(\mathrm{length}\right)^2 + \left(\mathrm{length}/\mathrm{time}\right)^2 \left(\mathrm{time}\right)^2 = \left(\mathrm{length}\right)^2,
\end{align*}
which are consistent. In a nearly identical fashion, one can replace the vectors $V_1$ and $V_2$ with vectors
\begin{align*}
    \mathcal{V}_1 = \left[x_1,y_1,z_1, \wavespeed(x,y,z,t) t_1\right]^{T}, \qquad \mathcal{V}_2 = \left[x_2,y_2,z_2, \wavespeed(x,y,z,t) t_2\right]^{T}.
\end{align*}
or equivalently
\begin{align*}
    \mathcal{V}_1 = \left[x_1,y_1,z_1,  \tilde{t}_1\right]^{T}, \qquad \mathcal{V}_2 = \left[x_2,y_2,z_2, \tilde{t}_2\right]^{T},
\end{align*}
where $\tilde{t}_1 = \wavespeed(x,y,z,t) t_1$ and $\tilde{t}_2 = \wavespeed(x,y,z,t) t_2$. Evidently,
\begin{align*}
        (\mathcal{V}_{1}-\mathcal{V}_{2})^{T} (\mathcal{V}_1-\mathcal{V}_2) = (x_1 -x_2)^2 + (y_1- y_2)^2 + (z_1-z_2)^2 + \wavespeed^2(x,y,z,t) (t_1-t_2)^2.
\end{align*}
In both cases, the dot product takes an identical, dimensionally consistent form. The vectors $\mathcal{V}_1$ and $\mathcal{V}_2$ are inspired by the work of physicists, who sometimes refer to these vectors as the 4-position vectors or `event' vectors~\cite{misner1973gravitation}.

We note that the following formulation for the metric is also possible
\begin{align*}
    M =  \begin{bmatrix}
        1 & 0 & 0& m_{14}(x,y,z,t) \\
        0 & 1 & 0& m_{24}(x,y,z,t) \\
        0 & 0 & 1& m_{34}(x,y,z,t) \\
        m_{14}(x,y,z,t) & m_{24}(x,y,z,t) & m_{34}(x,y,z,t) & m_{44}(x,y,z,t)
    \end{bmatrix},
\end{align*}
where $m_{14}$, $m_{24}$, and $m_{34}$ are characteristics speeds, and $m_{44}$ is a characteristic speed squared. The resulting inner product is dimensionally consistent, but the physical meaning is unclear. Therefore, we defer to the formulation in~Eq.~\eqref{metric_space_time}.

The metric $M = M(x,y,z,t)$ is assumed to be a continuous function of space and time. For implementation purposes, we insist that $M$ is evaluated at a single point, or at a finite series of points in the space-time domain. Consider a particular point of interest $q = (x_q, y_q, z_q, t_q)$. This point, for example, could be the next candidate for insertion by our Delaunay algorithm. We can then choose $M = M_q = M(x_q,y_q,z_q,t_q)$ for each point $q$, enforce the Delaunay criterion using the modified inner-product, and still obtain a valid Delaunay mesh, (see~\cite{borouchaki1997delaunay}, section 5.2). But, this fact comes with an important caveat: a Delaunay mesh based on the modified inner product (above) is technically not a standard Delaunay mesh, but rather an anisotropic Delaunay mesh. Delaunay meshes defined in this fashion, with a generic metric tensor $M$, were originally pioneered by the work of Borouchaki, George, Frey, and coworkers~\cite{borouchaki1997delaunay,borouchaki1997delaunayb,dobrzynski2008anisotropic}, and extended by Boissonnat, Wormser, Yvinec, and coworkers~\cite{boissonnat2015anisotropic,boissonnat2015anisotropicb}. More recent work in this area appears in~\cite{rouxel2016discretized,boissonnat2018geometric,boissonnat2018delaunay,boissonnat2019anisotropic}. The key idea is that anisotropy is introduced by the metric field $M(x,y,z,t)$: namely, the Delaunay criterion is enforced in metric space, and this results in an anisotropic mesh in physical space, (as mentioned previously). Furthermore, different metric fields can induce different anisotropic Delaunay meshes for the same (or similar) sets of points. This may disturb some readers, but we note that metric-tensor induced anisotropy can result in optimal meshes in some scenarios, (as will be shown). Also, even in the absence of optimality, we note that the anisotropic Delaunay approach enables practitioners to produce valid initial meshes for space-time applications in $\mathbb{R}^4$. This is important, as oftentimes, the primary challenge is to simply produce \emph{an} initial mesh, not necessarily an optimal mesh. The precise details of this initial mesh may be less important, as it will often be significantly modified during the course of an adaptive-meshing process.


\subsection{On Optimality} \label{optimal_sub_sec}

It is well-known that Delaunay meshes are optimal for elliptic boundary-value problems in $\mathbb{E}^2$~\cite{rippa1990minimum,rippa1990minimal,powar1992minimal}. In particular, they act as `energy minimizing triangulations,' as they minimize the energy functional of conventional finite element methods which operate on elliptic problems. In addition, a Delaunay mesh minimizes the error as measured by the energy norm, (see~\cite{rippa1990minimum}). 

In order to fix ideas, let us review the argument which was originally presented in~\cite{rippa1990minimum}. Consider the following elliptic problem
\begin{align*}
    -\Delta u = f \quad \mathrm{in} \quad \Omega, \qquad u\Big|_{\partial \Omega} = 0,
\end{align*}
where $\Delta = \nabla \cdot \nabla$ is the Laplacian, $\nabla = \left(\frac{\partial}{\partial x}, \frac{\partial}{\partial y} \right)$ is the gradient, $u \in H^{1}_{0}(\Omega)$, $f \in L_{2}(\Omega)$, and $\Omega$ is a simply-connected spatial domain in $\mathbb{R}^2$. The associated finite element method is 
\begin{align}
    a(u_h,v_h) = L(v_h), \label{fem_form}
\end{align}
\begin{align*}
    a(u_h, v_h) = \int_{\Omega} \left(\frac{\partial u_h}{\partial x} \frac{\partial v_h}{\partial x} + \frac{\partial u_h}{\partial y} \frac{\partial v_h}{\partial y} \right) dx dy, \qquad L(v_h) = \int_{\Omega} fv_h \, dx dy,
\end{align*}
where $u_h \in V_h$, $v_h \in V_h$, and $V_h \subset H^{1}_{0}(\Omega)$. This problem has a unique solution if $a(u_h,v_h)$ is symmetric, continuous, and coercive. We can establish symmetry of $a$ by inspection. We can also show that
\begin{align*}
    \left|a(u_h,v_h) \right| \leq \gamma \left\|u_h\right\| \left\|v_h \right\|, \qquad a(v_h,v_h) \geq \psi \left\| v_h \right\|^2, 
\end{align*}
which establishes continuity and coercivity of $a$. Here, $\gamma \in \mathbb{R}^{+}$, $\psi \in \mathbb{R}^{+}$, and
\begin{align*}
    \left\| v_h \right\| = \left[ \int_{\Omega} \left( v_{h}^{2} + \left(\frac{\partial v_h}{\partial x} \right)^2  + \left(\frac{\partial v_h}{\partial y} \right)^2 \right) dx dy \right]^{1/2}.
\end{align*}
Note that coercivity (above) follows from the standard Poincar\'e inequality
\begin{align*}
    \int_{\Omega} v_{h}^{2} \, dx dy \leq C \int_{\Omega} \left( \left(\frac{\partial v_h}{\partial x} \right)^2  + \left(\frac{\partial v_h}{\partial y} \right)^2\right) dx dy.   
\end{align*}
In addition, we can compute the following energy norm
\begin{align*}
    \left\| v_h \right\|_{a} = \sqrt{a(v_h, v_h)}.
\end{align*}
Next, we can define the energy functional $J(v_h)$ as follows
\begin{align}
    J(v_h) = a(v_h,v_h) - 2L(v_h).
    \label{energy_func}
\end{align}
The solution $u_h$ of Eq.~\eqref{fem_form} and the minimizer of $J(v_h)$,  are directly related. In particular, the following statement holds true
\begin{align*}
    J(u_h) = \min_{v_h \in V_h} J(v_h).
\end{align*}
Now, let $u_{h,1} \in V_h$ and $u_{h,2} \in V_h$ be two different solutions on meshes $\mathcal{T}_1$ and $\mathcal{T}_2$, respectively. We assume that these meshes contain the same points, but possess different connectivity. We can relate the energy functionals and the energy norms of the two solutions as follows
\begin{align*}
    \left\| u - u_{h,1} \right\|_{a}^{2} - \left\| u - u_{h,2} \right\|_{a}^{2} = J(u_{h,1}) - J(u_{h,2}),
\end{align*}
or equivalently
\begin{align*}
    \left\| u - u_{h,1} \right\|_{a}^{2}  = \left(J(u_{h,1}) - J(u_{h,2})\right) + \left\| u - u_{h,2} \right\|_{a}^{2}.
\end{align*}
If $\mathcal{T}_1$ is a standard (isotropic) Delaunay mesh, then we are guaranteed that $J(u_{h,1}) \leq J(u_{h,2})$ in accordance with~\cite{rippa1990minimum,rippa1990minimal}, and furthermore, that
\begin{align*}
     \left\| u - u_{h,1} \right\|_{a} \leq \left\| u - u_{h,2} \right\|_{a},
\end{align*}
In other words, the Delaunay mesh yields the minimum error between the numerical solution $u_h$ and the exact solution $u$, as it's the only mesh which is guaranteed to minimize $J(u_h)$.

Now, we claim that the same Delaunay optimality property can be extended to space-time problems in $\mathbb{R}^2$, via the following coordinate substitution 
\begin{align}
    x \Rightarrow x, \qquad y \Rightarrow \wavespeed(x) t, \label{key_sub}
\end{align}
where for simplicity's sake we assume that $\wavespeed \neq \wavespeed(t)$, $\wavespeed(x)>0$, and $\wavespeed(x) \in C^{\infty}(\mathbb{R})$; in other words, $\wavespeed(x)$ is a positive smooth scalar function. 

In order to establish the claim above, one may consider the following problem
\begin{align}
    -\left(\widetilde{\Delta} \widetilde{u} + \alpha(x) \frac{\partial \widetilde{u}}{\partial x} \right)  = g \quad \mathrm{in} \quad \widetilde{\Omega}, \qquad u\Big|_{\partial \widetilde{\Omega}} = 0, \label{space_time_example}
\end{align}
where $\widetilde{\Delta} = \widetilde{\nabla} \cdot \widetilde{\nabla}$ is the space-time Laplacian, $\widetilde{\nabla} = \left(\frac{\partial}{\partial x}, \frac{1}{\wavespeed(x)} \frac{\partial}{\partial t} \right)$ is the space-time gradient, $\alpha(x) = \frac{1}{\wavespeed(x)} \frac{\partial \wavespeed(x)}{\partial x}$ is a spatial advection-rate coefficient, $\widetilde{u} \in H^{1}_{0}(\widetilde{\Omega})$, $g \in L_{2}(\widetilde{\Omega})$, and $\widetilde{\Omega}$ is a simply-connected space-time domain in $\mathbb{R}^2$. 

We can replace $\widetilde{u}$ with $\widetilde{u}_h$ in Eq.~\eqref{space_time_example}, multiply by a test function $\widetilde{v}_h$, multiply by the scaling function $\wavespeed(x)$, and integrate over the domain in order to obtain
\begin{align*}
    -\int_{\widetilde{\Omega}} \left( \frac{\partial^2 \widetilde{u}_h}{\partial x^2} + \frac{1}{\wavespeed^2(x)}\frac{\partial^2 \widetilde{u}_h}{\partial t^2} \right) \widetilde{v}_{h} \wavespeed(x) \, dx dt -\int_{\widetilde{\Omega}} \frac{\partial \wavespeed(x)}{\partial x} \frac{\partial \widetilde{u}_h}{\partial x} \widetilde{v}_h  dx dt = \int_{\widetilde{\Omega}} g\widetilde{v}_{h} \wavespeed(x) \, dx dt,
\end{align*}
or equivalently, after integration by parts
\begin{align*}
    &\int_{\widetilde{\Omega}} \left(\frac{\partial \wavespeed(x)}{\partial x} \frac{\partial \widetilde{u}_h}{\partial x} \widetilde{v}_h + \wavespeed(x) \frac{\partial \widetilde{u}_h}{\partial x} \frac{\partial \widetilde{v}_h}{\partial x} + \frac{1}{\wavespeed(x)} \frac{\partial \widetilde{u}_h}{\partial t} \frac{\partial \widetilde{v}_h}{\partial t} \right) dx dt -\int_{\widetilde{\Omega}} \frac{\partial \wavespeed(x)}{\partial x} \frac{\partial \widetilde{u}_h}{\partial x} \widetilde{v}_h  dx dt \\
    & - \int_{\partial \widetilde{\Omega}} \left( \wavespeed(x) \frac{\partial \widetilde{u}_h}{\partial x} dt - \frac{1}{\wavespeed(x)} \frac{\partial \widetilde{u}_h}{\partial t} dx \right) \widetilde{v}_h = \int_{\widetilde{\Omega}} g\widetilde{v}_{h} \wavespeed(x) \, dx dt.
\end{align*}
The associated finite element method is 
\begin{align}
    b(\widetilde{u}_h,\widetilde{v}_h) =\mathcal{L}(\widetilde{v}_h), \label{fem_form_space_time}
\end{align}
\begin{align*}
    b(\widetilde{u}_h, \widetilde{v}_h) &= \int_{\widetilde{\Omega}} \left(\wavespeed(x) \frac{\partial \widetilde{u}_h}{\partial x} \frac{\partial \widetilde{v}_h}{\partial x}  + \frac{1}{\wavespeed(x)} \frac{\partial \widetilde{u}_h}{\partial t} \frac{\partial \widetilde{v}_h}{\partial t} \right) dx dt, \qquad 
    \mathcal{L}(\widetilde{v}_h) &= \int_{\widetilde{\Omega}} g\widetilde{v}_h \wavespeed(x) \, dx dt,
\end{align*}
where $\widetilde{u}_h \in \widetilde{V}_h$, $\widetilde{v}_h \in \widetilde{V}_h$, and $\widetilde{V}_h \subset H^{1}_{0}(\widetilde{\Omega})$. In order to guarantee the existence and uniqueness of a solution, we require
\begin{align*}
    \left|b(\widetilde{u}_h,\widetilde{v}_h) \right| \leq \varrho \vertiii{\widetilde{u}_h} \vertiii{\widetilde{v}_h}, \qquad b(\widetilde{v}_h,\widetilde{v}_h) \geq \vartheta \vertiii{\widetilde{v}_h}^2, 
\end{align*}
where
\begin{align*}
    \vertiii{v_h} = \left[\int_{\widetilde{\Omega}} \left( \wavespeed(x) \left( \widetilde{v}_{h}^{2} + \left(\frac{\partial \widetilde{v}_h}{\partial x} \right)^2\right)  + \frac{1}{\wavespeed(x)} \left(\frac{\partial \widetilde{v}_h}{\partial t} \right)^2 \right) dx dt \right]^{1/2}.
\end{align*}
Note that, in order to ensure coercivity (above), one requires the following weighted Poincar\'e inequality to hold
\begin{align*}
    \int_{\Omega} \widetilde{v}_{h}^{2} \wavespeed(x) \, dx dy \leq C \int_{\Omega} \left( \wavespeed(x) \left(\frac{\partial \widetilde{v}_h}{\partial x}\right)^2  + \frac{1}{\wavespeed(x)} \left(  \frac{\partial \widetilde{v}_h}{\partial t} \right)^2 \right) dx dy.   
\end{align*}
This is effectively a constraint on $\wavespeed(x)$.
Next, the energy norm and functional are given by
\begin{align}
     \vertiii{\widetilde{v}_h}_{b} = \sqrt{b(\widetilde{v}_h, \widetilde{v}_h)}, \qquad \mathcal{J}(\widetilde{v}_h) = b(\widetilde{v}_h,\widetilde{v}_h) - 2\mathcal{L}(\widetilde{v}_h). \label{energy_function_new}
\end{align}
Now, in accordance with the previous discussion, we can construct two solutions, $\widetilde{u}_{h,1}$ and $\widetilde{u}_{h,2}$ on separate space-time grids $\mathbb{T}_1$ and $\mathbb{T}_2$. We seek to minimize $\mathcal{J}(\widetilde{u}_h)$ on $\mathbb{T}_1$, such that
\begin{align*}
    \vertiii{ \widetilde{u} - \widetilde{u}_{h,1}}_{b}^{2}  = \left(\mathcal{J}(\widetilde{u}_{h,1}) - \mathcal{J}(\widetilde{u}_{h,2})\right) + \vertiii{\widetilde{u} - \widetilde{u}_{h,2}}_{b}^{2}.
\end{align*}
and
\begin{align*}
     \vertiii{ \widetilde{u} - \widetilde{u}_{h,1} }_{b} \leq \vertiii{ \widetilde{u} - \widetilde{u}_{h,2}}_{b}.
\end{align*}
It seems natural to attempt to minimize the energy functional $\mathcal{J}(\widetilde{u}_h)$ by choosing $\mathbb{T}_1$ to be a Delaunay mesh, since this choice is guaranteed to minimize $J(u_h)$ in the isotropic case. It turns out that this intuition is correct, but it comes with an important caveat: the mesh we seek is an \emph{anisotropic} Delaunay mesh. This fact is rigorously shown in~\ref{appendix_rough} through a careful extension of the arguments in~\cite{rippa1990minimal,powar1992minimal}. Broadly speaking, because of the coordinate substitution we performed in Eq.~\eqref{key_sub}, we must obtain a Delaunay mesh with respect to the following metric tensor in $\mathbb{R}^{2\times 2}$
\begin{align*}
    M = \begin{bmatrix}
        1 & 0 \\
        0 & \wavespeed^2(x)
    \end{bmatrix}.
\end{align*}
Equivalently, we must obtain a Delaunay mesh such that the distance between points $P_1 = \left[x_1, t_1\right]^{T}$ and $P_2 = \left[x_2, t_2\right]^{T}$ is measured as follows
\begin{align*}
    d_{M}(P_1, P_2) = (P_{2} -P_{1})^{T} M (P_{2} - P_1) = (\mathcal{P}_2 -\mathcal{P}_1)^{T}(\mathcal{P}_2 - \mathcal{P}_1), 
\end{align*}
where
\begin{align*}
    \mathcal{P}_1 = (x_1, \wavespeed(x) t_1) = (x_1, \widetilde{t}_1), \qquad \mathcal{P}_2 = (x_2, \wavespeed(x) t_2) = (x_2, \widetilde{t}_2).
\end{align*}

\begin{remark}
The space-time problem in Eq.~\eqref{space_time_example} is relevant for practical applications. In particular, we are often interested in computing an initial guess for the solution throughout a space-time domain based on the information on a boundary. This information can be diffused into the domain by solving the four-dimensional analog of Eq.~\eqref{space_time_example}, i.e.,
\begin{align}
    -\widetilde{\Delta} \widetilde{u} = g \quad \mathrm{in} \quad \widetilde{\Omega}, \qquad u\Big|_{\partial \widetilde{\Omega}} = 0, \label{space_time_example_new}
\end{align}
where $\widetilde{\Delta} = \frac{\partial^2}{\partial x^2}+ \frac{\partial^2}{\partial y^2}+ \frac{\partial^2}{\partial z^2}+ \frac{1}{\wavespeed^2(x,y,z)} \frac{\partial^2}{\partial t^2}$ is the space-time Laplacian, $\widetilde{u} \in H^{1}_{0}(\widetilde{\Omega})$, $g \in L_{2}(\widetilde{\Omega})$, and $\widetilde{\Omega}$ is a space-time domain in $\mathbb{R}^4$. We note that a non-zero boundary condition can be accommodated quite easily (see~\cite{rippa1990minimum}), and does not affect the optimality analysis given above. In particular, the optimality of Delaunay triangulations in two dimensions is still maintained for elliptic problems with non-homogeneous boundary conditions. 
\end{remark}

\subsection{On Slivers}

A methodology for detecting slivers in arbitrary-dimensional Delaunay meshes has already been proposed by Li~\cite{li2003generating}. In order to fix ideas, consider the following example: suppose that we have constructed a 4-simplex with five vertices $p_0$, $p_1$, $p_2$, $p_3$, and $p_4$. The tetrahedral facet opposite $p_0$ is given by $(p_1,p_2,p_3,p_4)$. This facet is contained in an affine hyperplane. In order for the 4-simplex to be a sliver, the minimum distance between $p_0$ and the hyperplane associated with $(p_1,p_2,p_3,p_4)$ cannot be large~\cite{edelsbrunner2000smoothing}. This should hold true for all combinations of vertices with opposite facets in the simplex. Therefore, slivers can be identified via a min-max property: namely, we find elements for which the minimum distance between a vertex and its opposite hyperplane is not large enough, i.e., the minimum distance drops below a certain threshold. Furthermore, following some careful calculations, we can associate this minimum distance with a region, often referred to as the `forbidden region'~\cite{edelsbrunner2000smoothing}. 

A set of helpful criteria for identifying the minimum distance and computing the forbidden region are summarized in~\cite{li2003generating}. For our purposes, it suffices to observe that these quantities are easily computable: namely, the minimum distance and forbidden region are all functions of the dimensionality of the simplex, the shortest edge length of the simplex, the hypercircumsphere hypervolume, the hypercircumsphere radius, and similar quantities defined on the opposite facets. The entire formulation can be adjusted on an as-needed-basis through two user-specified constants. Broadly speaking, the identification of slivers in $\mathbb{R}^n$ is a solved problem. 

For space-time meshes, the process of identifying slivers is also straightforward. One may simply follow the approach of~\cite{boissonnat2015anisotropic}, and define all of the input parameters (the shortest edge length, hypercircumsphere hypervolume, etc.) in terms of the metric field of Eq.~\eqref{metric_space_time}. In other words, the edge lengths, radii, etc. are computed using the quadratic form induced by the metric. A detailed procedure for eliminating slivers from anisotropic Delaunay meshes is already discussed in~\cite{boissonnat2015anisotropic}. A rigorous exploration of this process is beyond the scope of the present article; but our  main point is that the process for identifying and removing slivers is well-developed for anisotropic Delaunay meshes, and immediately extends to the present context. 

\subsection{On Causality}

Causality is a fundamental principle of physics. For example, if a wave is traveling in a particular direction with a given speed, then we expect the current state of the wave to be primarily influenced by its previous time history. This is particularly common for hyperbolic systems, in which disturbances travel (at some finite speed) along characteristics. If the system is linear, then there is a cone of influence associated with the characteristics. 
This concept is not universally relevant for all unsteady systems, especially if the governing equations are reversible. However, if one determines that causality is important for a particular application, then there are several approaches which can be used to  enforce it, even on fully unstructured grids.

A simple solution is to divide the space-time domain into space-time slabs, and then construct fully unstructured meshes on each slab. The earliest time slabs can be solved before the later time slabs, and causality can be weakly enforced in this fashion. One may even adapt the unstructured mesh within a given time slab in order to improve the quality of the solution on that slab. This idea is merely a natural extension of the synchronous slab-wise approaches of~\cite{hulbert1990space,Shakib91,masud1997space,tezduyar2007modelling}.

Causality can also be enforced directly through the finite volume or finite element method one chooses. In fact, the idea of `upwinding' is already well-established in the relevant literature~\cite{hesthaven2007nodal,toro2013riemann}. This concept can be extended to accommodate both space and time, as the inviscid flux for a fluid dynamics problem can be combined with the temporal flux, and then treated as a single temporal-inviscid flux, (see~\cite{wang2015high,wang2015discontinuousthesis,caplan2019four} for details). For example, in the context of discontinuous finite element methods, the solution in adjacent elements is coupled via a numerical flux, and this numerical flux can be biased towards earlier times (i.e.~upwinding in time)~\cite{jamet1978galerkin,van2002space,wang2015high,wang2015discontinuousthesis,caplan2019four,sivas2021air}. In a similar fashion, in the context of continuous finite elements, the solution is often stabilized using a residual-based term. In principle, this term can be upwind-biased in time, in order to enforce causality.  

There is already strong evidence that physically correct solutions can be obtained on fully unstructured grids, in conjunction with carefully chosen finite element or finite volume methods. We refer the interested reader to a linear advection-diffusion example, (see chapter 4 of Caplan~\cite{caplan2019four}), and a moving shockwave example, (see Corrigan et al.~\cite{corrigan2019moving}) to see evidence of this fact for finite element methods in four dimensions. In addition, Nishikawa and Padway~\cite{nishikawa2020adaptive,padway2022adaptive,padway2022anadaptive} have demonstrated similar results for finite volume methods in three dimensions. In all cases, the solution was obtained all-at-once, on a single space-time domain with a fully unstructured, adapted mesh. 

\subsection{On Cost}

The costs of space-time methods can be divided into at least three categories: i) implementation cost, ii) computational cost, and iii) memory footprint/data-storage cost. Most approaches for solving unsteady problems will require significant implementation costs; it remains for us to limit the latter two costs. Fortunately, these latter costs can be ameliorated with domain decomposition techniques, mesh adaptation, and solver acceleration strategies. In the previous section, we discussed the possibility of subdividing a space-time domain into space-time slabs, (i.e. space-time domain decomposition), and thereafter, obtaining the solution on the slabs in a sequential fashion~\cite{Shakib91}. The degrees of freedom on each domain can be further reduced by employing solution-adaptive mesh refinement strategies~\cite{caplan2019four,caplan2020four}. 
During this process, the solution is obtained on an initial coarse mesh, and then the mesh is subsequently refined multiple times in accordance with a problem-specific error indicator. In addition, depending on the precise nature of the mesh family, one may be able to employ a space-time multigrid strategy in order to accelerate convergence of the solution~\cite{horton1995space,weinzierl2012geometric,gander2016analysis,anselmann2023geometric}. This strategy can be employed as a preconditioning strategy for Newton-Krylov methods, or as a direct solver for the full approximation scheme (FAS) approach.

In light of the discussion above, it is important to emphasize that most space-time simulations with fully unstructured meshes \emph{will} be more expensive than their time-stepping counterparts. However, this cost can be controlled using the (above) strategies. In addition, space-time methods on unstructured grids offer an incredible opportunity to produce unparalleled space-time resolution of important solution features (as discussed previously). 

\subsection{Summary}

For all of the reasons discussed above, we believe that Delaunay space-time meshes are an important building block in the development of space-time methods overall. We acknowledge that other meshing approaches exist, but we believe that this fact is \emph{positive}. The field as a whole will benefit from having multiple meshing options which are useful for different applications.

\section{Preliminaries for Hypervolume Meshing}
\label{sec;preliminaries}

In this section, we establish a pair of principles which enable us to characterize the connectivity and orientation of elements.

\subsection{Consistent Orientation} \label{consistent_orientation}

It is import to note that all initial elements, along with all subsequent elements added during the point insertion process, are required to have a consistent, positive orientation, (facet normals pointing inward).  Standard operations within the anisotropic Delaunay algorithm rely on this consistency and will not work without it. The geometric predicate for determining the orientation of a pentatope is discussed in section~\ref{sec;metric_predicates}. To maintain additional consistency, whenever a pentatope $P(p_1,p_2,p_3,p_4,p_5)$ is associated with its corresponding facets, the facets will be defined as follows
\begin{align*}
P(p_1,p_2,p_3,p_4,p_5) \quad \rightarrow \quad &F(p_1,p_2,p_3,p_4), \, F(p_1,p_2,p_5,p_3), \, F(p_1,p_2,p_4,p_5),\\
                                         &F(p_2,p_3,p_4,p_5), \, F(p_3,p_1,p_4,p_5).
\end{align*}

\subsection{Facet Normals} \label{hypersurface_normals}

We can now discuss our process for constructing facet normals. In this work, hypersurface normal vectors for tetrahedral facets are calculated as follows
\begin{align*}
N = 
\begin{vmatrix}
u_x & u_y & u_z & u_t \\
v_x & v_y & v_z & v_t \\
w_x & w_y & w_z & w_t \\
i & j & k & l
\end{vmatrix},
\end{align*}
for a tetrahedral facet $F(a,b,c,d)$, where $u = b-a$, $v = c-a$, and $w = d-a$ as shown in work from Kercher et al.~\cite{kercher2021moving}. The coordinate directions are represented as $(x,y,z,t) \rightarrow (i,j,k,l)$.

\section{Point Insertion}
\label{sec;point_insertion}

The core of any Delaunay-based meshing algorithm is point insertion. A standard point insertion algorithm, see e.g.~the Bowyer-Watson algorithm~\cite{bowyer1981computing, watson1981computing}, involves: a) constructing a bounding tesseract or simplex that contains all points, b) identifying the elements which contain the point to be inserted in their circumhyperspheres, c) forming a cavity by removing these elements from the mesh, and d) connecting the facets of the cavity to the point to be inserted. The last step involves creating a new set of pentatope elements. In this section, we provide a detailed discussion of each of these steps in the point insertion process. It is important to note that the process of forming a cavity of pentatopes containing the point to be inserted in their respective circumhyperspheres requires an infinite level of numerical precision to function as intended, which is obviously impossible.  To evade this common issue, one may modify the traditional Bowyer-Watson algorithm as will be discussed shortly.

\subsection{Bounding Tesseract} \label{bounding_tesseract}
One requires the manual specification of a single pentatope, or multiple pentatopes, that will contain all points which are to be inserted. If a single pentatope is used for this purpose, then this pentatope is called a \emph{super pentatope} or \emph{bounding pentatope}. Although it is possible to bound all points by a single pentatope, we chose an alternative scheme since point insertion in a single pentatope will inevitably create new pentatopes with relatively small dihedral angles. In addition, it is not straightforward to construct a single pentatope which is guaranteed to contain an arbitrary, finite set of points. With this in mind, we create multiple pentatopes for bounding purposes by constructing a single bounding (super) tesseract, and thereafter, subdividing the tesseract into $N_b$ pentatopes.  Our subdivision of the tesseract is required to satisfy the following conditions:
\begin{enumerate}[label=(\alph*)]
    \item The subdivision forms a Delaunay triangulation of the tesseract.
    \item The subdivision only uses points which reside on the boundary of the tesseract. 
\end{enumerate}
%
%
It turns out that several subdivisions satisfy the given requirements. In particular, the Coxeter-Freudenthal-Kuhn subdivision~\cite{coxeter1934discrete,freudenthal1942simplizialzerlegungen,kuhn1960some} which splits the tesseract into $N_b =24$ pentatopes is a suitable choice. This subdivision uses the original 16 vertices of the tesseract, and all the pentatopes of the subdivision belong to a single equivalence class. The reference coordinates used for the tesseract decomposition are shown in Table~\ref{tesseract_coord_table}, and the pentatope indices are summarized in Table~\ref{tesseract_tab24}.

Somewhat surprisingly, we have  discovered two new subdivisions with $N_b = 22$ and 23 pentatopes, respectively. These new subdivisions use the original 16 vertices of the bounding tesseract. However, each subdivision is less uniform than the Coxeter-Freudenthal-Kuhn subdivision, as the associated subpentatopes do not have identical hypervolumes.  Figures~\ref{bounding_tesseract_1}--\ref{bounding_tesseract_3} depict each of the pentatopes for the $N_b = 23$ case. The indices for the $N_b = 22$ and 23 cases are summarized in Tables~\ref{tesseract_tab22} and \ref{tesseract_tab23}.

It may be possible to identify subdivisions with $N_b < 22$. However, we were unable to identify such subdivisions via our numerical experiments. Nevertheless, the discovery of subdivisions with $N_b < 24 = d!$ is already an interesting achievement. These subdivisions are directly analogous to the $N_b = 5$ subdivision of the 3-cube into tetrahedra. We note that the existence of subdivisions with $N_b < d!$ is actually contrary to some claims in the literature (see~\cite{shores2010bounds}), where it is mistakenly argued that the lower bound for $N_b$ is $d!$. 


\begin{table}[h!]
\centering
\begin{tabular}{|c|c|c|c|c|}
\hline
& $x$ $y$ $z$ $t$ & & $x$ $y$ $z$ $t$ \\
\hline
$p_1$ & (0, 0, 0, 0) & $p_2$ & (1, 0, 0, 0) \\
$p_3$ & (1, 1, 0, 0) & $p_4$ & (0, 1, 0, 0) \\
$p_5$ & (0, 0, 1, 0) & $p_6$ & (1, 0, 1, 0) \\
$p_7$ & (1, 1, 1, 0) & $p_8$ & (0, 1, 1, 0) \\
$p_9$ & (0, 0, 0, 1) & $p_{10}$ & (1, 0, 0, 1) \\
$p_{11}$ & (1, 1, 0, 1) & $p_{12}$ & (0, 1, 0, 1) \\
$p_{13}$ & (0, 0, 1, 1) & $p_{14}$ & (1, 0, 1, 1) \\
$p_{15}$ & (1, 1, 1, 1) & $p_{16}$ & (0, 1, 1, 1) \\
\hline
\end{tabular}
\caption{Ordering of coordinates for tesseract decompositions.}
\label{tesseract_coord_table}
\end{table}

\begin{table}[h!]
\centering
\begin{tabular}{|l|l|l|l|}
\hline
Pentatopes 1-6 & Pentatopes 7-12 & Pentatopes 13-18 & Pentatopes 19-24 \\
\hline
1       2       3    5    9 &
2       3       4    5    9 &
2       4       5    6    9 &
3       4       5    7    9 \\
4       5       6    7    9 &
4       6       7    8    9 &
2       4       6    9    10 &
4       6       8    9    10 \\
3       4       7    9    11 &
4       7       8    9    11 &
4       8       9    10    11 &
4       8       10    11    12 \\
5       6       7    9    13 &
6       7       8    9    13 &
6       8       9    10    13 &
7       8       9    11    13 \\
8       9       10    11    13 &
8       10      11    12    13 &
6       8       10    13    14 &
8       10      12    13    14 \\
7       8       11    13    15 &
8       11      12    13    15 &
8       12      13    14    15 &
8       12      14    15    16 \\
\hline
\end{tabular}
\caption{Indices for the subdivision of the tesseract into $N_b = 24$ pentatopes.}
\label{tesseract_tab24}
\end{table}
\begin{table}[h!]
\centering
\begin{tabular}{|l|l|l|l|}
\hline
Pentatopes 1-6 & Pentatopes 7-12 & Pentatopes 13-18 & Pentatopes 19-22 \\
\hline
1  2  4  5 13 &
1  2  4  9 13 &
2 11 13 14 15 &
2  7 11 13 15 \\
2  4  5  7 13 &
2  4  9 11 13 &
4  7 11 13 15 &
4  7 13 15 16 \\
2  3  7 11 13 &
4 11 13 15 16 &
2  7 13 14 15 &
2  3  4  7 13 \\
2  3  4 11 13 &
3  4  7 11 13 &
4  5  7  8 13 &
4  7  8 13 16 \\
4  9 11 12 13 &
4 11 12 13 16 &
2  5  6  7 13 &
-- \\
2  6  7 13 14 &
2  9 10 11 13 &
2 10 11 13 14 & 
-- \\
\hline
\end{tabular}
\caption{Indices for the subdivision of the tesseract into $N_b = 22$ pentatopes.}
\label{tesseract_tab22}
\end{table}
\begin{table}[h!]
\centering
\begin{tabular}{|l|l|l|l|}
\hline
Pentatopes 1-6 & Pentatopes 7-12 & Pentatopes 13-18 & Pentatopes 19-23 \\
\hline
1 3 4 6 9 &
1 6 8 9 13 & 
3 7 8 13 15 &
1 4 6 8 9 \\
3 6 7 8 13 &
3 6 7 13 15 &
3 4 6 8 9 &
3 6 8 9 13 \\
3 10 11 13 15 &
1 2 3 6 10 &
3 8 9 11 13 &
3 6 10 13 15 \\
1 3 6 9 10 &
3 6 9 10 13 &
8 11 12 13 15 &
3 4 8 9 11 \\
3 9 10 11 13 &
6 10 13 14 15 &
4 8 9 11 12 &
8 9 11 12 13 \\
8 12 13 15 16 &
1 5 6 8 13 &
3 8 11 13 15 & 
-- \\
\hline
\end{tabular}
\caption{Indices for the subdivision of the tesseract into $N_b = 23$ pentatopes.}
\label{tesseract_tab23}
\end{table}

\begin{figure}[h!]
\centering
\includegraphics[width=6cm,trim={1cm 0 1cm 0},clip]{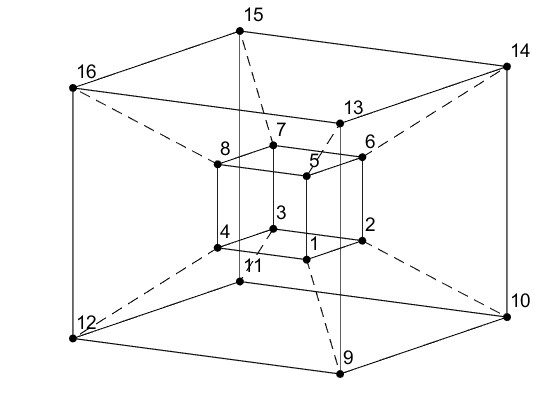} 
\includegraphics[width=6cm,trim={1cm 0 1cm 0},clip]{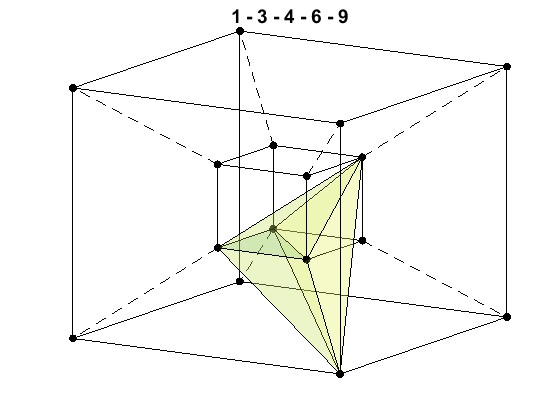}
\includegraphics[width=6cm,trim={1cm 0 1cm 0},clip]{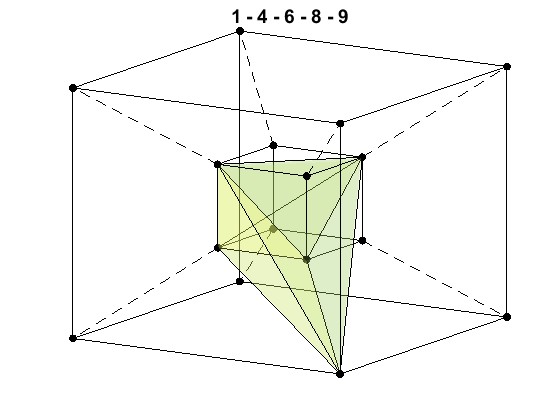} 
\includegraphics[width=6cm,trim={1cm 0 1cm 0},clip]{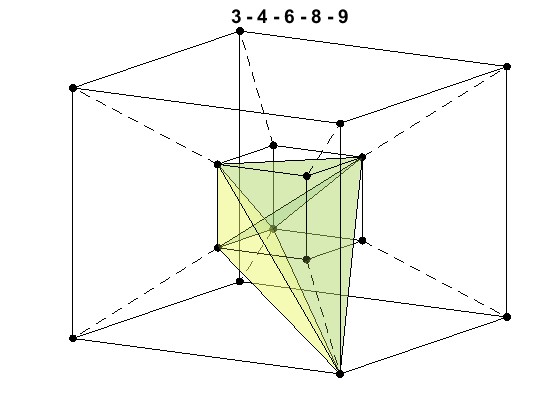}
\includegraphics[width=6cm,trim={1cm 0 1cm 0},clip]{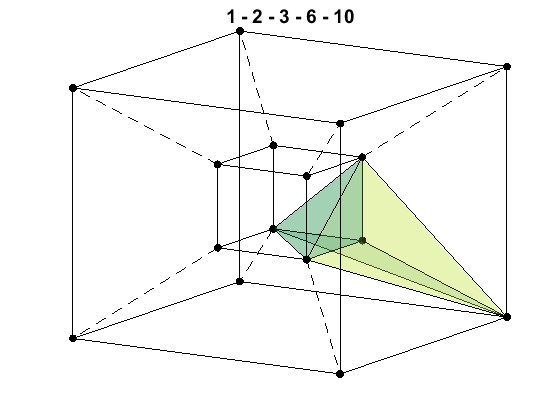} 
\includegraphics[width=6cm,trim={1cm 0 1cm 0},clip]{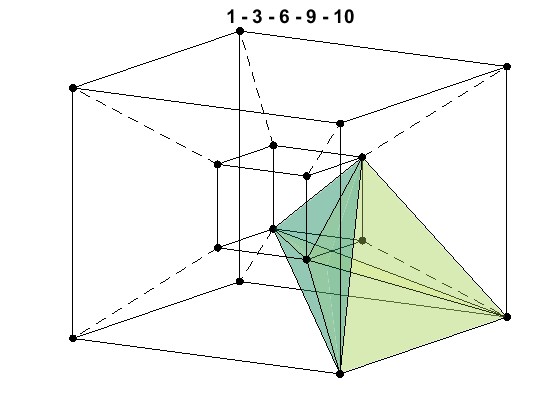}
\includegraphics[width=6cm,trim={1cm 0 1cm 0},clip]{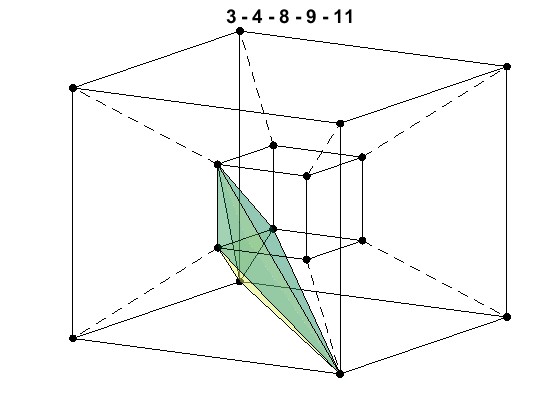}
\includegraphics[width=6cm,trim={1cm 0 1cm 0},clip]{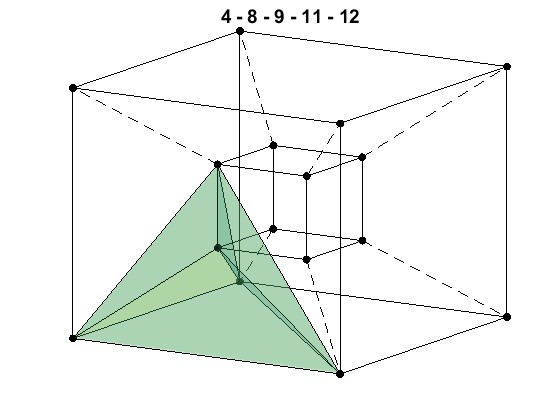} 
\caption{Illustration of the bounding tesseract sectioned into pentatopes (1-7), for the subdivision strategy with $N_b = 23$.}
\label{bounding_tesseract_1}
\end{figure}
\begin{figure}[h!]
\centering
\includegraphics[width=6cm,trim={1cm 0 1cm 0},clip]{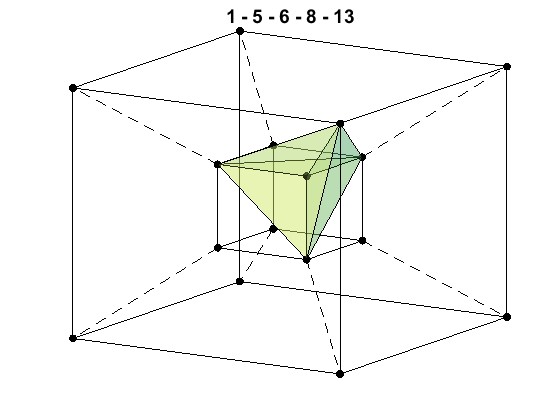}
\includegraphics[width=6cm,trim={1cm 0 1cm 0},clip]{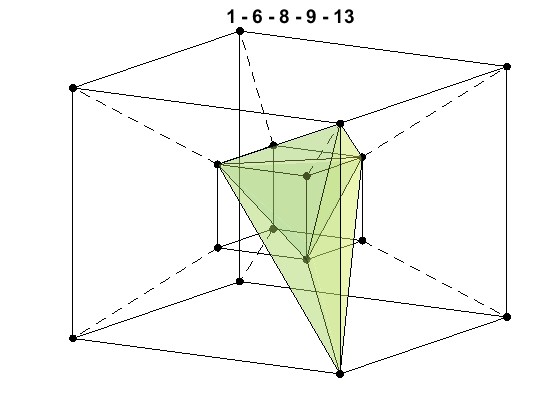} 
\includegraphics[width=6cm,trim={1cm 0 1cm 0},clip]{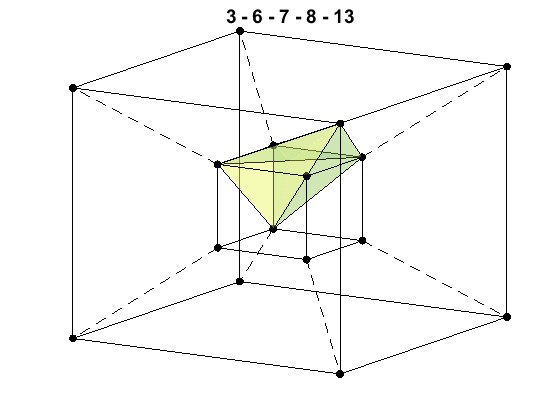}
\includegraphics[width=6cm,trim={1cm 0 1cm 0},clip]{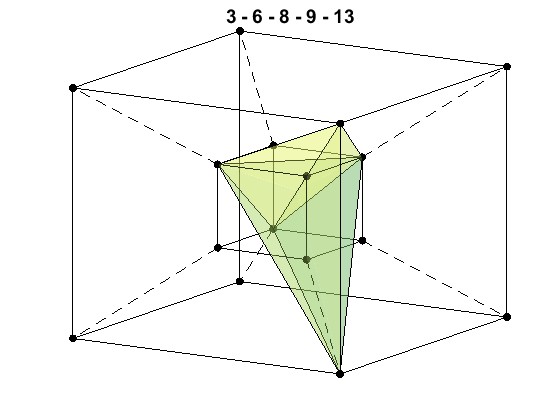}
\includegraphics[width=6cm,trim={1cm 0 1cm 0},clip]{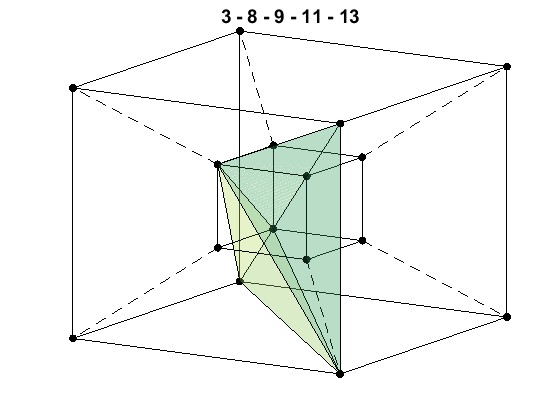}
\includegraphics[width=6cm,trim={1cm 0 1cm 0},clip]{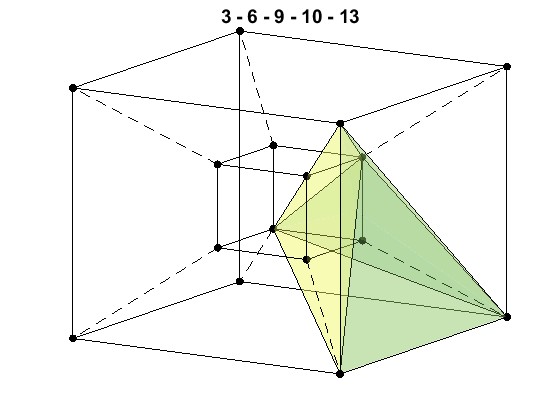} 
\includegraphics[width=6cm,trim={1cm 0 1cm 0},clip]{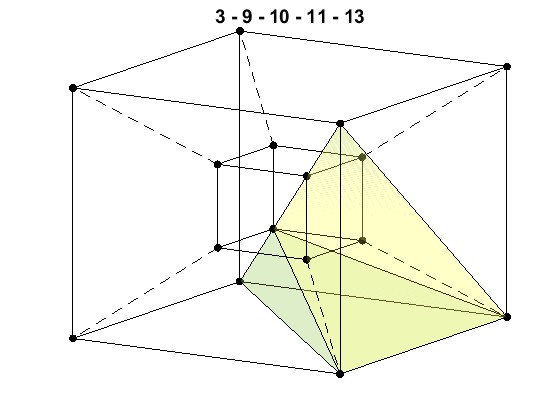}
\includegraphics[width=6cm,trim={1cm 0 1cm 0},clip]{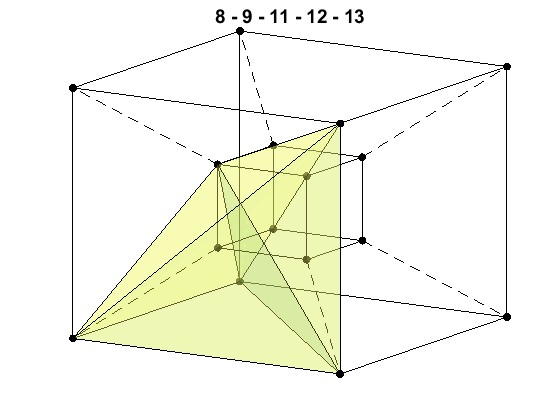} 
\caption{Illustration of the bounding tesseract sectioned into pentatopes (8-15), for the subdivision strategy with $N_b = 23$.}
\label{bounding_tesseract_2}
\end{figure}
\begin{figure}[h!]
\centering
\includegraphics[width=6cm,trim={1cm 0 1cm 0},clip]{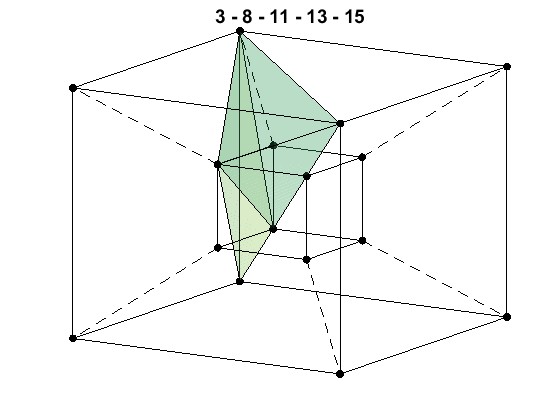}
\includegraphics[width=6cm,trim={1cm 0 1cm 0},clip]{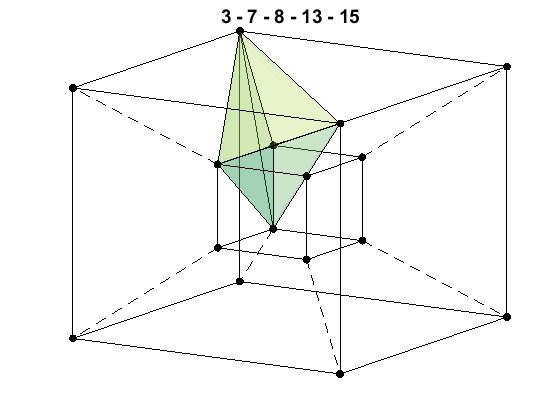}
\includegraphics[width=6cm,trim={1cm 0 1cm 0},clip]{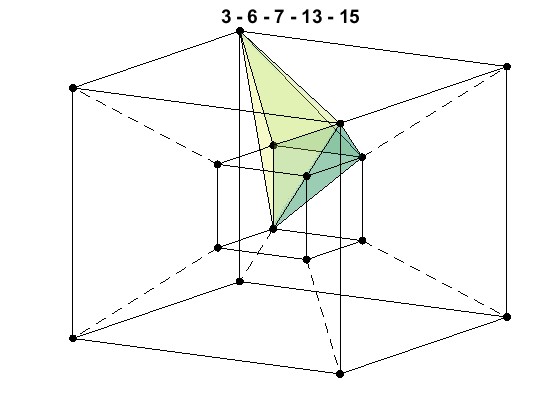}
\includegraphics[width=6cm,trim={1cm 0 1cm 0},clip]{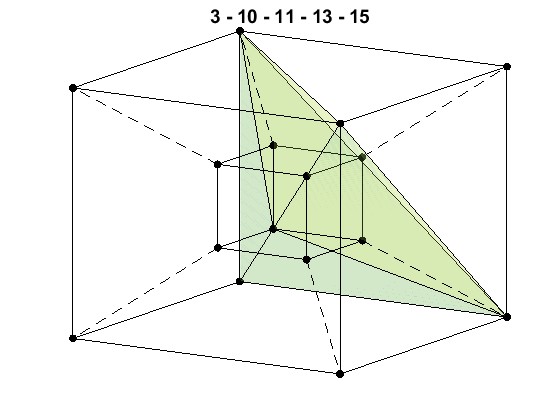} 
\includegraphics[width=6cm,trim={1cm 0 1cm 0},clip]{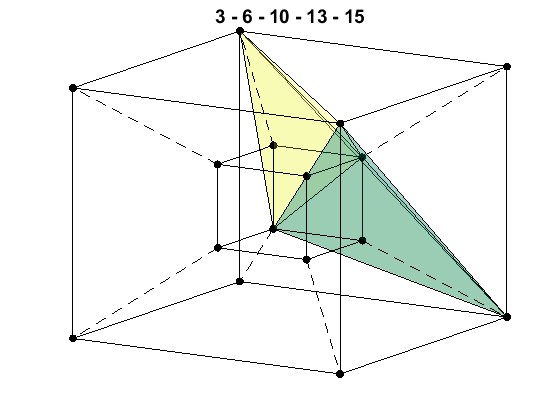}
\includegraphics[width=6cm,trim={1cm 0 1cm 0},clip]{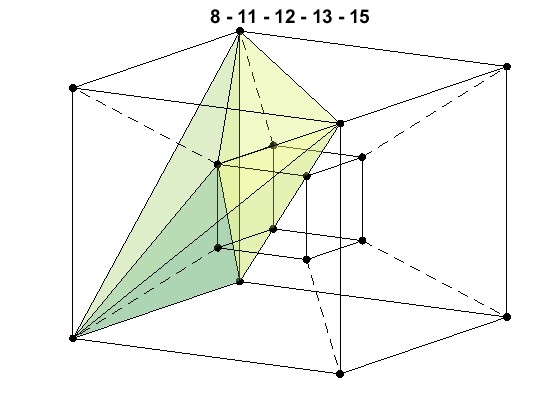} 
\includegraphics[width=6cm,trim={1cm 0 1cm 0},clip]{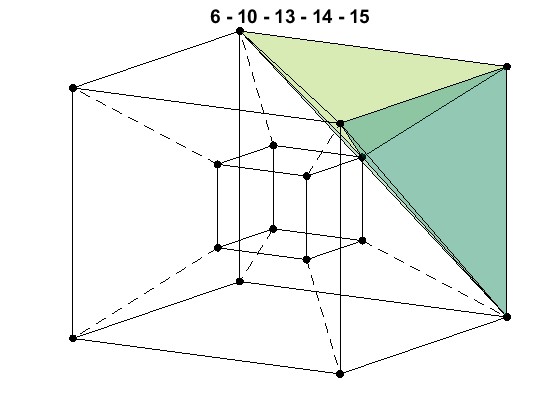}
\includegraphics[width=6cm,trim={1cm 0 1cm 0},clip]{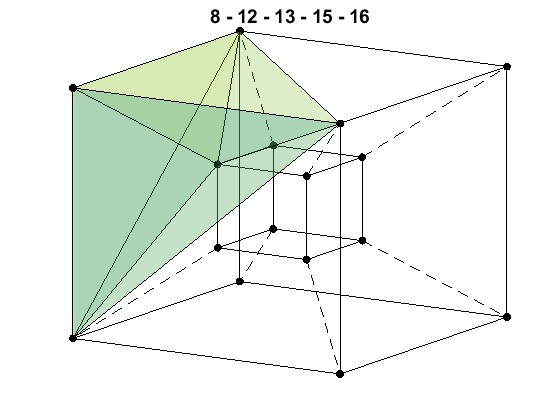}
\caption{Illustration of the bounding tesseract sectioned into pentatopes (16-23), for the subdivision strategy with $N_b = 23$.}
\label{bounding_tesseract_3}
\end{figure}

\pagebreak
\clearpage

\subsection{Base Element Search} \label{base_element}

We begin by finding a pentatope element, referred to as the \emph{base element}, that contains the point $p$ to be inserted. There are two prominent methods for identifying the base element: i) Lo's method~\cite{lo2014finite} which relies on calculating a series of volumes involving $p$, and ii) Si's method~\cite{hang2015tetgen} which relies on calculating a series of orientations of $p$ relative to the tetrahedron's faces. Although both methods are effective, we proceeded with Si's orientation-based method because it may be supplemented with the standard (non-metric-weighted) 4D orientation predicate for further accuracy. Scaled up to 4D, the method starts by identifying the tetrahedral facets associated with a given pentatope, then forming five temporary pentatopes from each facet combined with $p$. If each of the five temporary pentatopes \emph{n} has an orientation $\theta$ greater than or equal to zero, or each has an orientation $\theta$ less than or equal to zero, $p$ is located inside of the pentatope, making it the base element. Pseudocode for this condition is shown below.
\vspace{0.1in}
\begin{algorithmic}
    \If{$\theta_n \geq$ 0 $\forall n$}\\
        \quad inside = true
    \ElsIf{$\theta_n \leq$ 0 $\forall n$}\\
        \quad inside = true
    \Else\\
        \quad inside = false
    \EndIf
\end{algorithmic}
\vspace{0.1in}
If $p$ does not lie within the pentatope, the orientations for each of the temporary pentatopes help indicate the direction of the point relative to the current pentatope. The facet associated with the temporary pentatope with the smallest magnitude for its orientation out of the five is the facet closest to $p$. The neighboring pentatope sharing this tetrahedral facet is then selected as the next candidate for the base element, and the process is repeated until the base element is found, forming a path from the first pentatope tested. A simple 2D illustration of a search path to the base element is shown in Figure~\ref{cavity_operator}.
\begin{figure}[h!]
\centering
\includegraphics[width=6cm]{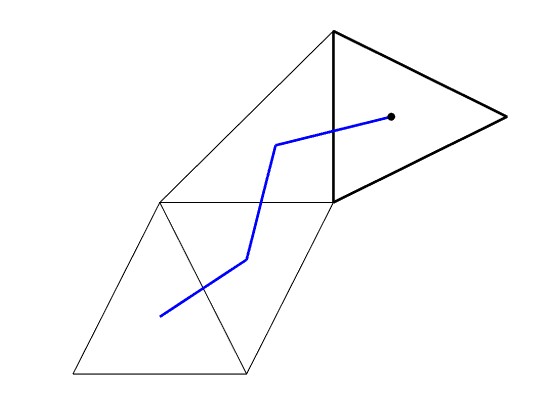}
\includegraphics[width=3cm]{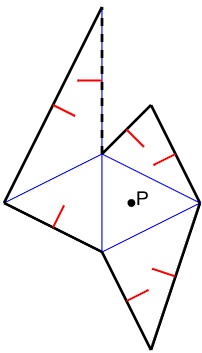}
\includegraphics[width=4cm]{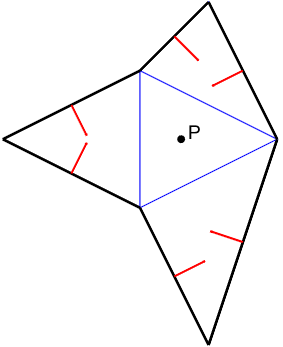}
\includegraphics[width=4cm]{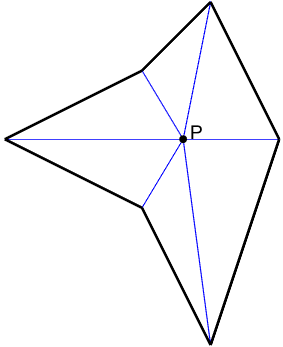}
\caption{2D Illustration of point insertion cavity operator. The top left illustrates finding the base element. The top right shows the initial cavity and the normals of its boundary edges. The dashed line indicates a boundary edge that is invisible to point $p$.  The bottom left shows the reshaped cavity with all of its boundary edges having visibility to point $p$. Finally, the bottom right shows the reconnection of the cavity boundary to point $p$.}
\label{cavity_operator}
\end{figure}



One of the most common problems encountered during a search for the base element is a search path forming a continuous closed loop within a specific set of elements. This occurs naturally due to limitations of the methodology. Fortunately, this problem can be easily remedied by keeping track of elements that have already been searched and not allowing the algorithm to return to them. Doing this will force a non-repeating path to be formed; although, the path may not be as direct as one would draw manually. If for some reason the number of iterations in searching for the base element exceeds the number of elements currently in the mesh without finding the base element, it may be useful to switch to an alternate strategy by performing a global search through all of the elements.  Resorting to a global search in cases where an unresolved search path is formed should be a rare occurrence, and its frequency may indicate issues with the implementation of the element search, for example, issues with the implementation of the geometric predicates or the level of precision used.

Another common issue worth mentioning is the appearance of \emph{ghost} points, meaning points to be inserted that are found to not be contained in any existing pentatope. For an arrangement of coplanar (hypercoplanar) points, points to be inserted will often lie on a tetrahedral facet shared by two pentatopes. Due to limitations in machine precision, points lying directly on tetrahedral facets may be labeled as not belonging to either pentatope sharing that tetrahedral facet.  This issue also extends to points located on other common features such as shared triangular faces and edges. One can avoid this issue by carefully setting tolerances associated with the resulting values of the orientation predicates for the base element test instead of comparing directly to zero.

\subsection{Cavity Operator} \label{cavity_formation}

Once the base pentatope has been identified, a cavity is formed by identifying the neighboring pentatopes with circumhyperspheres that contain point $p$, in accordance with the Delaunay criterion. We note that these circumhyperspheres are defined in metric space, where the metric $M = M_p = M(x_p, y_p, z_p, t_p)$ is associated with the point $p$ that is being inserted. Henceforth, we will refer to such circumhyperspheres as $M$-circumhyperspheres. The facet neighbors of the base pentatope make up an initial \emph{neighbor front}. Each pentatope in the neighbor front is tested to see if its $M$-circumhypersphere contains point $p$. The check for the presence of a point within a $M$-circumhypersphere is called an $inhypersphere_M$ check. If the point is within a pentatope's $M$-circumhypersphere, it is added to the cavity and its neighbors are added to the neighbor front if they do not yet belong to the cavity, and are not already a part of the neighbor front. Once the neighbor front is empty, meaning none of the pentatopes at the outside boundary of the cavity contain point $p$ in their respective $M$-circumhyperspheres, the process ends, and the cavity is considered complete.

Next, it is necessary to evaluate the validity of the cavity. For our purposes, a cavity is considered valid if point $p$ is visible to all the cavity's bounding tetrahedral facets. Facet visibility is determined by calculating the metric-induced inner product (denoted by $Q$) between two vectors -- the normal, inward facing vector of the facet ($N$) and the vector formed between the facet's centroid $c$ and point $p$ ($CP$),
\begin{align*}
    Q = (N)^{T} M (CP).
\end{align*}
In order to control visibility, one must set a lower limit on the value of $Q$. The limit set on the value $Q$ is heavily dependent on the user and their desired application, as it essentially represents the minimum quality for elements that the user is willing to accept. $Q = 0$ means that when a facet is reconnected to $p$, the pentatope formed will have no hypervolume, and positive values close to zero will form slivers. Although the user may decide, we will note that four dimensions introduces greater complexity, and the point insertion algorithm begins to lose stability the further the tolerance is set above zero. For our purposes in creating an initial hypervolume mesh from a given hypersurface mesh, the limit of $Q$ is set to double precision $(10^{-16})$, and all slivers with hypervolumes within double precision are accepted. The accuracy of $Q$ is very important, which is why it is recommended that this value is calculated using (at least) quadruple precision.

If $Q$ is negative or less than a tolerance specified by the user, the facet is considered not visible to point $p$. The pentatope associated with this facet is then removed from the cavity. This process continues until all the cavity's bounding facets have visibility to point $p$.  These steps are illustrated in Figure~\ref{cavity_operator}. There is no scenario where all the pentatopes associated with the cavity are removed. If this occurs, it is an indication that the proposed algorithm has not been properly implemented, and one should check for consistent element orientation and proper removal of pentatopes associated with \emph{invisible} facets.

\subsection{Reconnection and Review} \label{reconnection}

After an initial cavity has been identified and all its bounding tetrahedral facets have been deemed visible relative to point $p$, the pentatopes belonging to the cavity are removed from the mesh and replaced with pentatopes formed by connecting the bounding tetrahedral facets of the cavity to the inserted point. The final reconnection step is illustrated in Figure~\ref{cavity_operator}. To ensure consistent orientation, each new pentatope's local indexing is carefully specified to have a positive orientation. The entire point insertion process continues until all points from the hypersurface mesh have been inserted into the bounding tesseract.  

A complete summary of the point insertion process appears below:
\begin{enumerate}
    \item Locate the base pentatope.
    \item Evaluate the neighbor front until an initial cavity is defined.
    \item Remove pentatopes from the cavity until all the bounding tetrahedral facets are visible to the point to be inserted.
    \item Remove pentatopes associated with the cavity. Then connect the cavity's tetrahedral facets to the newly inserted point to form new pentatopes.
    \item Ensure positive orientations of all new pentatopes.
\end{enumerate}

A common error one may encounter when first attempting to implement this process is finding that points are still missing from the tessellation after the insertion algorithm has supposedly added all of the points. This stems from a lack of accuracy when identifying pentatopes belonging to each cavity -- i.e. pentatopes that contain the point $p$ within their respective $M$-circumhyperspheres (see section~\ref{cavity_formation}). If one is already using a 4D geometric predicate for this process, it may be insufficiently accurate, and changes may need to be implemented in accordance with section~\ref{sec;metric_predicates}. 

\section{Geometric Predicates}
\label{sec;metric_predicates}

The two tests that appear throughout the point insertion algorithm, 4D $orientation$ (or $orientation_{M}$) and $inhypersphere$ (or $inhypersphere_M$), hinge on the precision of predicates. Shewchuk pioneered geometric predicates in 2D and 3D~\cite{shewchuk1996adaptive} for isotropic Delaunay meshes in Euclidean space. In this work, we construct predicates in 4D for anisotropic Delaunay meshes. In particular, we explicitly incorporate the influence of a generic metric field $M = M(x,y,z,t)$ into the well-known orientation and inhypersphere predicates. To the authors' knowledge, our proposed approach is distinct from those of previous researchers. In order to illustrate this fact, let us review the standard approach for computing geometric predicates for anisotropic Delaunay meshes.

\subsection{The Standard Approach} \label{stand_app}

The standard approach consists of using the square root or the Cholesky decomposition of the metric $M$ to scale each point, (see~\cite{boissonnat2015anisotropic}, section 2.1 for details); and thereafter using these scaled points as inputs to ordinary Delaunay predicates. If we assume that $M$ is symmetric and positive definite, it is always possible to form a unique, symmetric, positive definite, square root  $M^{1/2} = (M^{1/2})^{T}$ where $M = M^{1/2} M^{1/2}$. In addition, it is also possible to form the Cholesky decomposition $M = C_{M}^{T} C_{M}$. In both cases, we ensure that the following important identities are satisfied: 
\begin{align*}
M = G_{M}^{T} G_{M} \quad \mathrm{and} \quad \mathrm{det}(G_{M})>0,
\end{align*}
where $G_{M}$ is a generic matrix factor in our decomposition.
This means we can choose $G_{M} = M^{1/2}$ or $G_{M} = C_{M}$. We can then scale our input points by the matrix $G_{M}$. In particular, suppose that we want to compute the orientation of five points: $a = (a_x, a_y, a_z, a_t)^{T}$, $b = (b_x, b_y, b_z, b_t)^{T}$, $c = (c_x, c_y, c_z, c_t)^{T}$, $d = (d_x, d_y, d_z, d_t)^{T}$, and $e = (e_x, e_y, e_z, e_t)^{T}$. The orientation predicate is given by
\begin{align}
    Orientation 
    = \begin{vmatrix}
        (a-e)^T \\
        (b-e)^T \\
        (c-e)^T \\
        (d-e)^T
    \end{vmatrix}. \label{baseline_orient}
\end{align}
%
Now, in order to incorporate the influence of the metric $M$, we simply perform the following substitution in the predicate above
\begin{align}
    a \Rightarrow G_{M} a, \quad b \Rightarrow G_{M} b, \quad c \Rightarrow G_{M} c, \quad d \Rightarrow G_{M} d, \quad e \Rightarrow G_{M} e. \label{transform_one}
\end{align}
The same substitution can be performed in the following in-hypersphere predicate
\begin{align}
    \nonumber InHypersphere &=
    \begin{vmatrix}
        (a-f)^{T} & (a-f)^{T} (a-f) \\
        (b-f)^{T} & (b-f)^{T} (b-f) \\
        (c-f)^{T} & (c-f)^{T} (c-f) \\
        (d-f)^{T} & (d-f)^{T} (d-f) \\
        (e-f)^{T} & (e-f)^{T} (e-f)  
    \end{vmatrix} 
    \\[1.0ex]
    \nonumber &=
    (a-f)^{T}(a-f)\begin{vmatrix}
        (b-f)^{T} \\
        (c-f)^{T} \\
        (d-f)^{T} \\
        (e-f)^{T}
    \end{vmatrix} 
    -(b-f)^{T}(b-f)\begin{vmatrix}
        (a-f)^{T} \\
        (c-f)^{T} \\
        (d-f)^{T} \\
        (e-f)^{T}
    \end{vmatrix} \\[1.0ex]
    \nonumber &+ (c-f)^{T}(c-f)\begin{vmatrix}
        (a-f)^{T} \\
        (b-f)^{T} \\
        (d-f)^{T} \\
        (e-f)^{T}
    \end{vmatrix} 
    -(d-f)^{T}(d-f)\begin{vmatrix}
        (a-f)^{T} \\
        (b-f)^{T} \\
        (c-f)^{T} \\
        (e-f)^{T}
    \end{vmatrix}
    \\[1.0ex]
    &+ (e-f)^{T}(e-f)\begin{vmatrix}
        (a-f)^{T} \\
        (b-f)^{T} \\
        (c-f)^{T} \\
        (d-f)^{T}
    \end{vmatrix}, \label{baseline_incircle}
\end{align}
where $a,b,c,d,e$ are defined above, $f=(f_x,f_y,f_z,f_t)^{T}$, and
\begin{align}
    f \Rightarrow G_{M} f. \label{transform_two}
\end{align}
A problem with this approach, is that the substitutions in Eqs.~\eqref{transform_one} and \eqref{transform_two} can introduce additional error into our calculations of the predicates. Evidently, there are two additional sources of error which appear, relative to a conventional predicate calculation:
\begin{itemize}
    \item \emph{Decomposition error}. Error that arises when we compute the decomposition of the metric tensor, $M = G_{M}^{T} G_{M}$.
    \item \emph{Multiplication error}. Error that arises when we multiply each point by the decomposition matrix factor, $G_{M}$.
\end{itemize}
It would appear that the first item is the most important, as it implies that we must compute the decomposition $M =G_{M}^{T} G_{M}$ with a high-degree of precision. 
In particular, we note that the conventional predicates---without the metric---depend on high-precision matrix determinants, which are well-understood, (see \cite{shewchuk1996adaptive}). However, the results of these high-precision predicates may be polluted if their inputs, i.e. the metric scaled points, are not computed with sufficient precision. Therefore, in cases involving the metric $M$, we believe that one should construct a high-precision matrix decomposition routine, or identify the necessary routine in a linear algebra library, in order to ensure sufficient precision of the predicate inputs. 
%
%

Alternatively, in what follows, we propose a more natural approach which avoids the matrix decomposition of $M$ entirely.

\subsection{An Alternative Approach} \label{alt_app}

Upon leveraging basic principles from linear algebra, one can show that the metric-weighted orientation and inhypersphere predicates can be computed as follows for the case of $d = 4$
\begin{align}
    Orientation_{M} 
    = \sqrt{\mathrm{det}(M)}\begin{vmatrix}
        (a-e)^T \\
        (b-e)^T \\
        (c-e)^T \\
        (d-e)^T
    \end{vmatrix}, \label{new_orient}
\end{align}

\begin{align}
    \nonumber InHypersphere_{M} = \sqrt{\mathrm{det}(M)}&\Bigg(
    (a-f)^{T}M(a-f)\begin{vmatrix}
        (b-f)^{T} \\
        (c-f)^{T} \\
        (d-f)^{T} \\
        (e-f)^{T}
    \end{vmatrix} 
    -(b-f)^{T}M(b-f)\begin{vmatrix}
        (a-f)^{T} \\
        (c-f)^{T} \\
        (d-f)^{T} \\
        (e-f)^{T}
    \end{vmatrix} \\[1.0ex]
    \nonumber &+ (c-f)^{T}M(c-f)\begin{vmatrix}
        (a-f)^{T} \\
        (b-f)^{T} \\
        (d-f)^{T} \\
        (e-f)^{T}
    \end{vmatrix} 
    -(d-f)^{T}M(d-f)\begin{vmatrix}
        (a-f)^{T} \\
        (b-f)^{T} \\
        (c-f)^{T} \\
        (e-f)^{T}
    \end{vmatrix}
    \\[1.0ex]
    &+ (e-f)^{T}M(e-f)\begin{vmatrix}
        (a-f)^{T} \\
        (b-f)^{T} \\
        (c-f)^{T} \\
        (d-f)^{T}
    \end{vmatrix} \Bigg). \label{new_incircle}
\end{align}
This formulation of the predicates is mathematically valid, as it yields the same results as the predicates in Eqs.~\eqref{baseline_orient} and \eqref{baseline_incircle} in conjunction with the substitutions from Eqs.~\eqref{transform_one} and \eqref{transform_two}. Furthermore, as mentioned previously, this formulation also avoids any error associated with computing the matrix factor $G_{M}$. In addition, this formulation only requires well-known predicates for computing $4 \times 4$ determinants. For example, the $inhypersphere_{M}$ predicate involves the computation of six $4 \times 4$ determinants. These computations can be performed with a high-degree of precision, on an as-needed basis. 

\subsection{Generalization to $d$-dimensions}

The predicate formulation in Eqs.~\eqref{new_orient} and \eqref{new_incircle} can be generalized to $d$-dimensions, where $d \geq 4$, as follows
\begin{align}
    Orientation_{M} 
    = \sqrt{\mathrm{det}(M)}\begin{vmatrix}
        (p_1-p_{d+1})^T \\
        \vdots \\
        (p_d-p_{d+1})^T
    \end{vmatrix}, \label{new_orient_dimension}
\end{align}
\begin{align}
    \nonumber InHypersphere_{M} &= (-1)^{d} \sqrt{\mathrm{det}(M)} \Bigg( (p_{1}-p_{d+2})^{T}M(p_{1}-p_{d+2})\begin{vmatrix}
        (p_{2} - p_{d+2})^{T} \\
        \vdots \\
        (p_{d+1} - p_{d+2})^T
    \end{vmatrix} 
    \\[1.0ex]
   \nonumber &+\sum_{i=2}^{d} (-1)^{i-1}
    (p_{i}-p_{d+2})^{T}M(p_{i}-p_{d+2})\begin{vmatrix}
        (p_1 - p_{d+2})^{T} \\
        \vdots \\
        (p_{i-1} - p_{d+2})^{T} \\
        (p_{i+1} - p_{d+2})^{T} \\
        \vdots \\
        (p_{d+1} - p_{d+2})^T
    \end{vmatrix} \\[1.0ex]
    &+(-1)^{d} (p_{d+1}-p_{d+2})^{T}M(p_{d+1}-p_{d+2})\begin{vmatrix}
        (p_{1} - p_{d+2})^{T} \\
        \vdots \\
        (p_{d} - p_{d+2})^T
    \end{vmatrix} 
    \Bigg),
\end{align}
where $M \in \mathbb{R}^{d\times d}$.

\subsection{Comparison of Standard and Alternative Approaches}

In this section, our objective is to demonstrate the differences between the standard approach (section~\ref{stand_app}), and the alternative approach (section~\ref{alt_app}), for computing the $inhypersphere_{M}$ predicate. Towards this end, we performed numerical tests on randomly-generated, symmetric positive-definite matrices $M$, where
\begin{align*}
    M = S^{T} S,
\end{align*}
and $S \in \mathbb{R}^{d\times d}$ is a randomly-generated matrix with entries on the interval $[0,10]$. In addition, we computed point locations $p_{1}, \ldots, p_{d+2}$, with randomly-generated coordinates $(x_1,\ldots,x_d)$ on the interval $[0,1]$. The $inhypersphere_{M}$ predicate was computed using double precision in Matlab, with the randomly generated inputs of $M$ and $p_{1}, \ldots, p_{d+2}$. The differences between the two approaches were quantified by computing the following normalized difference between the outputs
\begin{align*}
    \text{Normalized difference} = \frac{\left|(InHypersphere_{M})_{\text{standard}} - (InHypersphere_{M})_{\text{alternate}} \right|}{\left|(InHypersphere_{M})_{\text{standard}} \right|}.
\end{align*}
The normalized difference between the outputs was averaged over the course of $100$ randomly-generated metrics and point sets. These metrics and point sets were used twice: in particular, we performed tests on the standard approach in conjunction with the Cholesky decomposition, and thereafter, in conjunction with the square-root decomposition. During each test, we computed the decomposition error, as follows
\begin{align*}
    \text{Decomposition error} = \left\|M - G_{M}^{T} G_{M} \right\|_{F},
\end{align*}
where $\left\|\cdot \right\|_{F}$ is the Frobenius norm. This error was averaged over the course of the $100$ randomly generated metrics (as discussed above). 

Figure~\ref{cholesky_difference_fig} shows the average normalized difference between the standard and alternative approaches, for cases in which the Cholesky decomposition was used for the standard approach. In addition, Figure~\ref{cholesky_error_fig} shows the average error in the Cholesky decomposition. Figures~\ref{sqrt_difference_fig} and~\ref{sqrt_error_fig} show similar trends, for the standard approach in conjunction with the square-root decomposition. In each case, results for dimensions $d = 2, 3, 4, 5, 10, 20, 30, 40, 50, 60, 70, 80, 90$, and $100$ are shown. It is clear that the errors in the Cholesky decomposition and square-root decomposition grow with the number of dimensions $d$. In addition, the differences between the standard and alternative approaches have a tendency to grow, as the errors in the decompositions grow. This provides evidence for the idea that errors in the decompositions, especially as the number of dimensions increases, effect the reliability of the associated predicates. At the very least, these errors are correlated with differences in observed outputs.
\begin{figure}[h!]
\centering
\includegraphics[width=9cm]{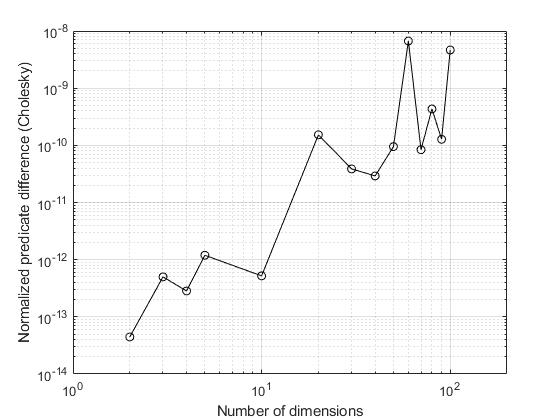}
\caption{Normalized difference between the output of the $inhypersphere_{M}$ predicate computed using the standard approach in conjunction with the Cholesky decomposition, and the alternative approach.}
\label{cholesky_difference_fig}
\end{figure}

\begin{figure}[h!]
\centering
\includegraphics[width=9cm]{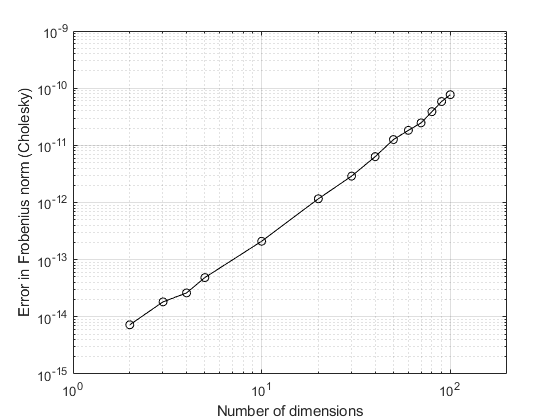}
\caption{Frobenius norm of the error in the Cholesky decomposition.}
\label{cholesky_error_fig}
\end{figure}

\begin{figure}[h!]
\centering
\includegraphics[width=9cm]{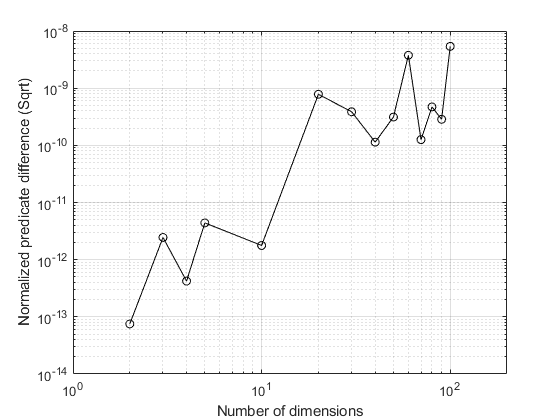}
\caption{Normalized difference between the output of the $inhypersphere_{M}$ predicate computed using the standard approach in conjunction with the square-root decomposition, and the alternative approach.}
\label{sqrt_difference_fig}
\end{figure}

\begin{figure}[h!]
\centering
\includegraphics[width=9cm]{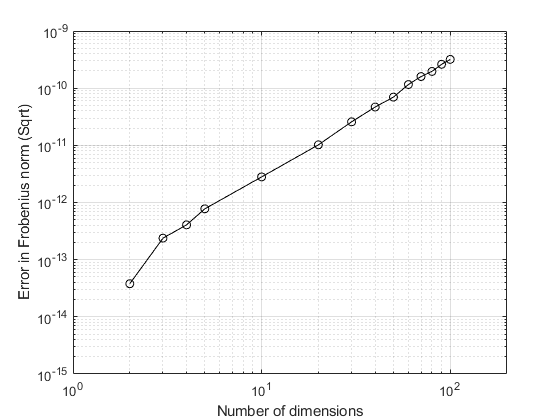}
\caption{Frobenius norm of the error in the square-root decomposition.}
\label{sqrt_error_fig}
\end{figure}

\pagebreak

\section{Quality Improvement}
\label{sec;quality_improvement}

In this section, we introduce three explicit algebraic expressions for assessing the quality of pentatope elements in Euclidean space. Next, we modify these quality heuristics to account for the presence of an anisotropic metric field $M$. This enables us to characterize element quality with respect to a general, Riemannian metric space. Finally, we introduce a wide-array of bistellar flip operations which can be used to improve element quality in both the isotropic and anisotropic cases. 

\subsection{Pentatope Quality Heuristics}

Liu and Joe~\cite{liu1994shape} developed a quality measure for any tetrahedron that ranges from zero to one, where zero is a flat, zero-volume tetrahedron, and one is a regular tetrahedron. Following the derivation of their measure, we arrive at a similar quality heuristic for pentatopes.

\begin{theorem}
    The quality of any pentatope $P(p_{1}, p_{2}, p_{3}, p_{4}, p_{5})$ can be determined by
    \begin{align*}
        \eta^{(1)} = \frac{5^{3/4}\sqrt{384v}}{\sum^{10}_{i=1}l^{2}_{i}},
    \end{align*}
    where $v$ is the hypervolume, and $l_{i}$'s are the edge lengths of the pentatope. Here, `quality' is defined as the degree of similarity between an arbitrary pentatope and a regular pentatope with the same hypervolume.
    \label{geometric_theorem}
\end{theorem}

\begin{proof}
    Consider a regular pentatope $P_{r}(r_{1}, r_{2}, r_{3}, r_{4}, r_{5})$ with edge length $a$, and the same hypervolume as an arbitrary pentatope $P(p_{1}, p_{2}, p_{3}, p_{4}, p_{5})$. The coordinates of the regular pentatope are defined as $r_{1} = (-a\sqrt{3}/2, 0, 0, 0)$, $r_{2} = (0, -a/2, 0, 0)$, $r_{3} = (0, a/2, 0, 0)$, $r_{4} = (-a\sqrt{3}/6, 0, a\sqrt{6}/3, 0)$, and $r_{5} = (-a\sqrt{3}/6, 0, a\sqrt{6}/12, a\sqrt{10}/4)$.  Let $R = [(r_2 - r_1), (r_3 - r_1), (r_4 - r_1), (r_5 - r_1)]$ and $T = [(p_2 - p_1), (p_3 - p_1), (p_4 - p_1), (p_5 - p_1)]$. Then
    \begin{align*}
        R = a
        \begin{bmatrix}
            \sqrt{3}/2 & \sqrt{3}/2 & \sqrt{3}/3 & \sqrt{3}/3\\
            -1/2 & 1/2 & 0 & 0\\
            0 & 0 & \sqrt{6}/3 & \sqrt{6}/12\\
            0 & 0 & 0 & \sqrt{10}/4
        \end{bmatrix}, \qquad
        R^{-1} = \frac{1}{a}
        \begin{bmatrix}
            \frac{\sqrt{3}}{3} & -1 & -\frac{\sqrt{6}}{6} & -\frac{\sqrt{10}}{10}\\
            \frac{\sqrt{3}}{3} & 1 & -\frac{\sqrt{6}}{6} & -\frac{\sqrt{10}}{10}\\
            0 & 0 & \frac{\sqrt{6}}{2} & -\frac{\sqrt{10}}{10}\\
            0 & 0 & 0 & \frac{2\sqrt{10}}{5}
        \end{bmatrix},
    \end{align*}
    and
    \begin{align*}
        T^{T}T = 
        \begin{bmatrix}
            d_{12} & (d_{12} + d_{13} - d_{23})/{2} & (d_{12} + d_{14} - d_{24})/{2} & (d_{12} + d_{15} - d_{25})/{2}\\
            (d_{12} + d_{13} - d_{23})/{2} & d_{13} & (d_{13} + d_{14} - d_{34})/{2} & (d_{13} + d_{15} - d_{35})/{2}\\
            (d_{12} + d_{14} - d_{24})/{2} & (d_{13} + d_{14} - d_{34})/{2} & d_{14} & (d_{14} + d_{15} - d_{45})/{2}\\
            (d_{12} + d_{15} - d_{25})/{2} & (d_{13} + d_{15} - d_{35})/{2} & (d_{14} + d_{15} - d_{45})/{2} & d_{15}
        \end{bmatrix}.
    \end{align*}
    Here, $a = (96v/\sqrt{5})^{1/4}$ and $d_{ij} = (t_j - t_i)^{T}(t_j - t_i)$, $1 \leq i < j \leq 5$. The hypervolume $v$ is defined as
    \begin{align*}
        v &= \frac{\sqrt{d+1}}{2^{d/2}d!}a^{d} = \frac{\sqrt{4+1}}{2^{4/2}4!}a^{4} = \frac{\sqrt{5}}{96}a^{4}.
    \end{align*}
    Next, we define
    \begin{align*}
        &A(R,T) \equiv (R^{-1})^{T}T^{T}TR^{-1} = \frac{1}{a^2}
        \begin{bmatrix}
            \alpha & \# & \# & \#\\
            \# & \beta & \# & \#\\
            \# & \# & \gamma & \#\\
            \# & \# & \# & \delta
        \end{bmatrix},
    \end{align*}
    where $\alpha = (2d_{12} + 2d_{13} - d_{23})/3$, $\beta = d_{23}$, $\gamma = (3d_{14} + 3d_{34} + 3d_{24} - d_{12} - d_{13} - d_{23})/6$, and $\delta = (4d_{15} + 4d_{25} + 4d_{35} + 4d_{45} - d_{12} - d_{13} - d_{14} - d_{23} - d_{24} - d_{34})/10$. In addition, irrelevant values are indicated by the \# sign.

    The quality heuristic $\eta^{(1)}$ can be defined as the ratio of the geometric and arithmetic means for the inscribed 4-ellipsoid in the pentatope. As a result, $\eta^{(1)}$ can be written as follows
    \begin{align}
        \eta^{(1)} \equiv \frac{4\sqrt[4]{\lambda_{1}\lambda_{2}\lambda_{3}\lambda_{4}}}{\lambda_{1} + \lambda_{2} + \lambda_{3} + \lambda_{4}},
        \label{geo_ari_ratio}
    \end{align}
    where $\lambda_{1}$, $\lambda_{2}$, $\lambda_{3}$, and $\lambda_{4}$ are the eigenvalues of $A(R,T)$. Next, we have that
    \begin{align}
        \nonumber \lambda_{1} + \lambda_{2} + \lambda_{3} + \lambda_{4} &= \mathrm{trace}(A(R,T))\\ 
        \nonumber &= \frac{2}{5a^2}(d_{12} + d_{13} + d_{14} + d_{15} + d_{23} + d_{24} + d_{25} + d_{34} + d_{35} + d_{45})\\
        &= \frac{2}{5a^2}\sum^{10}_{i=1}l^{2}_{i}.
        \label{trace_def}
    \end{align}
    Because $P_r$ and $P$ have the same hypervolume,
    \begin{align}
        \lambda_{1} \lambda_{2} \lambda_{3} \lambda_{4} = \mathrm{det}(A(R,T)) = 1.
        \label{det_def}
    \end{align}

    By substituting Eqs.~\eqref{trace_def} and~\eqref{det_def} into Eq.~\eqref{geo_ari_ratio}, the desired quality heuristic for a pentatope is obtained
    \begin{align*}
        \eta^{(1)} &= \frac{4\sqrt[4]{\mathrm{det}(A(R,T))}}{\mathrm{trace}(A(R,T))} = \frac{4}{\frac{2}{5a^{2}}\sum^{10}_{i=1}l^{2}_{i}} = \frac{10\Bigg( \frac{96}{\sqrt{5}}v \Bigg)^{1/2}}{\sum^{10}_{i=1}l^{2}_{i}} = \frac{5^{3/4}\sqrt{384v}}{\sum^{10}_{i=1}l^{2}_{i}}.
    \end{align*}
\end{proof}

\begin{corollary}
    The quality of any pentatope $P(p_{1}, p_{2}, p_{3}, p_{4}, p_{5})$ can be determined by
    \begin{align*}
        \eta^{(2)} = \frac{6\sum^{10}_{i=1}l^{2}_{i}}{\sqrt{\Theta\left(\left\{l_{i}^{2} \right\}_{i=1}^{10} \right)}},
    \end{align*}
    where $l_{i}$ are the edge lengths of the pentatope, and
    \begin{align*}
        \Theta\left(\left\{l_{i}^{2} \right\}_{i=1}^{10} \right) &\equiv  600 (l_{1}^{2} - l_{2}^{2})^2 + 900 l_{5}^{4} + 
     100 (-2 (l_{1}^{2} + l_{2}^{2}) + l_{5}^{2})^2 \\ \nonumber
     &+ 
     75 (l_{1}^{2} - l_{2}^{2} - 3 l_{6}^{2} + 3 l_{8}^{2})^2 + 
     25 (l_{1}^{2} + l_{2}^{2} - 3 l_{3}^{2} + l_{5}^{2} - 3 (l_{6}^{2} + l_{8}^{2}))^2 \\
     \nonumber &+ 
     25 (l_{1}^{2} + l_{2}^{2} - 6 l_{3}^{2} - 2 l_{5}^{2} + 3 (l_{6}^{2} + l_{8}^{2}))^2 \\
     \nonumber &+ 
     45 (l_{1}^{2} - l_{2}^{2} + l_{6}^{2} - 4 l_{7}^{2} - l_{8}^{2} + 4 l_{9}^{2})^2 \\
     \nonumber &+ 
     15 (l_{1}^{2} + l_{2}^{2} + 2 l_{3}^{2} - 8 l_{4}^{2} - 2 l_{5}^{2} - l_{6}^{2} + 4 l_{7}^{2} - l_{8}^{2} + 
        4 l_{9}^{2})^2 \\ 
       \nonumber &+ 
     30 (-l_{1}^{2} - l_{2}^{2} + l_{3}^{2} + 2 l_{4}^{2} - l_{5}^{2} + l_{6}^{2} + 2 l_{7}^{2} + l_{8}^{2} + 2 l_{9}^{2} - 
        6 l_{10}^{2})^2 \\ \nonumber
        &+ 
     9 (l_{1}^{2} + l_{2}^{2} + l_{3}^{2} - 4 l_{4}^{2} + l_{5}^{2} + l_{6}^{2} - 4 l_{7}^{2} + l_{8}^{2} - 
        4 (l_{9}^{2} + l_{10}^{2}))^2.
    \end{align*}
    \label{rms_theorem}
\end{corollary}

\begin{proof}
    In the previous theorem, a quality heuristic was defined based on the ratio of the geometric and arithmetic means of the pentatope's inscribed 4-ellipsoid. It turns out, we can also define a heuristic based on the ratio of the arithmetic mean and the root-mean-square (RMS) as follows
    \begin{align}
        \eta^{(2)} \equiv \frac{\lambda_{1} + \lambda_{2} + \lambda_{3} + \lambda_{4}}{2\sqrt{\lambda_{1}^2 +\lambda_{2}^2 +\lambda_{3}^2+\lambda_{4}^2}}.
        \label{rms_ari_ratio}
    \end{align}
    The denominator of this expression is twice the Frobenius norm of $A(R,T)$
    \begin{align}
       \nonumber \sqrt{\lambda_{1}^2 +\lambda_{2}^2 +\lambda_{3}^2+\lambda_{4}^2} & = \sqrt{\mathrm{trace}(A(R,T) A(R,T)^T)} = \left\| A(R,T) \right\|_{F} \\
       \nonumber &=\frac{1}{30 a^2} \Big[600 (d_{12} - d_{13})^2 + 900 d_{23}^2 + 
     100 (-2 (d_{12} + d_{13}) + d_{23})^2 \\ \nonumber
     &+ 
     75 (d_{12} - d_{13} - 3 d_{24} + 3 d_{34})^2 + 
     25 (d_{12} + d_{13} - 3 d_{14} + d_{23} - 3 (d_{24} + d_{34}))^2 \\
     \nonumber &+ 
     25 (d_{12} + d_{13} - 6 d_{14} - 2 d_{23} + 3 (d_{24} + d_{34}))^2 \\
     \nonumber &+ 
     45 (d_{12} - d_{13} + d_{24} - 4 d_{25} - d_{34} + 4 d_{35})^2 \\
     \nonumber &+ 
     15 (d_{12} + d_{13} + 2 d_{14} - 8 d_{15} - 2 d_{23} - d_{24} + 4 d_{25} - d_{34} + 
        4 d_{35})^2 \\ 
       \nonumber &+ 
     30 (-d_{12} - d_{13} + d_{14} + 2 d_{15} - d_{23} + d_{24} + 2 d_{25} + d_{34} + 2 d_{35} - 
        6 d_{45})^2 \\ \nonumber
        &+ 
     9 (d_{12} + d_{13} + d_{14} - 4 d_{15} + d_{23} + d_{24} - 4 d_{25} + d_{34} - 
        4 (d_{35} + d_{45}))^2\Big]^{1/2} \\ \nonumber
        &= \frac{1}{30 a^2} \sqrt{\Theta(d_{12}, d_{13}, d_{14}, d_{15}, d_{23}, d_{24}, d_{25}, d_{34}, d_{35}, d_{45})}
        \\
        &= \frac{1}{30 a^2} \sqrt{\Theta\left(l_{1}^{2}, l_{2}^{2},l_{3}^{2},l_{4}^{2},l_{5}^{2},l_{6}^{2},l_{7}^{2},l_{8}^{2},l_{9}^{2},l_{10}^{2} \right)} = \frac{1}{30 a^2} \sqrt{\Theta\left(\left\{l_{i}^{2} \right\}_{i=1}^{10} \right)}.
        \label{fro_norm}
    \end{align}
    Next, if we substitute Eqs.~\eqref{trace_def} and \eqref{fro_norm} into Eq.~\eqref{rms_ari_ratio} we obtain the desired result
    \begin{align*}
        \eta^{(2)} = \frac{\mathrm{trace}(A(R,T))}{2 \sqrt{\mathrm{trace}(A(R,T) A(R,T)^T)}} = \frac{\frac{2}{5a^{2}}\sum^{10}_{i=1}l^{2}_{i}}{\frac{2}{30 a^2} \sqrt{\Theta\left(\left\{l_{i}^{2} \right\}_{i=1}^{10} \right)}} =  \frac{6\sum^{10}_{i=1}l^{2}_{i}}{\sqrt{\Theta\left(\left\{l_{i}^{2} \right\}_{i=1}^{10} \right)}}.
    \end{align*}
\end{proof}


\begin{corollary}
    The quality of any pentatope $P(p_{1}, p_{2}, p_{3}, p_{4}, p_{5})$ can be determined by
    \begin{align*}
        \eta^{(3)} = 6 \cdot 5^{3/4} \sqrt{\frac{ 384 v}{\Theta\left(\left\{l_{i}^{2} \right\}_{i=1}^{10} \right)}},
    \end{align*}
    where $v$ is the hypervolume, and $l_{i}$'s are the edge lengths of the pentatope.
    \label{geom_rms_theorem}
\end{corollary}

\begin{proof}
    The quality heuristic above is simply the ratio of the geometric mean and the RMS, for the inscribed 4-ellipsoid. We can obtain this quality heuristic immediately by multiplying the quality heuristics of Theorem~\ref{geometric_theorem} and Corollary~\ref{rms_theorem} together.
\end{proof}

\begin{remark}
    Suppose that the eigenvalues of the inscribed 4-ellipsoid are non-negative, i.e. $\lambda_1 \geq 0$, $\lambda_2 \geq 0$, $\lambda_3 \geq 0$, and $\lambda_4 \geq 0$. Under these circumstances, it is well known that the following relationship holds for the RMS, arithmetic mean (AM), and geometric mean (GM) of the eigenvalues
    \begin{align}
        \mathrm{RMS} \geq \mathrm{AM} \geq \mathrm{GM} \geq 0. \label{inequality_of_means}
    \end{align}
    Therefore, upon dividing Eq.~\eqref{inequality_of_means} by the RMS, we conclude that
    \begin{align*}
        1 \geq \eta^{(2)} \geq \eta^{(3)} \geq 0.
    \end{align*}
    In addition, upon dividing Eq.~\eqref{inequality_of_means} by the AM, we conclude that
    \begin{align*}
        1 \geq \eta^{(1)} \geq 0.
    \end{align*}
    Therefore, each quality heuristic is guaranteed to lie on the interval $[0,1]$. In addition, the third quality heuristic is generally, more conservative than the second. 
\end{remark}

\begin{remark}
    In~\cite{knupp2001algebraic}, Knupp refers to the first quality heuristic, $\eta^{(1)}$, as the `mean ratio shape measure' for triangles and tetrahedra in two and three dimensions, respectively. The three-dimensional version of this heuristic was first introduced by Liu and Joe~\cite{liu1994shape}, as mentioned previously. Here, we have extended this measure to four dimensions, and developed two new heuristics: $\eta^{(2)}$ and $\eta^{(3)}$. 
\end{remark}

\begin{remark}
    All three heuristics, $\eta^{(1)}$, $\eta^{(2)}$, and $\eta^{(3)}$, have desirable properties. In particular, they are dimension-free, domain-general, scale-free, orientation-free, unitless, and referenced, in accordance with the definitions of Table 1 in~\cite{knupp2001algebraic}. Equivalently, in less technical language: the heuristics can be applied to simplices in any number of dimensions $d\geq 2$, they can be applied to all pentatopes (skewed and non-skewed), they do not depend on the size of the pentatope, they do not depend on the orientation of the pentatope, they are non-dimensional, and they assume the value of one for a regular pentatope.  
\end{remark}

\subsection{Pentatope Quality Heuristics in Metric Space}

The quality heuristics from the previous section, $\eta^{(1)}$, $\eta^{(2)}$, and $\eta^{(3)}$ can be easily modified in order to account for the presence of a given metric field $M = M(x,y,z,t)$. One merely needs to replace the Euclidean lengths $l_i$'s and volumes $v$ with metric-scaled lengths and volumes:
\begin{align*}
    l_i \Rightarrow l_{{M}_{i}}, \qquad v \Rightarrow v_{M}.
\end{align*}
The computation of these metric-scaled lengths and volumes can be performed in two ways: i) by evaluating the underlying metric field point-wise, or ii) through numerical integration.

\subsubsection{Point-wise Approach}

Since the metric field is theoretically defined at an infinite number of points, it is convenient to evaluate it at a single point that is associated with our particular element. With this in mind, one may consider the centroid of a space-time pentatope: $(x_c, y_c, z_c, t_c)$. One may evaluate the metric field at this point $M_c = M(x_c, y_c, z_c, t_c)$. Thereafter, one may compute the length of an edge in metric space between points $V_1 = (x_1, y_1, z_1, t_1)$ and $V_2 = (x_2, y_2, z_2, t_2)$  as follows
\begin{align*}
    l_{M_{c}} = \sqrt{(V_1 - V_2)^{T} M_{c} (V_1 - V_2)}.
\end{align*}
In addition, one may compute the volume in metric space directly based on the volume in Euclidean space
\begin{align*}
    v_{M_{c}} = v \sqrt{\det(M_c)}.
\end{align*}

\subsubsection{Integration Approach}

One can also compute the lengths and volumes in metric space more accurately by using numerical integration. In particular, one may compute the length of an edge in metric space as follows
\begin{align*}
    l_{M} = \int_{0}^{1} \sqrt{(V_2 - V_1)^{T} M(V_1 + (V_2 - V_1)\tau) (V_2 - V_1)} d\tau.
\end{align*}
In a similar fashion, we have that
\begin{align*}
    v_{M} = \int_{T^4} \sqrt{\mathrm{det}(M(x,y,z,t))} \, dx dy dz dt.
\end{align*}
In both of the expressions above, the suggested integration can be carried out using numerical quadrature rules, such as those proposed by Gauss, or more recently by~\cite{frontin2021foundations,chuluunbaatar2022pi}.

\subsection{Bistellar Flips}

Bistellar flips have been used successfully in order to improve mesh quality, in accordance with the pioneering work of Shewchuk~\cite{shewchuk2003updating, cheng2013delaunay}. Bistellar flips involve reconnecting/repartitioning groups of low quality elements to create alternative groups of higher quality elements, while simultaneously maintaining the original external facets of the group. There exists  $(d+1)$ possible flipping operations in each dimension, which are associated with the transformations of $k$ elements into $(d+2)-k$ elements, where $d$ is the number of dimensions (introduced previously as $d = 4$) and $k$ is the number of initial elements. In 4D, the five basic flipping operations are: $(1 \rightarrow 5)$, $(2 \rightarrow 4)$, $(3 \rightarrow 3)$, $(4 \rightarrow 2)$, and $(5 \rightarrow 1)$. The number of vertices required for the starting and ending configurations can be determined by Eq.~\eqref{number_vertex}. The basic 4D bistellar flips and their corresponding connectivity transformations are shown in Figures~\ref{flip_1_5}-\ref{flip_3_3}. We note that some of the 3D representations of the 4D flips look identical before and after the transformation because the representation is essentially a \emph{shadow} of a higher dimensional object, and is unable to capture all of its detail.
\begin{align}
    n_{vertex} =
    \begin{cases}
        d+1 & \text{$if \quad k = 1$}\\
        d+2 & \text{$if \quad k > 1$}
    \end{cases}.
    \label{number_vertex}
\end{align}

\begin{figure}[h!]
\centering
$\vcenter{\hbox{\includegraphics[width=3.5cm]{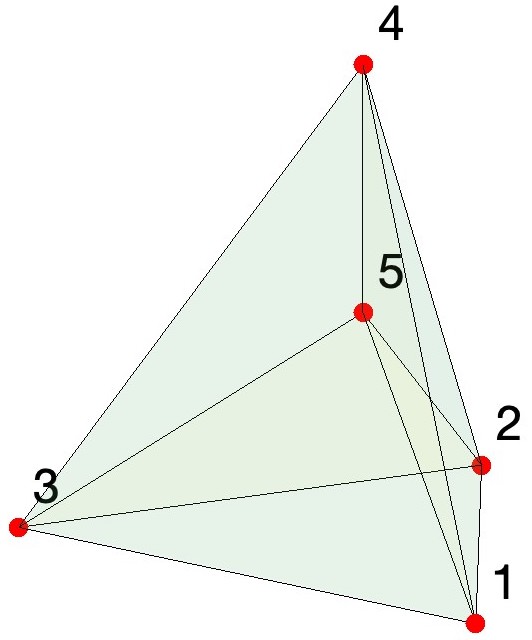}}}$
$\vcenter{\hbox{\includegraphics[width=3.5cm]{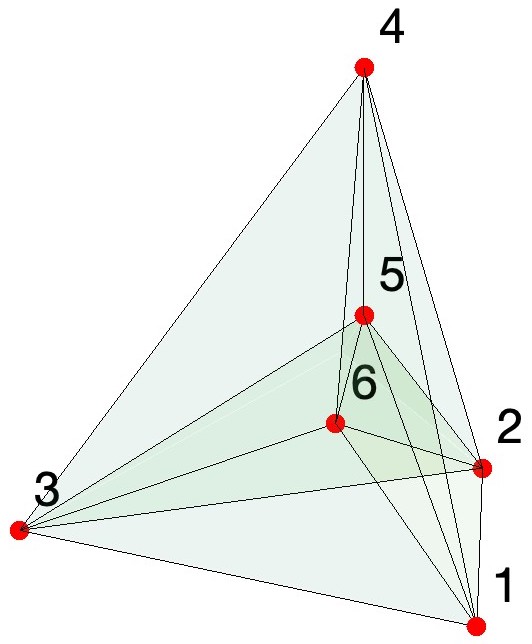}}}$
$\vcenter{\hbox{
\begin{tabular}{|c|c|}
\hline
Stage 1 & Stage 2 \\
\hline
1 2 3 4 5 & 1 2 3 4 6 \\
& 2 3 4 5 6 \\
& 1 3 4 5 6 \\
& 1 2 4 5 6 \\
& 1 2 3 5 6 \\
\hline
\end{tabular}
}}$
\caption{Flip ($1\rightarrow 5$) partitions one pentatope (left) into five (right) by inserting a point within the element and reconnecting. Reversing the process is also valid via flip ($5\rightarrow 1$). The table summarizes the connectivity transformation.}
\label{flip_1_5}
\end{figure}

\begin{figure}[h!]
\centering
$\vcenter{\hbox{\includegraphics[width=2.7cm]{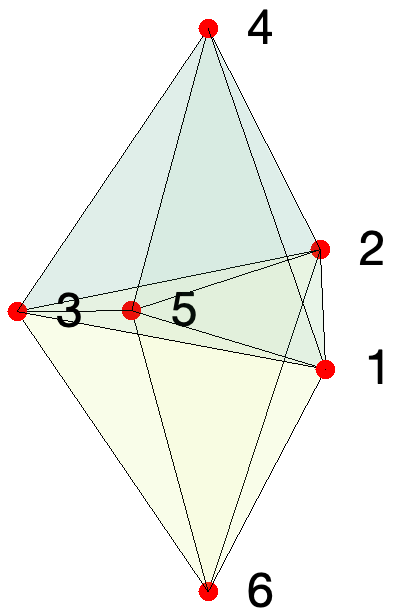}}}$
$\vcenter{\hbox{\includegraphics[width=2.7cm]{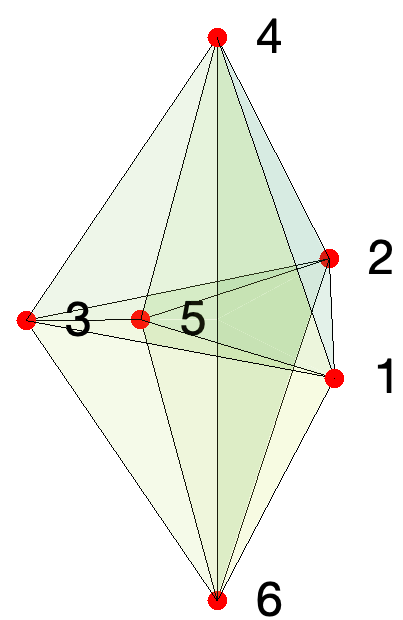}}}$
$\vcenter{\hbox{
\begin{tabular}{|c|c|}
\hline
Stage 1 & Stage 2 \\
\hline
1 2 3 4 5 & 1 2 3 4 6 \\
1 2 3 5 6 & 1 2 4 5 6 \\
& 2 3 4 5 6 \\
& 1 3 4 5 6 \\
\hline
\end{tabular}
}}$
\caption{Flip ($2\rightarrow 4$) partitions two pentatopes (left) into four (right) through reconnection. Reversing the process is also valid via flip ($4\rightarrow 2$). The table summarizes the connectivity transformation.}
\label{flip_2_4}
\end{figure}
%
\begin{figure}[h!]
\centering
$\vcenter{\hbox{\includegraphics[width=3.5cm]{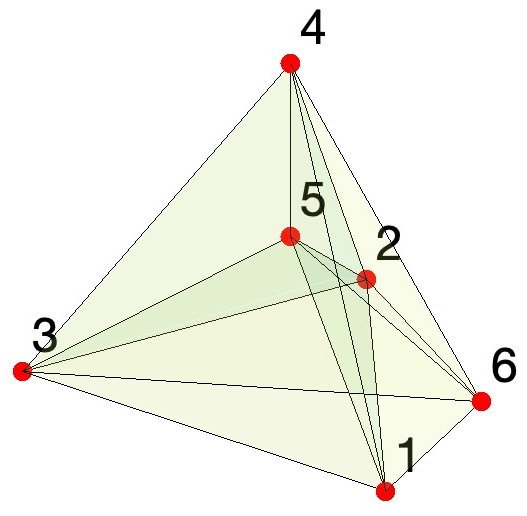}\hspace{1cm}}}$
$\vcenter{\hbox{
\begin{tabular}{|c|c|}
\hline
Stage 1 & Stage 2 \\
\hline
1 2 3 4 5 & 1 2 3 4 6 \\
1 2 4 5 6 & 2 3 4 5 6 \\
1 3 4 5 6 & 1 2 3 5 6 \\
\hline
\end{tabular}
}}$
\caption{Flip ($3\rightarrow 3$) partitions three pentatopes into three alternative pentatopes through reconnection. 3D representations of both configurations are identical. Reversing the process is also valid. The table summarizes the connectivity transformation.}
\label{flip_3_3}
\end{figure}

\noindent The basic 4D flips have already been identified in previous work; other authors refer to them as 4D \emph{Pachner moves}~\cite{banburski2015pachner}. Expanding on the basic set of 4D flips, we derive a new, more exhaustive set of possible flips by extending lower-dimensional flips into higher dimensions as shown in Table~\ref{extended_flips_table}. These extended flips are formed by taking lower-dimensional subsimplices (facets, faces, or edges) of our pentatopes, performing lower-dimensional flips on these entities, then updating the connectivity.  For example, suppose that a grouping of three pentatopes share a triangular face. This face can be split into three triangles by inserting a point and performing a $(1 \rightarrow 3)$ flip in 2D. Thereafter, a reconnection operation can be performed in order to obtain nine pentatopes from the original three. This corresponds to an extended 4D flip, called a $(3 \rightarrow 9)$ flip, which directly inherits from the lower-dimensional $(1 \rightarrow 3)$ 2D flip. We have identified a total of 12 such extended flips, and the complete set is shown in Figures~\ref{flip_4_8}--\ref{flip_8_16} along with their respective connectivity transformations. It is worth noting that some of the 4D extended flips have multiple versions. As an example, the 3D extended $(4 \rightarrow 4)$ flip has three different $(8 \rightarrow 8)$ configurations in 4D.

\begin{remark}
    It is important to note that bistellar flips are not always guaranteed to preserve the Delaunay criterion. Therefore, if one's particular application requires that the Delaunay property is satisfied exactly, bistellar flips must be employed with some degree of caution.
\end{remark}

\begin{table}[h!]
\begin{center}
\begin{tabular}{|p{3.5cm}|p{3.5cm}|p{3.5cm}|}
\hline
\multicolumn{3}{|c|}{\textbf{4D Extended Flips}} \\
\hline
$1D \Rightarrow 4D$ & $2D \Rightarrow 4D$ & $3D \Rightarrow 4D$\\
\hline
$(1 \rightarrow 2) \Rightarrow (4 \rightarrow 8)$ & $(1 \rightarrow 3) \Rightarrow (3 \rightarrow 9)$ & $(1 \rightarrow 4) \Rightarrow (2 \rightarrow 8)$\\
& $(2 \rightarrow 2) \Rightarrow (6 \rightarrow 6)$ & $(2 \rightarrow 3) \Rightarrow (4 \rightarrow 6)$\\
& $(2 \rightarrow 4) \Rightarrow (6 \rightarrow 12)$a & $(4 \rightarrow 4) \Rightarrow (8 \rightarrow 8)^*$\\
& & $(2 \rightarrow 6) \Rightarrow (4 \rightarrow 12)$\\
& & $(3 \rightarrow 6) \Rightarrow (6 \rightarrow 12)$b\\
& & $(4 \rightarrow 8) \Rightarrow (8 \rightarrow 16)$\\
\hline
\end{tabular}
\caption{Lower dimensional flips extended to four dimensions. The * indicates that the $(8 \rightarrow 8)$ flip has three different realizations.}
\label{extended_flips_table}
\end{center}
\end{table}
%

\begin{figure}[h!]
\centering
$\vcenter{\hbox{\includegraphics[width=3cm]{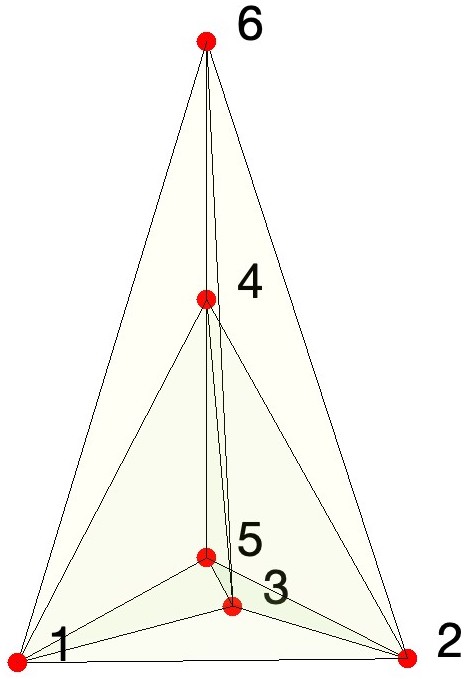}}}$
$\vcenter{\hbox{\includegraphics[width=3cm]{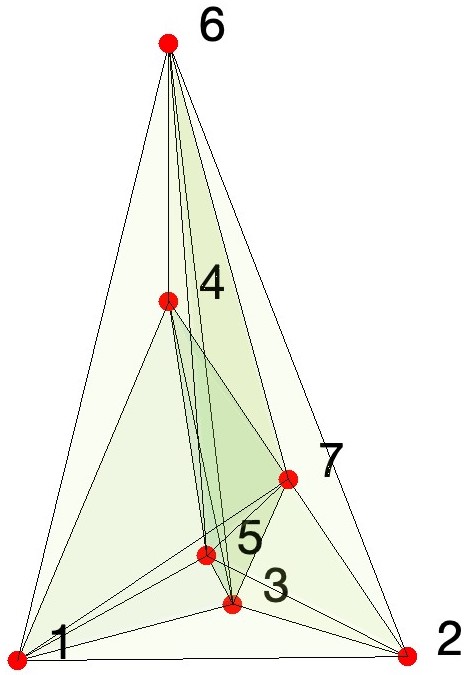}}}$
$\vcenter{\hbox{
\begin{tabular}{|c|c|}
\hline
Stage 1 & Stage 2 \\
\hline
1 2 3 4 5 & 1 3 4 5 7 \\
1 2 4 5 6 & 1 2 3 5 7 \\
2 3 4 5 6 & 1 4 5 6 7 \\
1 2 3 4 6 & 1 2 5 6 7 \\
& 3 4 5 6 7 \\
& 2 3 5 6 7 \\
& 1 3 4 6 7 \\
& 1 2 3 6 7 \\
\hline
\end{tabular}
}}$
\caption{Flip ($4 \rightarrow 8$) partitions four pentatopes (left) into eight (right) by inserting a point along the shared edge and reconnecting. Reversing the process is also valid via flip ($8 \rightarrow 4$).  The flip is a 4D extension of the 1D, ($1 \rightarrow 2$) flip. The table summarizes the connectivity transformation.}
\label{flip_4_8}
\end{figure}

\begin{figure}[h!]
\centering
$\vcenter{\hbox{\includegraphics[width=3.5cm]{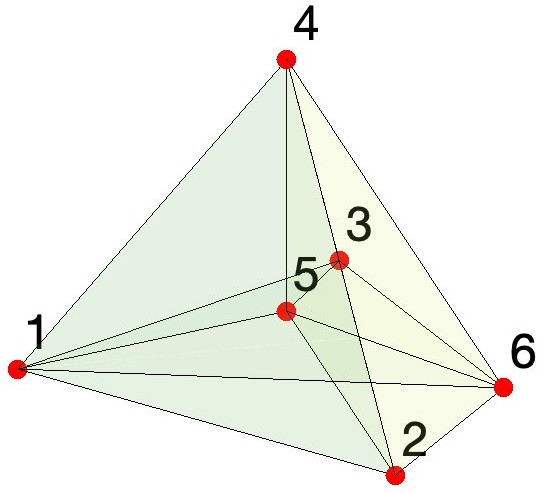}}}$
$\vcenter{\hbox{\includegraphics[width=3.5cm]{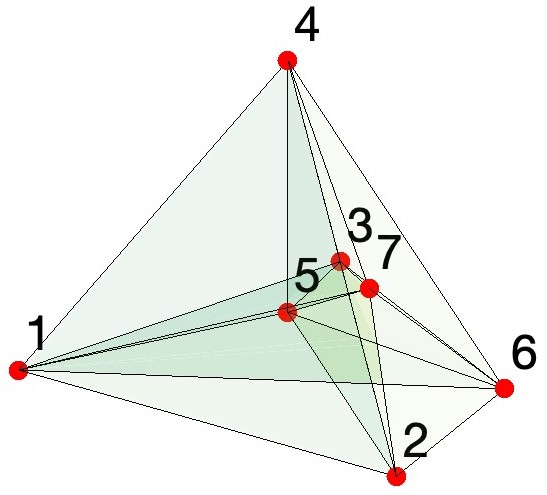}}}$
$\vcenter{\hbox{
\begin{tabular}{|c|c|}
\hline
Stage 1 & Stage 2 \\
\hline
1 2 3 4 5 & 1 2 3 4 7 \\
1 2 4 5 6 & 1 3 4 5 7 \\
2 3 4 5 6 & 1 2 3 5 7 \\
& 1 4 5 6 7 \\
& 1 2 5 6 7 \\
& 1 2 4 6 7 \\
& 3 4 5 6 7 \\
& 2 3 5 6 7 \\
& 2 3 4 6 7 \\
\hline
\end{tabular}
}}$
\caption{Flip ($3 \rightarrow 9$) partitions three pentatopes (left) into nine (right) by inserting a point within the shared triangle and reconnecting. Reversing the process is also valid via flip ($9 \rightarrow 3$). The flip is a 4D extension of the 2D, ($1 \rightarrow 3$) flip. The table summarizes the connectivity transformation.}
\label{flip_3_9}
\end{figure}

\begin{figure}[h!]
\centering
$\vcenter{\hbox{\includegraphics[width=3.5cm]{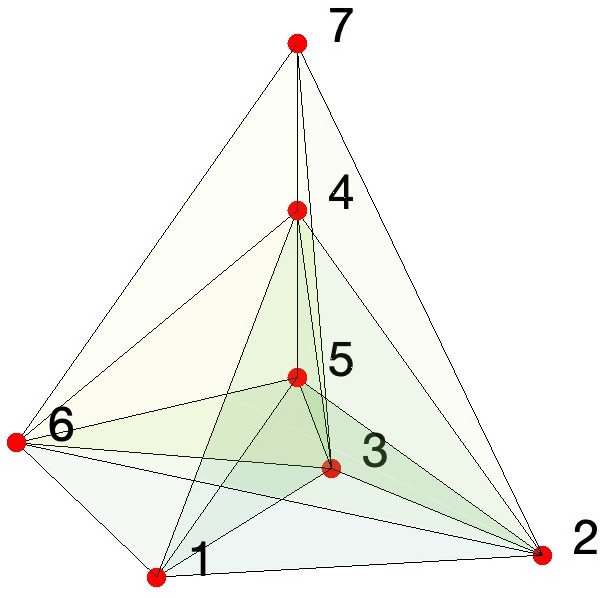}}}$
$\vcenter{\hbox{\includegraphics[width=3.5cm]{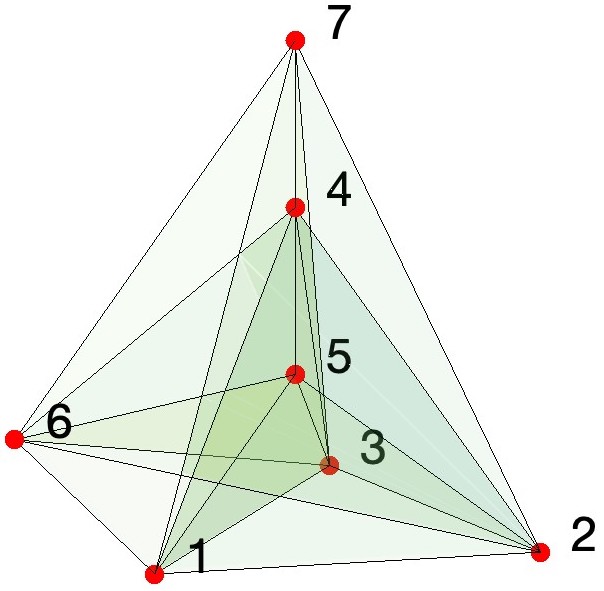}}}$
$\vcenter{\hbox{
\begin{tabular}{|c|c|}
\hline
Stage 1 & Stage 2 \\
\hline
1 2 3 4 5 & 1 2 3 4 7 \\
1 2 4 5 6 & 1 2 4 6 7 \\
1 3 4 5 6 & 1 3 4 6 7 \\
2 3 4 5 7 & 1 2 3 5 7 \\
2 4 5 6 7 & 1 2 5 6 7 \\
3 4 5 6 7 & 1 3 5 6 7 \\
\hline
\end{tabular}
}}$
\caption{Flip ($6\rightarrow6$) partitions six pentatopes (left) into  six alternative pentatopes (right) through reconnection on the shared grouping of triangles. The flip is reversible. In addition, the flip is a 4D extension of the 2D, ($2\rightarrow 2$) flip. The table summarizes the connectivity transformation.}
\label{flip_6_6}
\end{figure}

\begin{figure}[h!]
\centering
$\vcenter{\hbox{\includegraphics[width=5.5cm]{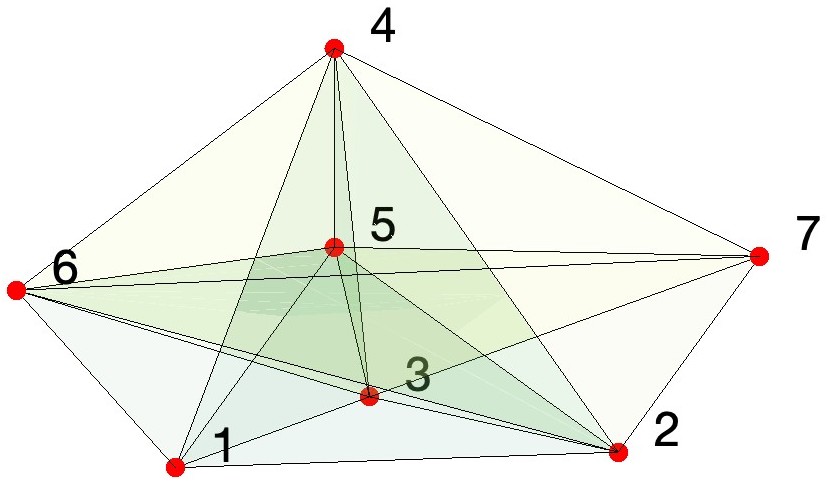}}}$
$\vcenter{\hbox{\includegraphics[width=5.5cm]{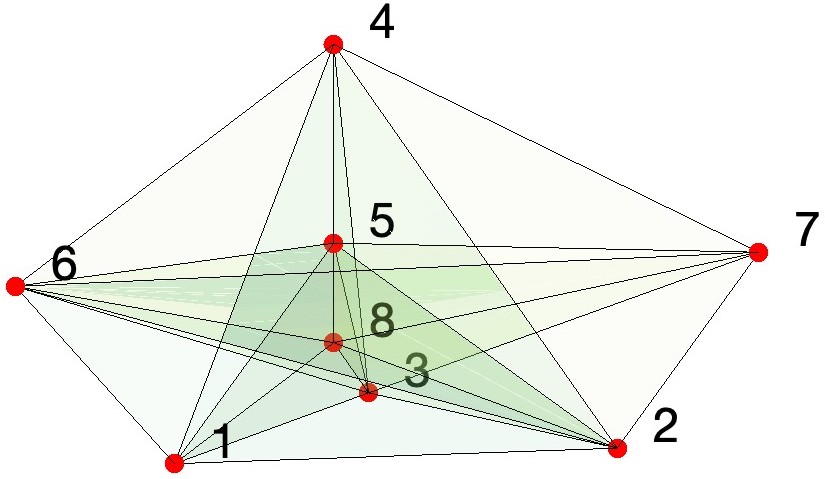}}}$
$\vcenter{\hbox{
\begin{tabular}{|c|c|c|}
\hline
Stage 1 & \multicolumn{2}{|c|}{Stage 2}  \\
\hline
1 2 3 4 5 & 1 2 3 4 8 & 2 3 4 7 8 \\
1 2 4 5 6 & 1 2 4 6 8 & 2 4 6 7 8 \\
1 3 4 5 6 & 1 3 4 6 8 & 3 4 6 7 8 \\
2 3 4 5 7 & 1 2 3 5 8 & 2 3 5 7 8 \\
2 4 5 6 7 & 1 2 5 6 8 & 2 5 6 7 8 \\
3 4 5 6 7 & 1 3 5 6 8 & 3 5 6 7 8 \\
\hline
\end{tabular}
}}$
\caption{Flip ($6\rightarrow 12$)a partitions six pentatopes (left) into twelve (right) by inserting a point within the shared grouping of triangles and reconnecting. Reversing the process is also valid via flip ($12\rightarrow 6$)a. The flip is a 4D extension of the 2D, ($2\rightarrow 4$) flip. The table summarizes the connectivity transformation.}
\label{flip_6_12a}
\end{figure}

\begin{figure}[h!]
\centering
$\vcenter{\hbox{\includegraphics[width=2.3cm]{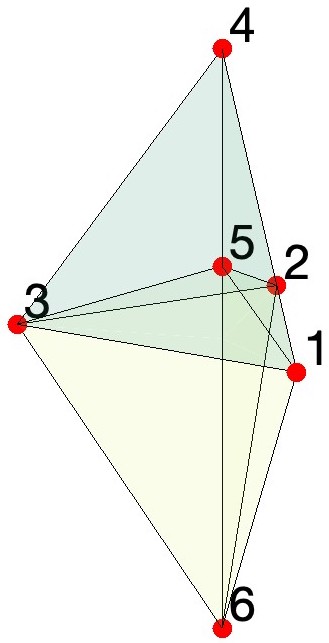}}}$
$\vcenter{\hbox{\includegraphics[width=2.3cm]{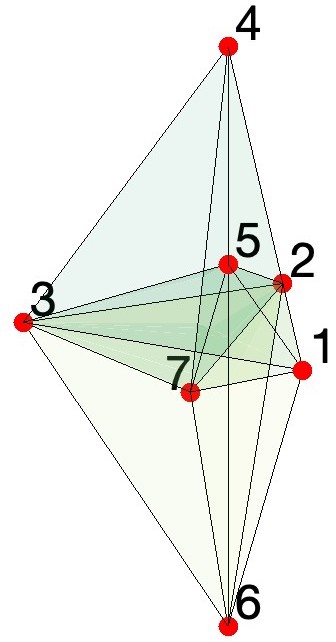}\hspace{1cm}}}$
$\vcenter{\hbox{
\begin{tabular}{|c|c|}
\hline
Stage 1 & Stage 2 \\
\hline
1 2 3 4 5 & 1 2 3 4 7 \\
1 2 3 5 6 & 2 3 4 5 7 \\
& 1 3 4 5 7 \\
& 1 2 4 5 7 \\
& 2 3 5 6 7 \\
& 1 3 5 6 7 \\
& 1 2 5 6 7 \\
& 1 2 3 6 7 \\
\hline
\end{tabular}
}}$
\caption{Flip ($2\rightarrow 8$) partitions two pentatopes (left) into eight (right) by inserting a point within the shared tetrahedron and reconnecting. Reversing the process is also valid via flip ($8 \rightarrow 2$). The flip is a 4D extension of the 3D, ($1\rightarrow 4$) flip. The table summarizes the connectivity transformation.}
\label{flip_2_8}
\end{figure}
%

\begin{figure}[h!]
\centering
$\vcenter{\hbox{\includegraphics[width=4cm]{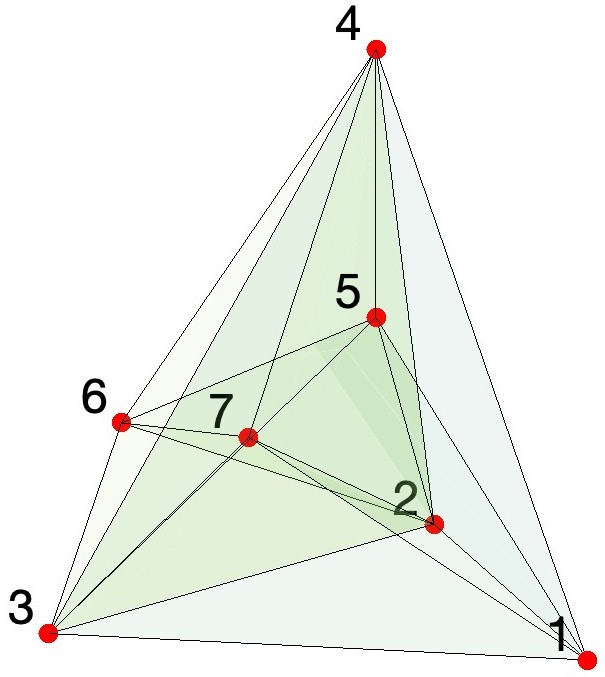}\hspace{1cm}}}$
$\vcenter{\hbox{
\begin{tabular}{|c|c|}
\hline
Stage 1 & Stage 2 \\
\hline
1 2 3 4 5 & 1 2 3 6 7 \\
2 3 4 5 6 & 1 3 4 6 7 \\
1 2 3 4 7 & 1 2 4 5 6 \\
2 3 4 6 7 & 1 3 4 5 6 \\
& 1 2 3 5 6 \\
& 1 2 4 6 7 \\
\hline
\end{tabular}
}}$
\caption{Flip ($4\rightarrow 6$) partitions four pentatopes into six through reconnection of a shared grouping of tetrahedra. The 3D representations of both configurations are identical. Reversing the process is also valid via flip ($6\rightarrow 4$). The flip is a 4D extension of the 3D, ($2\rightarrow 3$) flip. The table summarizes the connectivity transformation.}
\label{flip_4_6}
\end{figure}
%

\begin{figure}[h!]
\centering
$\vcenter{\hbox{\includegraphics[width=2.3cm]{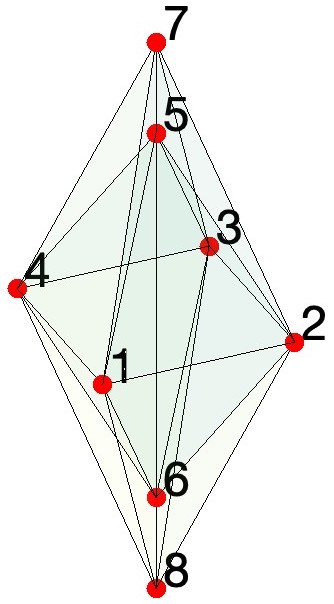}}}$
$\vcenter{\hbox{\includegraphics[width=2.3cm]{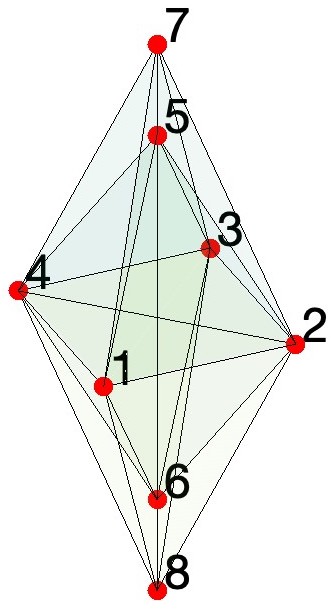}\hspace{1cm}}}$
$\vcenter{\hbox{
\begin{tabular}{|c|c|}
\hline
Stage 1 & Stage 2 \\
\hline
1 2 5 6 7 & 1 2 4 5 7 \\
2 3 5 6 7 & 2 3 4 5 7 \\
3 4 5 6 7 & 1 2 4 6 7 \\
1 4 5 6 7 & 2 3 4 6 7 \\
1 2 5 6 8 & 1 2 4 5 8 \\
2 3 5 6 8 & 2 3 4 5 8 \\
3 4 5 6 8 & 1 2 4 6 8 \\
1 4 5 6 8 & 2 3 4 6 8 \\
\hline
\end{tabular}
}}$
\caption{Flip ($8\rightarrow 8$) version 1, partitions eight pentatopes (left) into eight alternative pentatopes (right) through reconnection of a shared grouping of tetrahedra. The flip is reversible. In addition, the flip is a 4D extension of the 3D, ($4\rightarrow 4$) flip. The table summarizes the connectivity transformation.}
\label{flip_8_8_v1}
\end{figure}

\begin{figure}[h!]
\centering
$\vcenter{\hbox{\includegraphics[width=2.3cm]{flip_8_8_I1.jpg}}}$
$\vcenter{\hbox{\includegraphics[width=2.3cm]{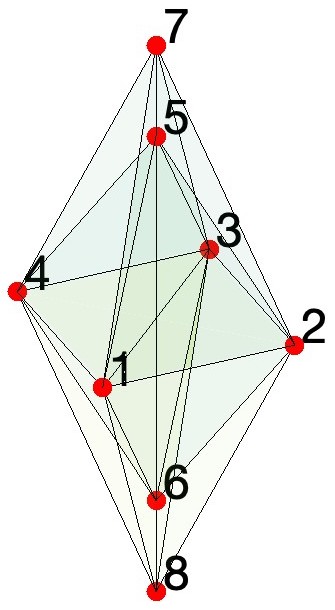}\hspace{1cm}}}$
$\vcenter{\hbox{
\begin{tabular}{|c|c|}
\hline
Stage 1 & Stage 2 \\
\hline
1 2 5 6 7 & 1 2 3 5 7 \\
2 3 5 6 7 & 1 3 4 5 7 \\
3 4 5 6 7 & 1 2 3 6 7 \\
1 4 5 6 7 & 1 3 4 6 7 \\
1 2 5 6 8 & 1 2 3 5 8 \\
2 3 5 6 8 & 1 3 4 5 8 \\
3 4 5 6 8 & 1 2 3 6 8 \\
1 4 5 6 8 & 1 3 4 6 8 \\
\hline
\end{tabular}
}}$
\caption{Flip ($8\rightarrow 8$) version 2, partitions eight pentatopes (left) into eight alternative pentatopes (right) through reconnection of a shared grouping of tetrahedra. The flip is reversible. In addition, the flip is a 4D extension of the 3D, ($4 \rightarrow 4$) flip. The table summarizes the connectivity transformation.}
\label{flip_8_8_v2}
\end{figure}

\begin{figure}[h!]
\centering
$\vcenter{\hbox{\includegraphics[width=2.3cm]{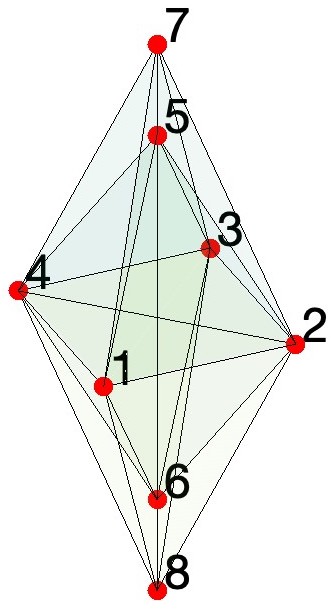}}}$
$\vcenter{\hbox{\includegraphics[width=2.3cm]{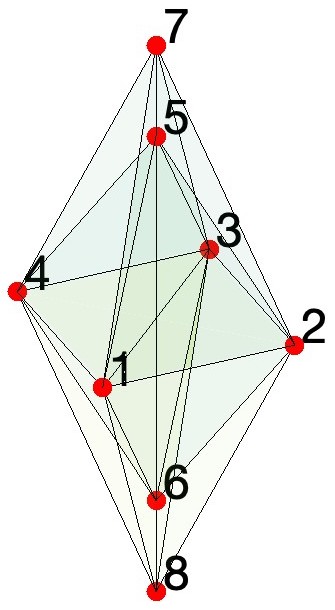}\hspace{1cm}}}$
$\vcenter{\hbox{
\begin{tabular}{|c|c|}
\hline
Stage 1 & Stage 2 \\
\hline
1 2 4 5 7 & 1 2 3 5 7 \\
2 3 4 5 7 & 1 3 4 5 7 \\
1 2 4 6 7 & 1 2 3 6 7 \\
2 3 4 6 7 & 1 3 4 6 7 \\
1 2 4 5 8 & 1 2 3 5 8 \\
2 3 4 5 8 & 1 3 4 5 8 \\
1 2 4 6 8 & 1 2 3 6 8 \\
2 3 4 6 8 & 1 3 4 6 8 \\
\hline
\end{tabular}
}}$
\caption{Flip ($8 \rightarrow 8$) version 3, partitions eight pentatopes (left) into eight alternative pentatopes (right) through reconnection of a shared grouping of tetrahedra. The flip is reversible. In addition, the flip is a 4D extension of the 3D, ($4\rightarrow 4$) flip. The table summarizes the connectivity transformation.}
\label{flip_8_8_v3}
\end{figure}

\begin{figure}[h!]
\centering
$\vcenter{\hbox{\includegraphics[width=4cm]{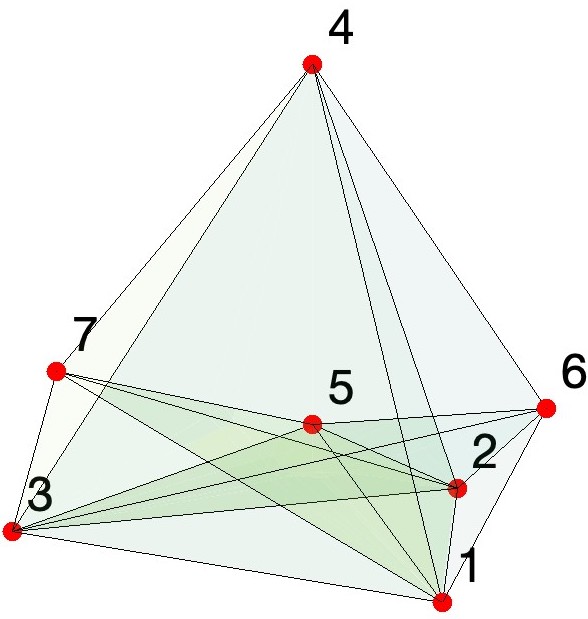}}}$
$\vcenter{\hbox{\includegraphics[width=4cm]{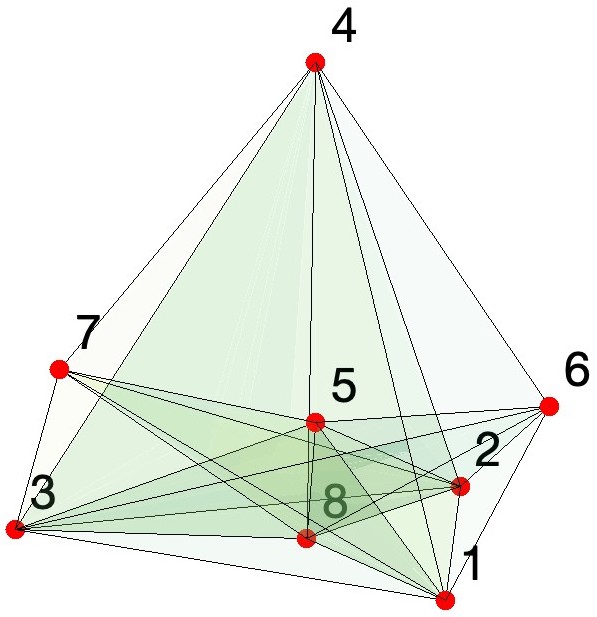}}}$
$\vcenter{\hbox{
\begin{tabular}{|c|c|c|}
\hline
Stage 1 & \multicolumn{2}{|c|}{Stage 2} \\
\hline
1 2 3 4 6 & 2 3 4 6 8 & 2 3 4 7 8  \\
1 2 3 5 6 & 1 3 4 6 8 & 1 3 4 7 8\\
1 2 3 4 7 & 1 2 4 6 8 & 1 2 4 7 8\\
1 2 3 5 7 & 2 3 5 6 8 & 2 3 5 7 8 \\
& 1 3 5 6 8 & 1 3 5 7 8 \\
& 1 2 5 6 8 & 1 2 5 7 8\\
\hline
\end{tabular}
}}$
\caption{Flip ($4 \rightarrow 12$) partitions four pentatopes (left) into twelve (right) by inserting a point within the shared triangle and reconnecting. Reversing the process is also valid via flip ($12 \rightarrow 4$). The flip is a 4D extension of the 3D, ($2\rightarrow 6$) flip. The table summarizes the connectivity transformation.}
\label{flip_4_12}
\end{figure}
%

\begin{figure}[h!]
\centering
$\vcenter{\hbox{\includegraphics[width=5.5cm]{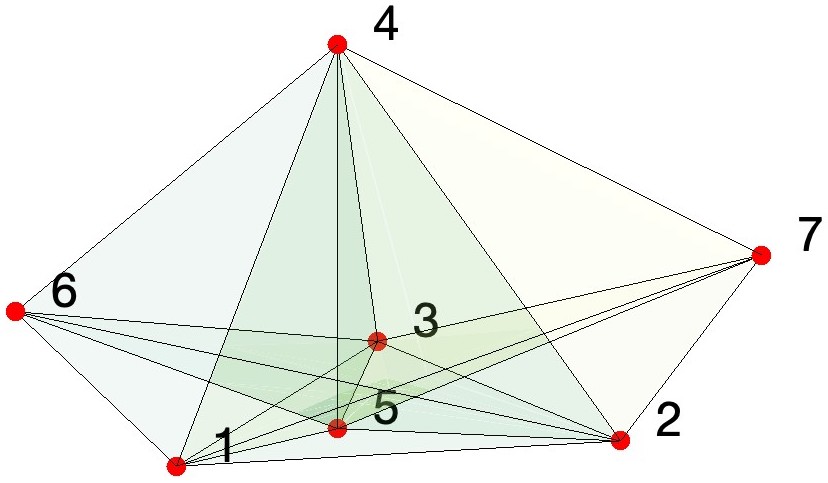}}}$
$\vcenter{\hbox{\includegraphics[width=5.5cm]{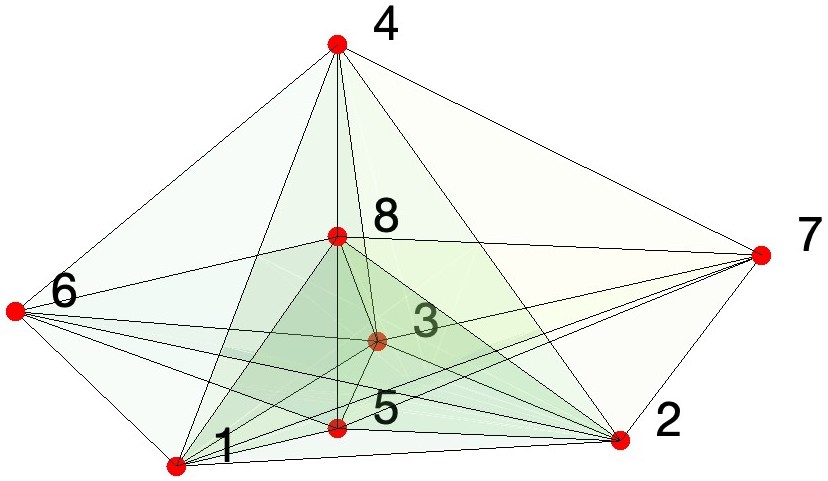}}}$
$\vcenter{\hbox{
\begin{tabular}{|c|c|c|}
\hline
Stage 1 & \multicolumn{2}{|c|}{Stage 2}  \\
\hline
1 2 4 5 6 & 1 2 5 6 8 & 1 2 5 7 8\\
2 3 4 5 6 & 1 2 4 6 8 & 1 2 4 7 8 \\
1 3 4 5 6 & 2 3 5 6 8 & 2 3 5 7 8\\
1 2 4 5 7 & 2 3 4 6 8 & 2 3 4 7 8\\
2 3 4 5 7 & 1 3 5 6 8 & 1 3 5 7 8\\
1 3 4 5 7 & 1 3 4 6 8 & 1 3 4 7 8\\
\hline
\end{tabular}
}}$
\caption{Flip ($6 \rightarrow 12$)b partitions six pentatopes (left) into twelve (right) by inserting a point within the shared grouping of tetrahedra and reconnecting. Reversing the process is also valid via flip ($12 \rightarrow 6$)b. The flip is a 4D extension of the 3D, ($3\rightarrow 6$) flip. The table summarizes the connectivity transformation.}
\label{flip_6_12b}
\end{figure}

\begin{figure}[h!]
\centering
$\vcenter{\hbox{\includegraphics[width=4cm]{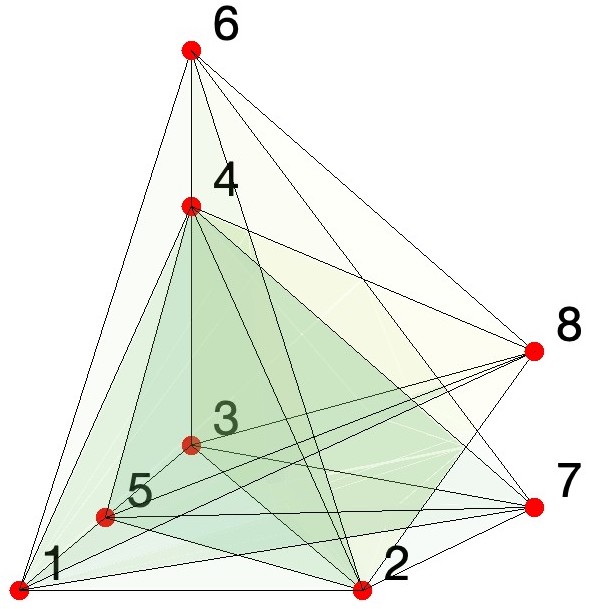}}}$
$\vcenter{\hbox{\includegraphics[width=4cm]{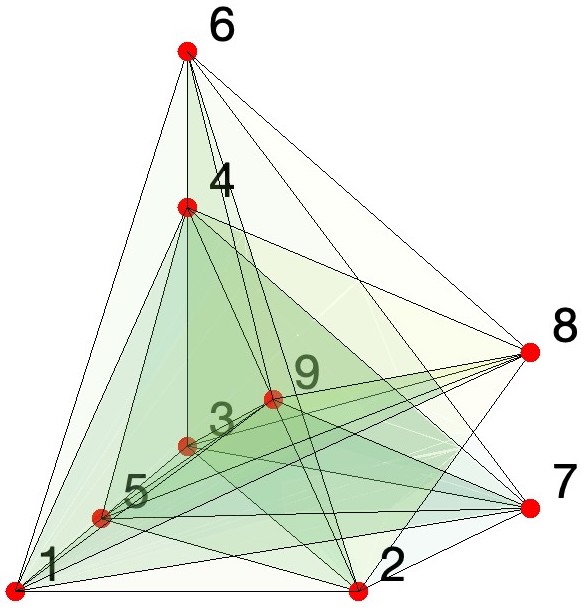}}}$
$\vcenter{\hbox{
\begin{tabular}{|c|c|c|}
\hline
Stage 1 & \multicolumn{2}{|c|}{Stage 2} \\
\hline
1 2 4 5 7 & 2 3 6 7 9 & 1 4 5 8 9 \\
2 3 4 5 7 & 1 4 5 7 9 & 1 2 5 8 9 \\
1 2 4 6 7 & 1 2 5 7 9 & 3 4 5 8 9\\
2 3 4 6 7 & 3 4 5 7 9 & 2 3 5 8 9 \\
1 2 4 5 8 & 2 3 5 7 9 & 1 4 6 8 9\\
2 3 4 5 8 & 1 4 6 7 9 & 1 2 6 8 9\\
1 2 4 6 8 & 1 2 6 7 9 & 3 4 6 8 9\\
2 3 4 6 8 & 3 4 6 7 9 & 2 3 6 8 9\\
\hline
\end{tabular}
}}$
\caption{Flip ($8 \rightarrow 16$) partitions eight pentatopes (left) into sixteen (right) by inserting a point on the shared edge and reconnecting. Reversing the process is also valid via flip ($16 \rightarrow 8$). The flip is a 4D extension of the 3D, ($4 \rightarrow 8$) flip. The table summarizes the connectivity transformation.}
\label{flip_8_16}
\end{figure}

\pagebreak
\clearpage

\section{Numerical Examples}
\label{sec;numerical_examples}

In this section, we assess the validity of our algorithms for point insertion and quality improvement, which were originally presented in sections~\ref{sec;point_insertion} and \ref{sec;quality_improvement}, respectively. These algorithms were implemented as an extension of the JENRE$^{\text{\textregistered}}$  Multiphysics Framework, which is United States government-owned software developed by the Naval Research Laboratory with other collaborating institutions. This software was used in earlier work for space-time finite element methods~\cite{anderson2023surface,Cor19_SCITECH}.

\subsection{Hypervolume Convergence Study}

In this test, our objective was to demonstrate the validity of our Delaunay point-insertion algorithm by generating multiple hypervolume meshes using both the isotropic and anisotropic approaches, and thereafter, assess the volumetric error associated with each mesh. Towards this end, we began by constructing a simple hypercylinder geometry with a radius of $R = 1$ and a length of $L = 4$. The hypercylinder geometry consisted of a 3-sphere which was extruded in the temporal direction in order to generate a space-time cylinder, (as shown in Figure~\ref{simple_hypercylinder_fig}). We note that the hypercylinder is convenient to work with because it is convex. Based on the techniques of~\cite{anderson2023surface}, we generated a family of tetrahedral hypersurface meshes which conformed to the hypercylinder geometry. The diameters of the tetrahedral elements in each mesh were controlled by the following parameters: $h_{\mathrm{sphere}}$, which prescribed the size of elements on the surface of each sphere; and $h_{\mathrm{time}}$, which prescribed the size of elements along the temporal axis of the hypercylinder. The family of meshes was generated by setting $h_{\mathrm{sphere}} = 1$ and $h_{\mathrm{time}} = 1$, and successively refining the meshes by reducing the parameters by factors of 1.5 or 2.0. The properties of each hypersurface mesh are summarized in the second column of Table~\ref{hypercylinder_hypervolume_table}.
\begin{figure}[h!]
\centering
\includegraphics[width=7cm]{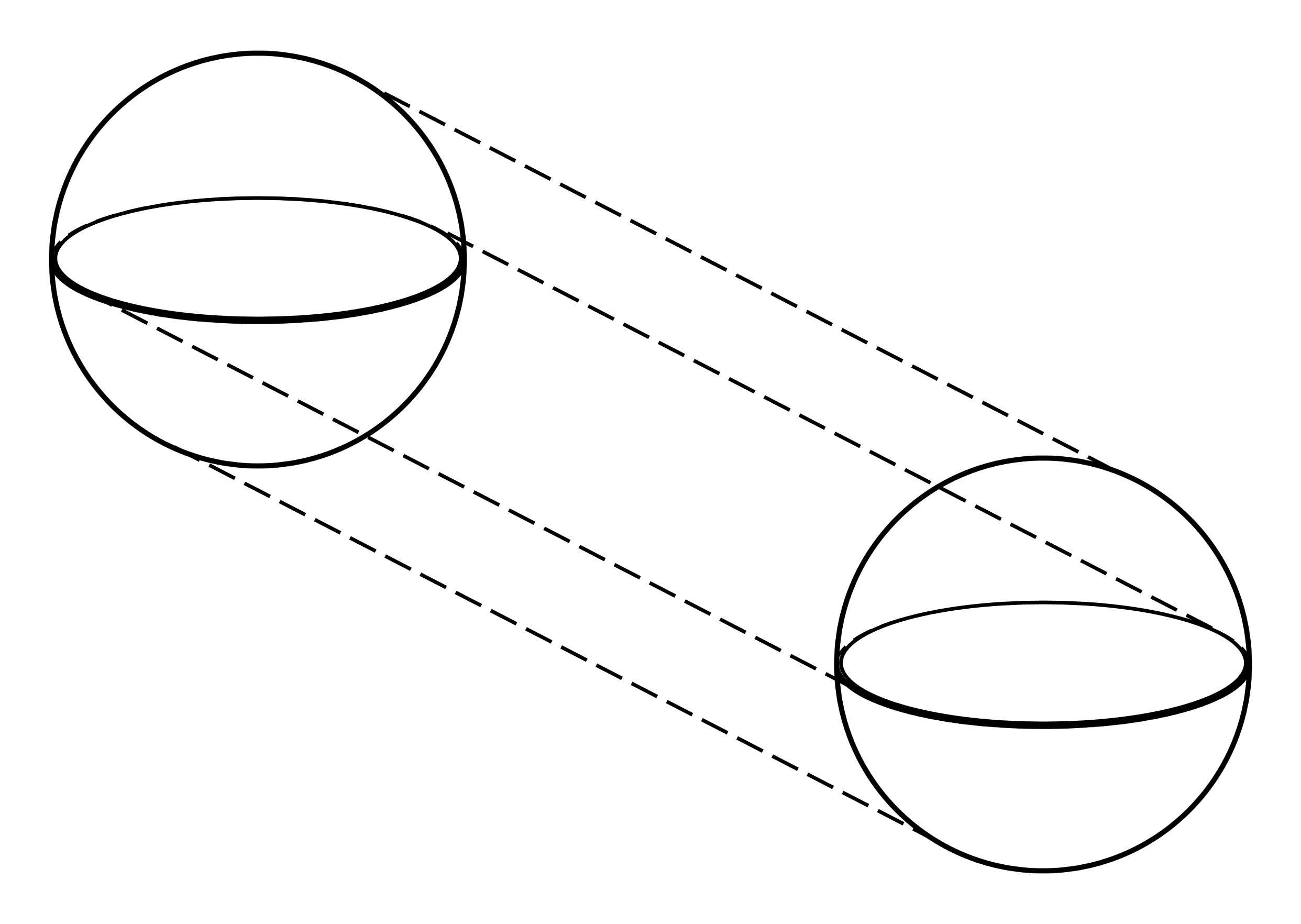}
\caption{An illustration of the hypercylinder geometry. Note: this drawing is \emph{not} to scale.}
\label{simple_hypercylinder_fig}
\end{figure}


After the hypersurface meshes were generated, we proceeded to construct a family of hypervolume meshes which conformed (at least partially) to the hypersurface meshes. We note that boundary facets of the hypervolume meshes were not always guaranteed to coincide with boundary facets of the original hypersurface meshes, as we did not implement a boundary-recovery procedure for the work in this paper. Nevertheless, each hypervolume mesh conformed to the discrete representation of the hypercylinder geometry. This fact was guaranteed due to the convexity of the hypercylinder. The hypervolume meshes were generated by applying the point-insertion process (from section~\ref{sec;point_insertion}) to each of the points in a given hypersurface mesh. In Table~\ref{hypercylinder_hypervolume_table}, we show results for the `isotropic' Delaunay algorithm with $M = \mathbb{I}$, and in Table~\ref{anisotropic_hypercylinder_hypervolume_table}, we show results for the anisotropic Delaunay algorithm with
\begin{align}
    M = \begin{bmatrix}
        1 & 0 & 0 & 0 \\
        0 & 1 & 0 & 0 \\
        0 & 0 & 1 & 0 \\
        0 & 0 & 0 & \wavespeed^2
    \end{bmatrix},
\end{align}
where
\begin{align}
    \wavespeed = \wavespeed_{0} + \frac{1}{\beta}\sqrt{\exp{(-(t-2)^{2})}},
\end{align}
and where $\wavespeed_0 = 1.0$ and $\beta = 0.1$.
\begin{table}[h!]
\begin{center}
\begin{tabular}{| c |r|p{17mm}|r|r| }
\hline
Mesh & Hypersurface Tets & Pentatopes (Isotropic Delaunay) & Vertices & Time \\
\hline
1& 638 & 1,627 & 145 & 1.0 \\
2& 7,505 & 23,197 & 1,581 & 10.5 \\
3& 17,596 & 55,402 & 3,681 & 32.6 \\
4& 40,563 & 114,077 & 8,456 & 503 \\
5& 101,894 & 305,204 & 21,272 & 212 \\
6& 269,623 & 804,103 & 56,154 & 650 \\
\hline
\end{tabular}
\caption{A summary of hypermesh properties for the isotropic Delaunay approach. In the table above, we show the number of hypersurface tetrahedral elements, number of hypervolume pentatope elements, number of vertices, and mesh generation times for a sequence of hypervolume meshes for the hypercylinder test case. The mesh generation times have been normalized relative to the time for generating the first mesh. Note that the implementation is not optimized, and therefore the computational times should be viewed with this in mind.} \label{hypercylinder_hypervolume_table}
\end{center}
\end{table}
\begin{table}[h!]
\begin{center}
\begin{tabular}{| c |p{20mm}|r|r| }
\hline
Mesh & Pentatopes (Anisotropic Delaunay) & Vertices & Time \\
\hline
1 & 1,716   & 145    & 1.0 \\
2 & 22,215  & 1,581  & 28.2 \\
3 & 57,472  & 3,681  & 24.1 \\
4 & 126,393 & 8,456  & 1,103 \\
5 & 330,017 & 21,272 & 150 \\
6 & 889,441 & 56,154 & 499 \\
\hline
\end{tabular}
\caption{A summary of hypermesh properties for the anisotropic Delaunay approach. In the table above, we show the number of hypersurface tetrahedral elements, number of hypervolume pentatope elements, number of vertices, and mesh generation times for a sequence of anisotropic hypervolume meshes for the hypercylinder test case. The mesh generation times have been normalized relative to the time for generating the first mesh. Note that the implementation is not optimized, and therefore the computational times should be viewed with this in mind.} \label{anisotropic_hypercylinder_hypervolume_table}
\end{center}
\end{table}


Figure~\ref{hypercylinder_meshes} shows a projected view of anisotropic Delaunay hypervolume meshes 1 and 2, and Figure~\ref{hypercylinder_meshes_expanded_combined} shows a projected and exploded (expanded) view of meshes 1--3. In addition, Figure~\ref{hypercylinder_meshes_expanded_combined} shows a projected and exploded view of  \emph{isotropic} Delaunay hypervolume meshes 1--3. We note that the un-exploded isotropic and anisotropic meshes look very similar, and therefore, we only show the un-exploded views of the anisotropic Delaunay meshes. In each case, the meshes are visualized by plotting the five tetrahedral facets associated with each pentatope. The plotting is performed using Matlab's tetrahedral meshing function \emph{tetramesh}. The coordinates of the tetrahedral facets are inherently four-dimensional, and therefore, we have created the following projection in order to convert the coordinates into a three-dimensional format
\begin{align*}
    \vec{v}_{\mathrm{proj}} = \left(v_1, v_2, v_3\right) + v_4 \vec{e},
\end{align*}
where
\begin{align*}
    \vec{e} = \frac{1}{\sqrt{3}} \left(1, 1, 1\right).
\end{align*}
This projection creates an image in which the spatial coordinate axes are orthogonal to one another, and the temporal axis is non-orthogonal to the spatial axes. Figure~\ref{tesseract_project} shows an image of a tesseract which is visualized using this projection. The tesseract appears to be constructed from two cubes, connected by diagonal lines. The edges of the cubes correspond to the spatial $x$, $y$, and $z$ directions, and the diagonal edges are aligned with the temporal direction. 

\begin{figure}[h!]
\centering
\includegraphics[width=6cm]{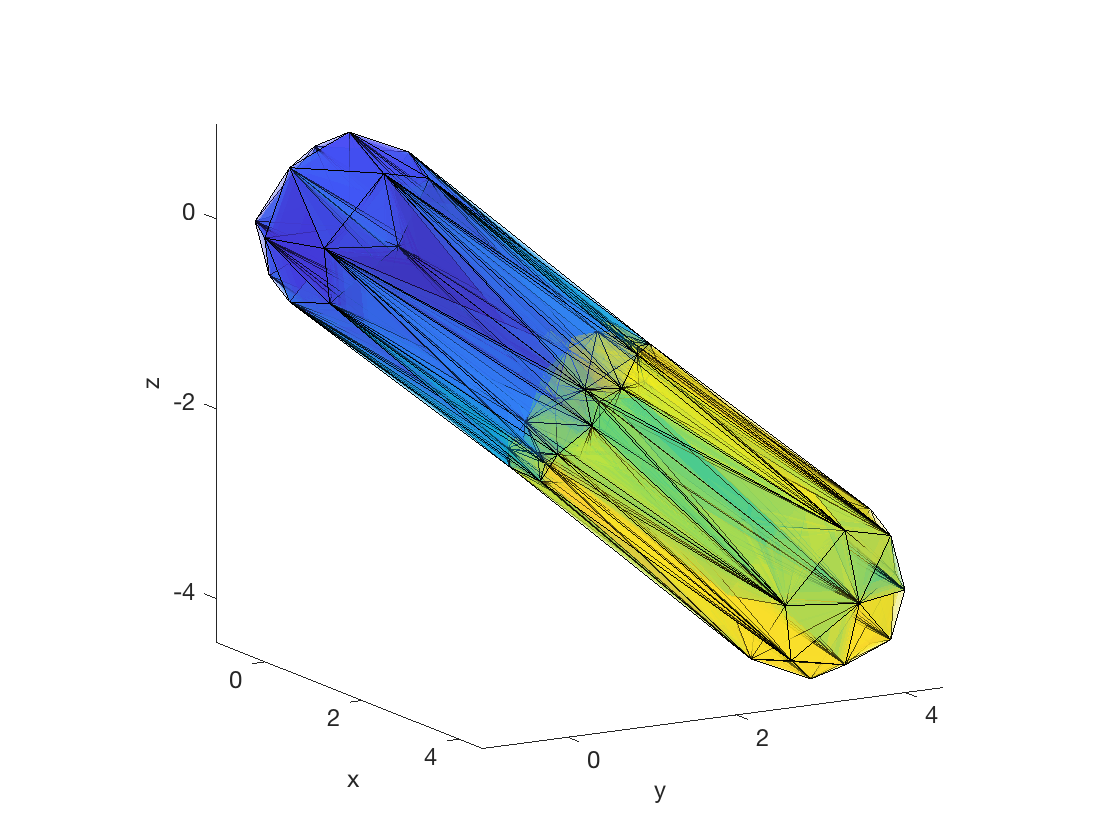}
\includegraphics[width=6cm]{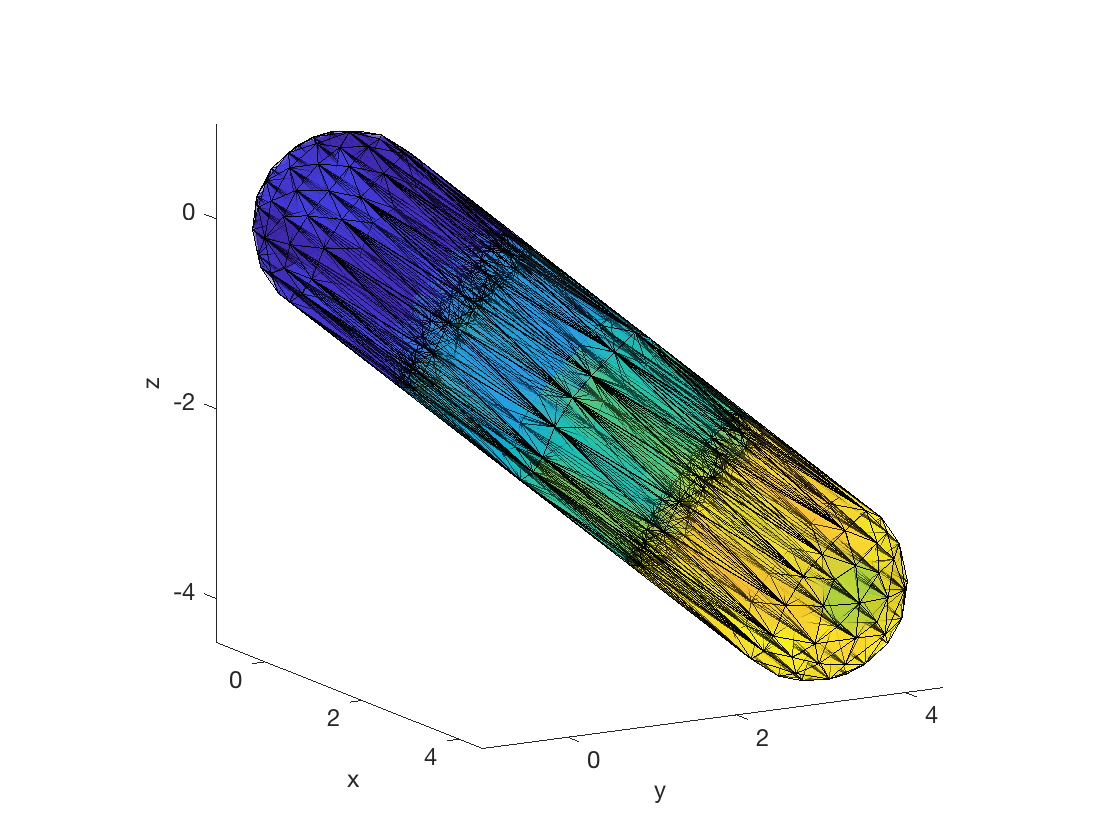}
\caption{Illustrations of hypervolume meshes 1 and 2, for the anisotropic Delaunay approach, left and right, respectively.}
\label{hypercylinder_meshes}
\end{figure}

\begin{figure}[h!]
\centering
\includegraphics[width=6cm]{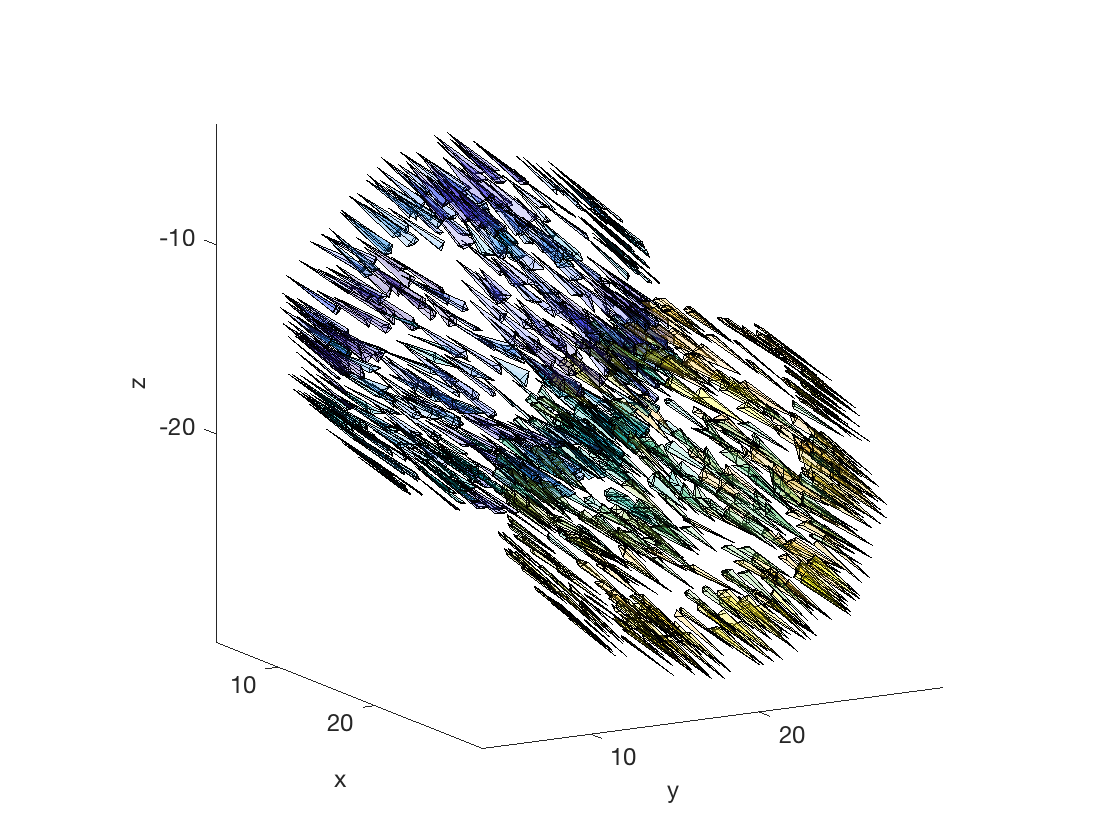}
\includegraphics[width=6cm]{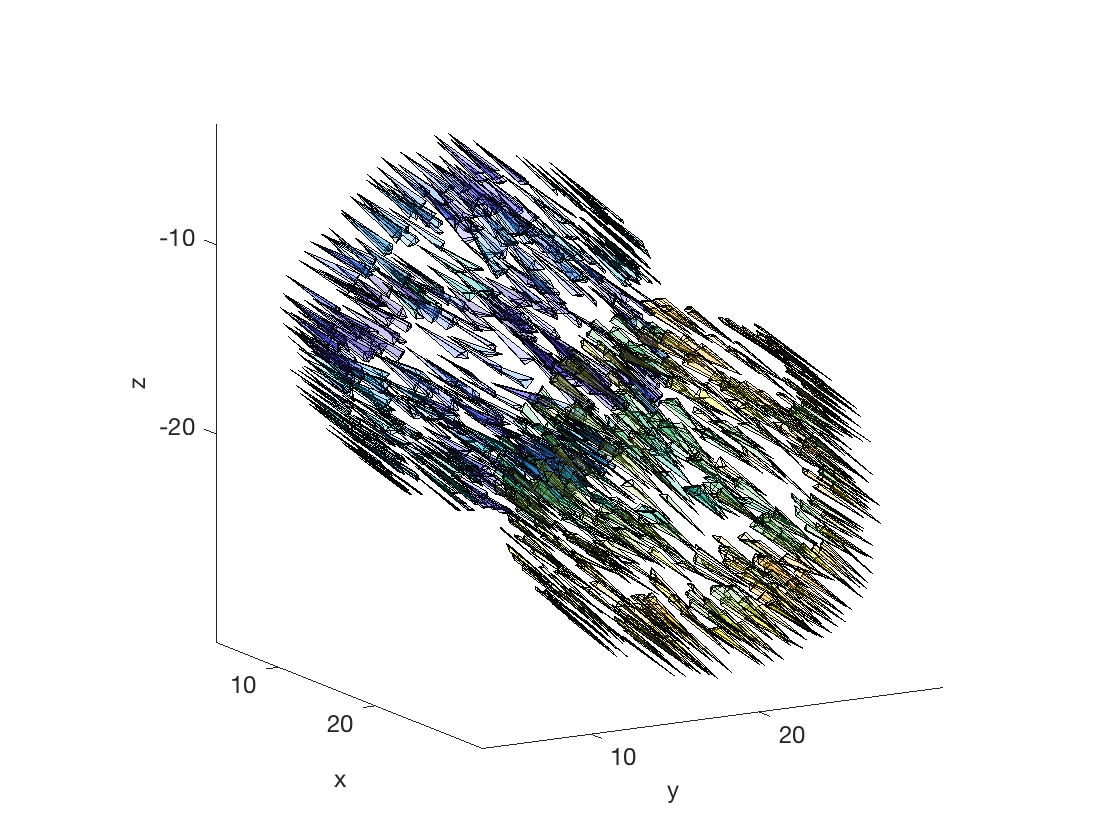}
\includegraphics[width=6cm]{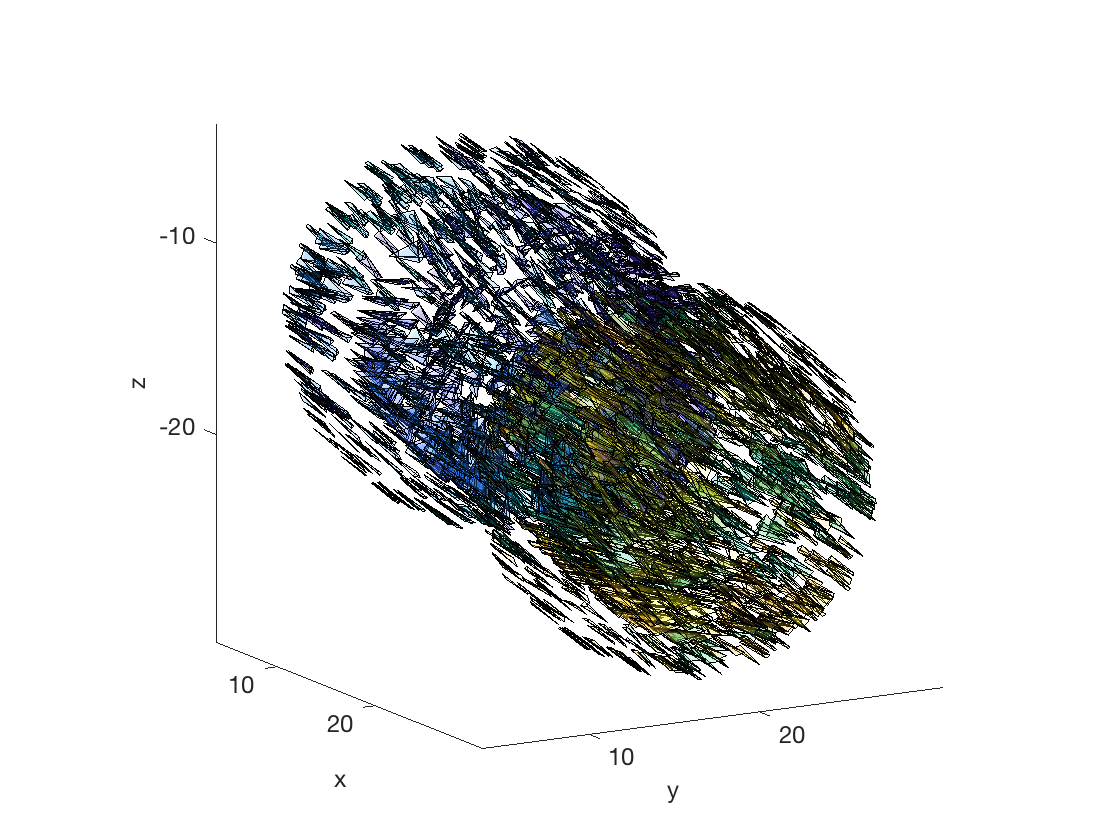}
\includegraphics[width=6cm]{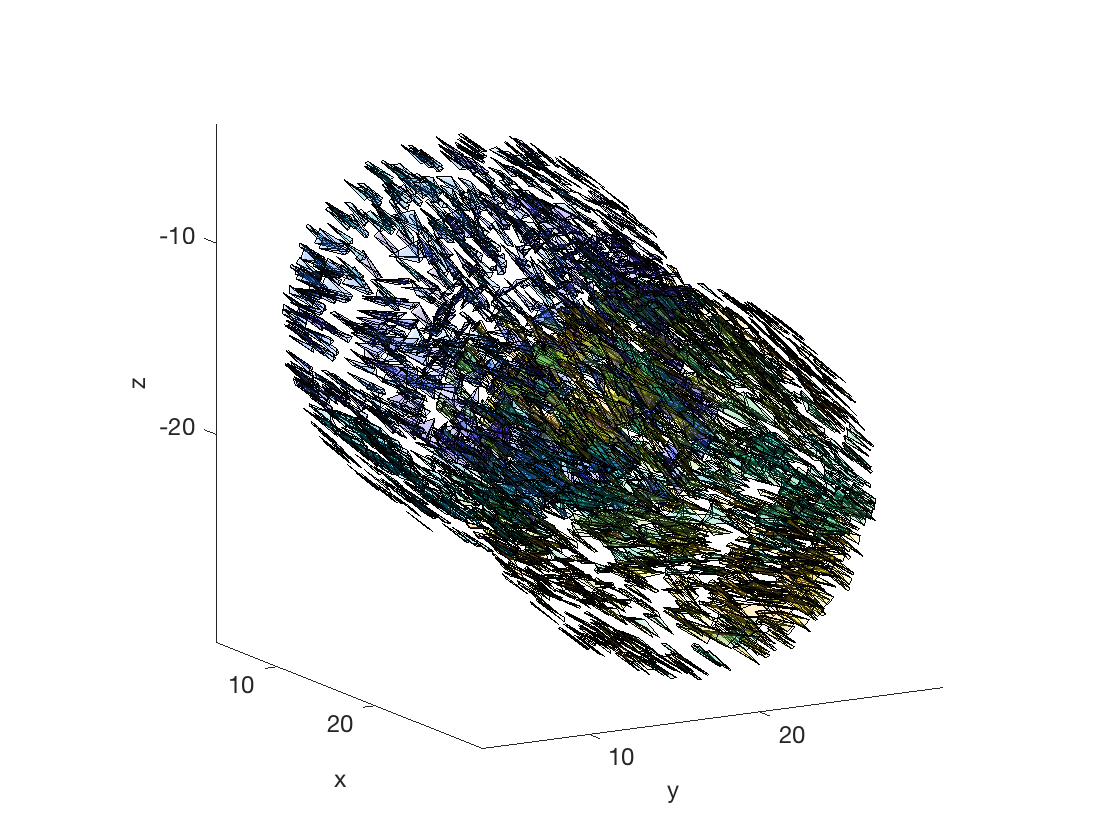}
\includegraphics[width=6cm]{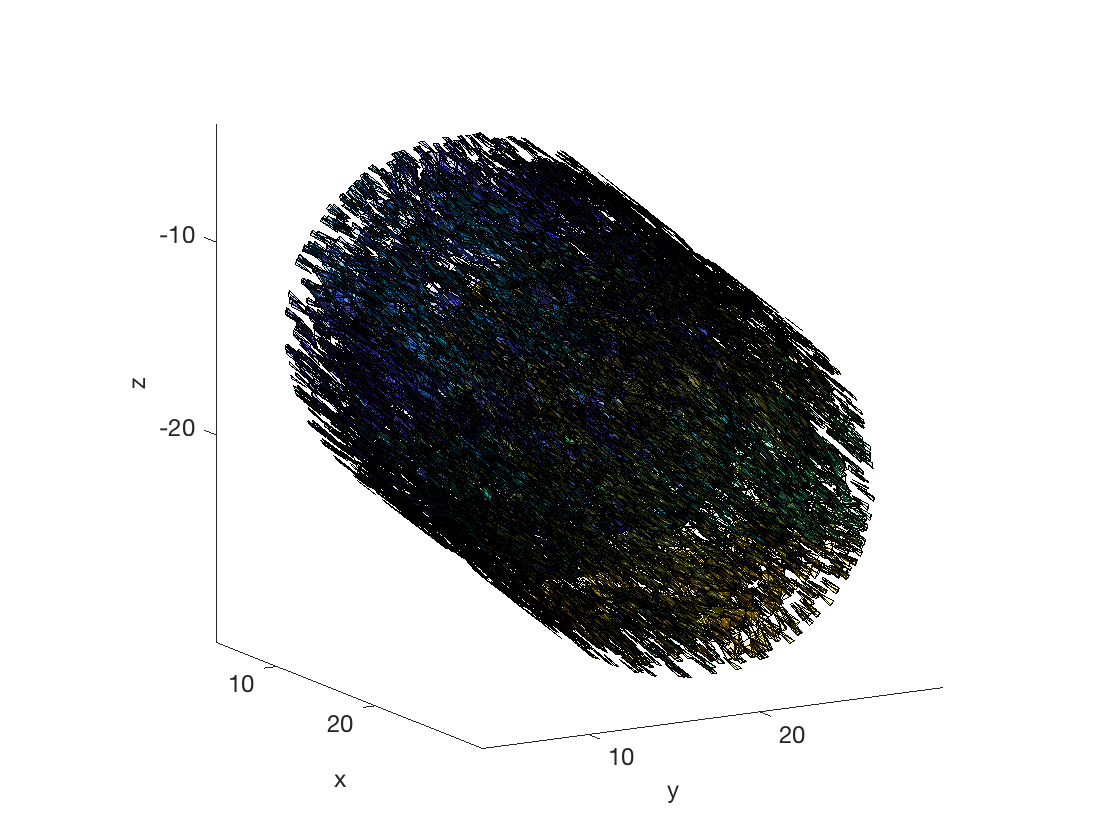}
\includegraphics[width=6cm]{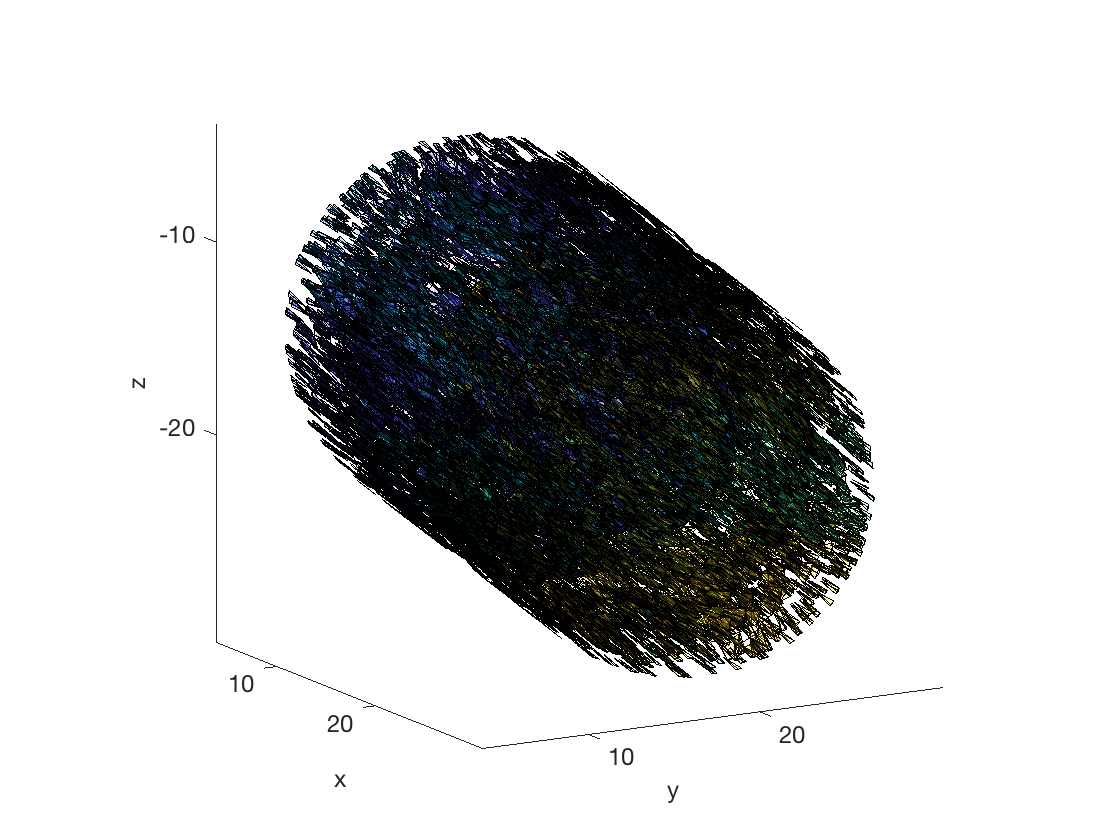}
\caption{Illustrations of `exploded' hypervolume meshes 1, 2, and 3, for the isotropic and anisotropic Delaunay approaches, columns 1 and 2, respectively.  These illustrations have been zoomed-in to show the time range $t = 1.5$ to $t = 2.5$.}
\label{hypercylinder_meshes_expanded_combined}
\end{figure}


\begin{figure}[h!]
\centering
\includegraphics[width=7cm]{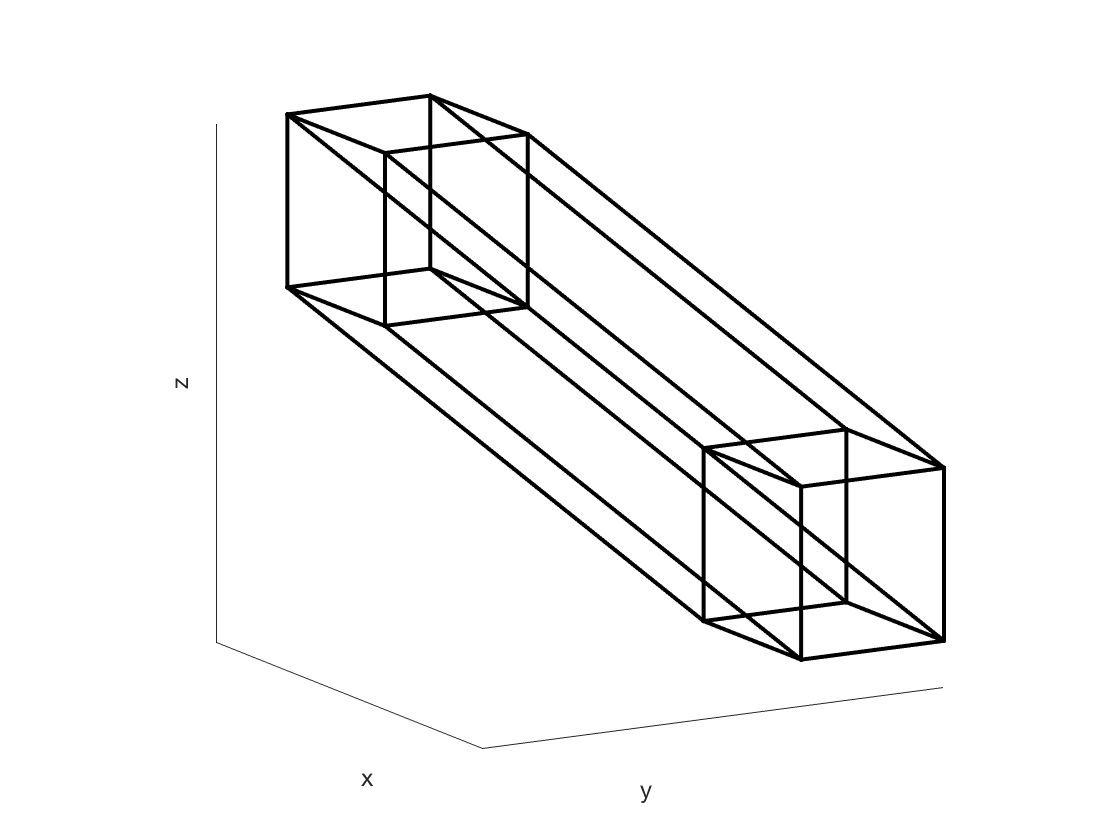}
\caption{Projection of a space-time tesseract.}
\label{tesseract_project}
\end{figure}


The `hypervolume error' associated with each hypervolume mesh was calculated as follows
\begin{align*}
    HV_{\mathrm{error}} = \left| HV_{\mathrm{exact}} - HV_{\mathrm{approx}} \right|,
\end{align*}
where $HV_{\mathrm{exact}}$ is the exact hypercylinder hypervolume given by
\begin{align*}
    HV_{\mathrm{exact}} = \frac{4}{3} \pi R^3 L,
\end{align*}
and $HV_{\mathrm{approx}}$ was the summation of the hypervolumes of the pentatopes in each mesh. Figures~\ref{hypercylinder_error_fig} and~\ref{aniso_hypercylinder_error_fig} show plots of the hypervolume error versus the characteristic mesh spacing for the isotropic and anisotropic approaches. We estimated the characteristic mesh spacing by raising the total number of pentatope elements in each mesh to the -1/3 power. Based on the figures, we can clearly see that the hypervolume error converges at a rate of 2nd order. This is what we would expect when using straight-sided elements to approximate a curved surface. 
\begin{figure}[h!]
\centering
\includegraphics[width=8cm]{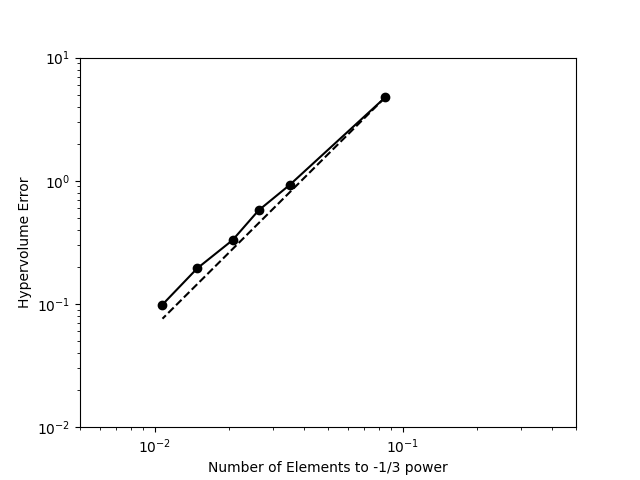}
\caption{The plot above shows the hypervolume error, $HV_{\mathrm{error}}$, of isotropic Delaunay meshes for the hypercylinder test case. The errors are plotted against the characteristic mesh spacing for a sequence of increasingly refined hypervolume meshes. In addition, a dashed line associated with 2nd-order convergence is plotted for reference.}
\label{hypercylinder_error_fig}
\end{figure}

\begin{figure}[h!]
\centering
\includegraphics[width=8cm]{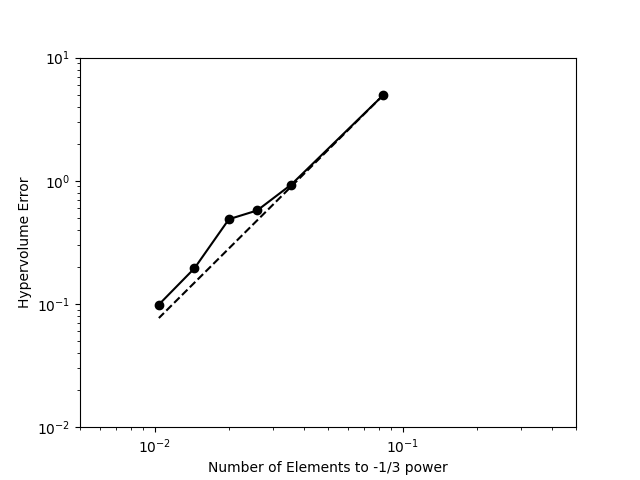}
\caption{The plot above shows the hypervolume error, $HV_{\mathrm{error}}$, of anisotropic Delaunay meshes for the hypercylinder test case. The errors are plotted against the characteristic mesh spacing for a sequence of increasingly refined hypervolume meshes. In addition, a dashed line associated with 2nd-order convergence is plotted for reference.}
\label{aniso_hypercylinder_error_fig}
\end{figure}

\subsection{Quality Improvement Study}

For this numerical experiment, we generated sets of points distributed randomly within a tesseract domain. Thereafter, a mesh for these points was constructed using the isotropic version of the Delaunay point-insertion algorithm of section~\ref{sec;point_insertion}. 
Once a valid mesh containing all of the points was generated within the bounding tesseract, all pentatopes associated with the points from the bounding tesseract were removed, leaving a valid mesh consisting of only the randomly generated points.

We then employed the following algorithm in order to improve the mesh quality. We acknowledge that this algorithm is not necessarily guaranteed to improve the quality of the elements in a given mesh. However, the purpose of this test was to simply establish evidence for the algorithm's effectiveness in practice.

First, we computed the quality of each element using all three of the quality heuristics, $\eta^{(1)}$, $\eta^{(2)}$, and $\eta^{(3)}$, in conjunction with $M = \mathbb{I}$. However, detailed tests were only carried out with $\eta^{(1)}$ for the sake of brevity. Next, the worst quality pentatope was identified and dubbed the \emph{starter pentatope}.

After the starter pentatope had been identified, all possible flips were determined based on its relationship with the set of all pentatopes that share at least one of its vertices. Of the valid flips, we identified those that increased the quality of the lowest quality element, and produced valid reconnections, (i.e. flips that preserve the hypervolume and do not introduce zero hypervolume elements). From this set of flips, the flip with the largest increase in quality of the lowest quality element was executed. It is worth noting that failing to find valid flips for a particular starter pentatope is a common occurrence. Any pentatopes resulting from a flip were \emph{frozen} and no longer allowed to change. In addition, pentatopes chosen as the starter pentatope were logged in order to prevent them from being chosen a second time. Our algorithm terminated when all pentatopes were either frozen or already chosen as the starter pentatope in a previous iteration. 

The tests were performed with small meshes of 50, 100, 150, 200, 250, and 300 points. 
%
%
%
The quality improvement algorithm was successfully employed on each mesh.
In order to investigate the degree of quality improvement, the average quality of the worst one, five, ten, and twenty percent of the total number of elements was calculated before and after executing the flips, using the heuristic $\eta^{(1)}$. The results are shown in Tables~\ref{qi_results_1} and~\ref{qi_results_2}. Figure~\ref{flip_history_one} illustrates how often each of the flips are used when employing the quality improvement algorithm on a cloud of 300 random points. It can be seen that the quality heuristic favors flips that reduce or do not change the total number of elements. Furthermore, it does not use flips that insert additional points.

The initial and final hypervolumes are consistent (up to quadruple precision), before and after quality improvement operations have been attempted, which demonstrates the validity of the bistellar flips. In addition, the flips produce consistent improvement for low quality elements in each mesh. More complex algorithms associated with the bistellar flips may produce better results; however, the purpose of this numerical experiment is to simply demonstrate the validity of these flips and a potential use case.

Finally, we repeated a subset of these tests for heuristics $\eta^{(2)}$ and $\eta^{(3)}$ and obtained similar results---as shown in Tables~\ref{qi_results_eta_2} and \ref{qi_results_eta_3}, and Figures~\ref{flip_history_two} and \ref{flip_history_three}. It appears that these heuristics do incorporate flips that add additional points to the mesh.
%
\begin{table}[h!]
\begin{center}
\begin{tabular}{| c| c| c| c| c| c| }
\hline
\multicolumn{2}{|c|}{} & \multicolumn{2}{|c|}{No.~Pentatopes} \\ 
\hline
No.~Points & No. Flips & Initial & Final \\
\hline
50 & 29 & 492 & 454 \\
100 & 82 & 1,421 & 1,325  \\
150 & 153 & 2,443 & 2,267  \\
200 & 241 & 3,583 & 3,317  \\
250 & 320 & 4,767 & 4,407  \\
300 & 416 & 6,010 & 5,524  \\
\hline
\end{tabular}
\caption{Information for $\eta^{(1)}$ quality improvement cases on randomized meshes.} \label{qi_results_1}
\end{center}
\end{table}
\begin{table}[h!]
\begin{center}
\begin{tabular}{| c| c| c| c| c| c| c| c| c| r| }
\hline
& \multicolumn{2}{|c|}{1\% AMQ} & \multicolumn{2}{|c|}{5\% AMQ} & \multicolumn{2}{|c|}{10\% AMQ} & \multicolumn{2}{|c|}{20\% AMQ} & \\
\hline
No. Points & Initial & Final & Initial & Final & Initial & Final & Initial & Final & Time \\
\hline
50  & 0.1156 & 0.1569 & 0.1609 & 0.1917 & 0.1944 & 0.2272 & 0.2485 & 0.2851 & 1.0 \\
100 & 0.1521 & 0.1638 & 0.2070 & 0.2291 & 0.2440 & 0.2673 & 0.2937 & 0.3180 & 8.1 \\
150 & 0.1169 & 0.1312 & 0.1768 & 0.1932 & 0.2121 & 0.2296 & 0.2621 & 0.2828 & 21.6 \\
200 & 0.1061 & 0.1254 & 0.1658 & 0.1867 & 0.2015 & 0.2239 & 0.2497 & 0.2747 & 58.7 \\
250 & 0.1072 & 0.1253 & 0.1645 & 0.1859 & 0.2031 & 0.2235 & 0.2517 & 0.2749 & 122.3 \\
300 & 0.1023 & 0.1174 & 0.1595 & 0.1795 & 0.1982 & 0.2199 & 0.2464 & 0.2716 & 211.2 \\
\hline
\end{tabular}
\caption{Quality improvement results for the average minimum quality (AMQ) for different percentages of the worst quality elements for randomized meshes. The execution times have been normalized in reference to the time required for the algorithm to process 492 initial elements. Note that the implementation is not optimized, and therefore the computational times should be viewed with this in mind. In particular, the implementation is based on an unoptimized region-to-region data structure.} \label{qi_results_2}
\end{center}
\end{table}
\begin{table}[h!]
\begin{center}
\begin{tabular}{| c| c| c| c| c| c| c| c|}
\hline
\multicolumn{2}{|c|}{} & \multicolumn{2}{|c|}{No.~Pentatopes} & \multicolumn{2}{|c|}{20\% AMG} \\ 
\hline
No.~Points & No. Flips & Initial & Final & Initial & Final \\
\hline
50 & 47 & 492 & 524 & 0.6781 & 0.6826 \\
100 & 125 & 1421 & 1545 & 0.6879 & 0.6899 \\
150 & 237 & 2443 & 2641 & 0.6787 & 0.6854 \\
200 & 350 & 3583 & 3841 & 0.6745 & 0.6852 \\
250 & 477 & 4767 & 5039 & 0.6736 & 0.6829 \\
300 & 608 & 6010 & 6296 & 0.6732 & 0.6796 \\
\hline
\end{tabular}
\caption{Information for $\eta^{(2)}$ quality improvement cases on randomized meshes.} \label{qi_results_eta_2}
\end{center}
\end{table}
\begin{table}[h!]
\begin{center}
\begin{tabular}{| c| c| c| c| c| c| c| c|}
\hline
\multicolumn{2}{|c|}{} & \multicolumn{2}{|c|}{No.~Pentatopes} & \multicolumn{2}{|c|}{20\% AMG} \\ 
\hline
No.~Points & No. Flips & Initial & Final & Initial & Final \\
\hline
50 & 34 & 492 & 452 & 0.1779 & 0.2027 \\
100 & 86 & 1421 & 1333 & 0.2123 & 0.2269 \\
150 & 156 & 2443 & 2287 & 0.1873 & 0.2020 \\
200 & 243 & 3583 & 3345 & 0.1785 & 0.1943 \\
250 & 330 & 4767 & 4417 & 0.1792 & 0.1954 \\
300 & 419 & 6010 & 5568 & 0.1752 & 0.1922 \\
\hline
\end{tabular}
\caption{Information for $\eta^{(3)}$ quality improvement cases on randomized meshes.} \label{qi_results_eta_3}
\end{center}
\end{table}
\begin{figure}[h!]
    \centering
\pgfplotstableread[row sep=\\,col sep=&]{
    interval & noFlip \\
    4--2     & 121    \\
    3--3     & 149    \\
    6--6     & 5      \\
    4--6     & 5      \\
    6--4     & 127    \\
    8--8     & 9      \\
    }\mydata
    $\vcenter{\hbox{
    \begin{tikzpicture}
        \begin{axis}[
                ybar,
                symbolic x coords={4--2,3--3,6--6,4--6,6--4,8--8},
                xtick=data,
                nodes near coords,
                ylabel={No. Executions},
            ]
            \addplot table[x=interval,y=noFlip]{\mydata};
        \end{axis}
    \end{tikzpicture}
    }}$
    \caption{Number of flips executed for a cloud of 300 random points for the $\eta^{(1)}$ quality heuristic.}
    \label{flip_history_one}
\end{figure}
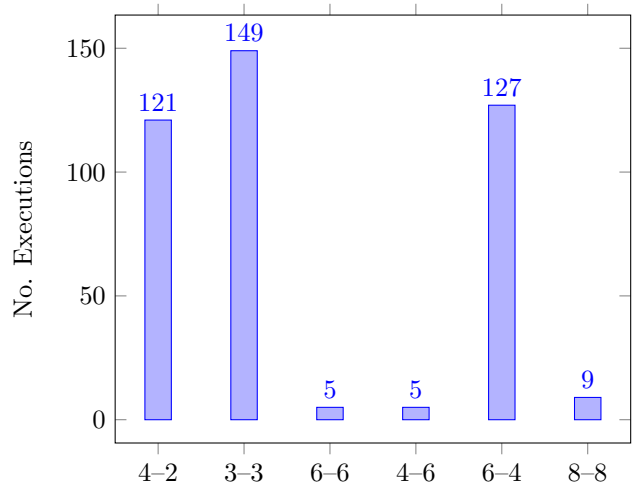
\begin{figure}[h!]
    \centering
\pgfplotstableread[row sep=\\,col sep=&]{
    interval & noFlip \\
    2--4     & 17     \\
    4--2     & 91     \\
    3--3     & 241    \\
    4--8     & 13     \\
    3--9     & 14     \\
    6--6     & 12     \\
    2--8     & 2      \\
    6--12    & 9      \\
    12--6    & 1      \\
    4--6     & 108    \\
    6--4     & 73     \\
    4--12    & 9      \\
    12--4    & 1      \\
    8--8     & 13     \\
    8--16    & 13     \\
    }\mydata
    $\vcenter{\hbox{
    \begin{tikzpicture}
        \begin{axis}[
                ybar,
                symbolic x coords={2--4,4--2,3--3,4--8,3--9,6--6,2--8,6--12,12--6,4--6,6--4,4--12,12--4,8--8,8--16},
                xtick=data,
                nodes near coords,
                ylabel={No. Executions},
                width=15cm, height=8cm
            ]
            \addplot table[x=interval,y=noFlip]{\mydata};
        \end{axis}
    \end{tikzpicture}
    }}$
    \caption{Number of flips executed for a cloud of 300 random points for the $\eta^{(2)}$ quality heuristic.}
    \label{flip_history_two}
\end{figure}
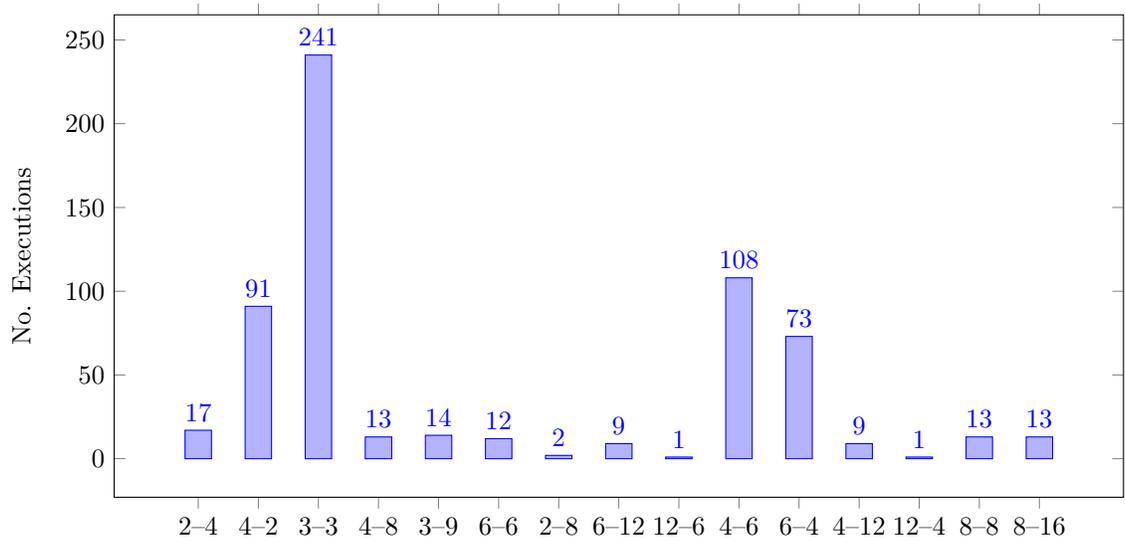
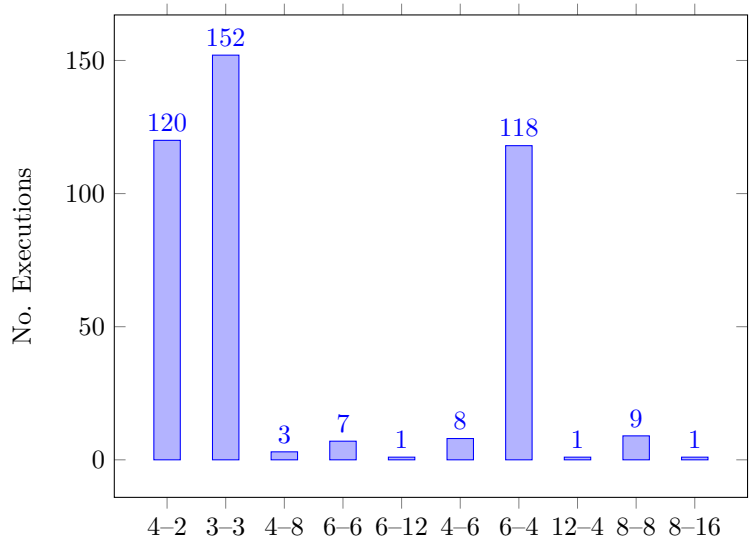
\begin{figure}[h!]
    \centering
\pgfplotstableread[row sep=\\,col sep=&]{
    interval & noFlip \\
    4--2     & 120    \\
    3--3     & 152    \\
    4--8     & 3      \\
    6--6     & 7      \\
    6--12    & 1      \\
    4--6     & 8      \\
    6--4     & 118    \\
    12--4    & 1      \\
    8--8     & 9      \\
    8--16    & 1      \\
    }\mydata
    $\vcenter{\hbox{
    \begin{tikzpicture}
        \begin{axis}[
                ybar,
                symbolic x coords={4--2,3--3,4--8,6--6,6--12,4--6,6--4,12--4,8--8,8--16},
                xtick=data,
                nodes near coords,
                ylabel={No. Executions},
                width=10cm, height=8cm
            ]
            \addplot table[x=interval,y=noFlip]{\mydata};
        \end{axis}
    \end{tikzpicture}
    }}$
    \caption{Number of flips executed for a cloud of 300 random points for the $\eta^{(3)}$ quality heuristic.}
    \label{flip_history_three}
\end{figure}

\pagebreak
\clearpage

\section{Conclusion}
\label{sec;conclusion}

This work establishes a clear link between anisotropic Delaunay meshing and space-time applications. In addition, it presents explicit implementation details for the anisotropic approach, including a robust method for 4D point insertion along with explicit descriptions of metric-weighted orientation and in-hypersphere geometric predicates. Furthermore, two new tesseract subdivisions, a new set of quality heuristics for pentatope elements, and new 4D bistellar flips for groupings (collections) of multiple pentatope elements have been presented. Finally, our algorithms were verified using a pair of numerical experiments: i) the point-insertion algorithm was used to generate a family of hypervolume meshes for a hypercylinder geometry, and the hypervolume errors of these meshes were shown to converge at the expected rate; ii) an algorithm for quality improvement was developed using the bistellar flips, and this algorithm was shown to consistently raise the average quality of the worst elements in a series of randomized meshes.  

In future work, 4D boundary recovery will be implemented in order to generate high-quality, \emph{constrained}, anisotropic Delaunay boundary meshes. This will be an important endeavor, especially for hypervolume mesh generation on non-convex domains. The work carried out in this paper, e.g.~developing a robust point-insertion algorithm and developing methods for mesh quality improvement, are important steps towards achieving the goal of fully-automatic and robust constrained hypervolume mesh generation in 4D.

\section*{Software Availability}

These algorithms were implemented as an extension of the JENRE$^{\text{\textregistered}}$  Multiphysics Framework~\cite{Cor19_SCITECH}, which is United States government-owned software developed by the Naval Research Laboratory with other collaborating institutions. This software is not available for public use or dissemination.

\section*{Declaration of Competing Interests}

The authors declare that they have no known competing financial interests or personal relationships that could have appeared to influence the work reported in this paper.

\section*{Funding}

This research received funding from the United States Naval Research Laboratory (NRL) under grant number N00173-22-2-C008. In turn, the NRL grant itself was funded by Steven Martens, Program Officer for the Power, Propulsion and Thermal Management Program, Code 35, in the United States Office of Naval Research.

\pagebreak
\clearpage

\appendix

\section{Minimizing Roughness and Anisotropic Delaunay Meshes} \label{appendix_rough}

In this section, we will prove that an anisotropic Delaunay triangulation has minimal roughness (and is therefore optimal) for the space-time problem which was introduced in section~\ref{optimal_sub_sec}.

\subsection{Set-up}

For theoretical purposes, it is useful to generate anisotropic Delaunay meshes through a simple, iterative process. With this in mind, we consider a variation of the well-known Local Optimization Procedure (LOP), originally developed by Lawson~\cite{lawson1977software}. Suppose that we have an initial, non-degenerate triangulation of a set of points. This triangulation does not need to be Delaunay. We can then loop over the edges in the triangulation, and carry out the following procedure:

\begin{enumerate}
    \item Determine if the edge is shared by a pair of triangles.  If so, determine whether or not these triangles form a strictly convex quadrilateral. 
    \item Identify the midpoint $v$ of the edge. 
    \item For the triangles in step 1, check to see if the two opposite vertices are contained within their circumcircles. This check should be performed using a metric-weighted predicate, based on the metric defined at $v$
    \begin{align*}
        M_v = M(x_v, t_v) = \begin{bmatrix}
            1 & 0 \\
            0 & \wavespeed_{v}^{2}
        \end{bmatrix},
    \end{align*}
    where $\wavespeed_v = \wavespeed(x_v,t_v)$.
    \item If the circumcircle criterion is violated, flip the edge inside of the quadrilateral.
\end{enumerate}
This algorithm will produce an anisotropic Delaunay mesh, whose anisotropy is controlled by the values of the metric field evaluated at the mesh points.

\subsection{Simplified Analysis}

In what follows, we take inspiration from the approach of Rippa~\cite{rippa1990minimal} and Powar~\cite{powar1992minimal}, and perform our analysis on a pair of triangles. Rippa and Powar proved optimality of an isotropic Delaunay triangulation by proving optimality of each pair of triangles resulting from the flip operation in step 4. In our case, we can consider two triangles which form a strictly convex quadrilateral, prove optimality of these triangles in conjunction with an \emph{anisotropic} circumcircle criterion, and then extend this result to the entire mesh via LOP. Towards this end, we introduce a generic set of vertices $u_1 = [x_1, t_1]^{T}, u_2 = [x_2, t_2]^{T}, u_3 = [x_3, t_3]^{T}$, and $u_4 = [x_4, t_4]^{T}$ which form a strictly convex quadrilateral, as shown in Figure~\ref{triangle_map_fig}.
\begin{figure}[h!]
\centering
\includegraphics[width=11cm]{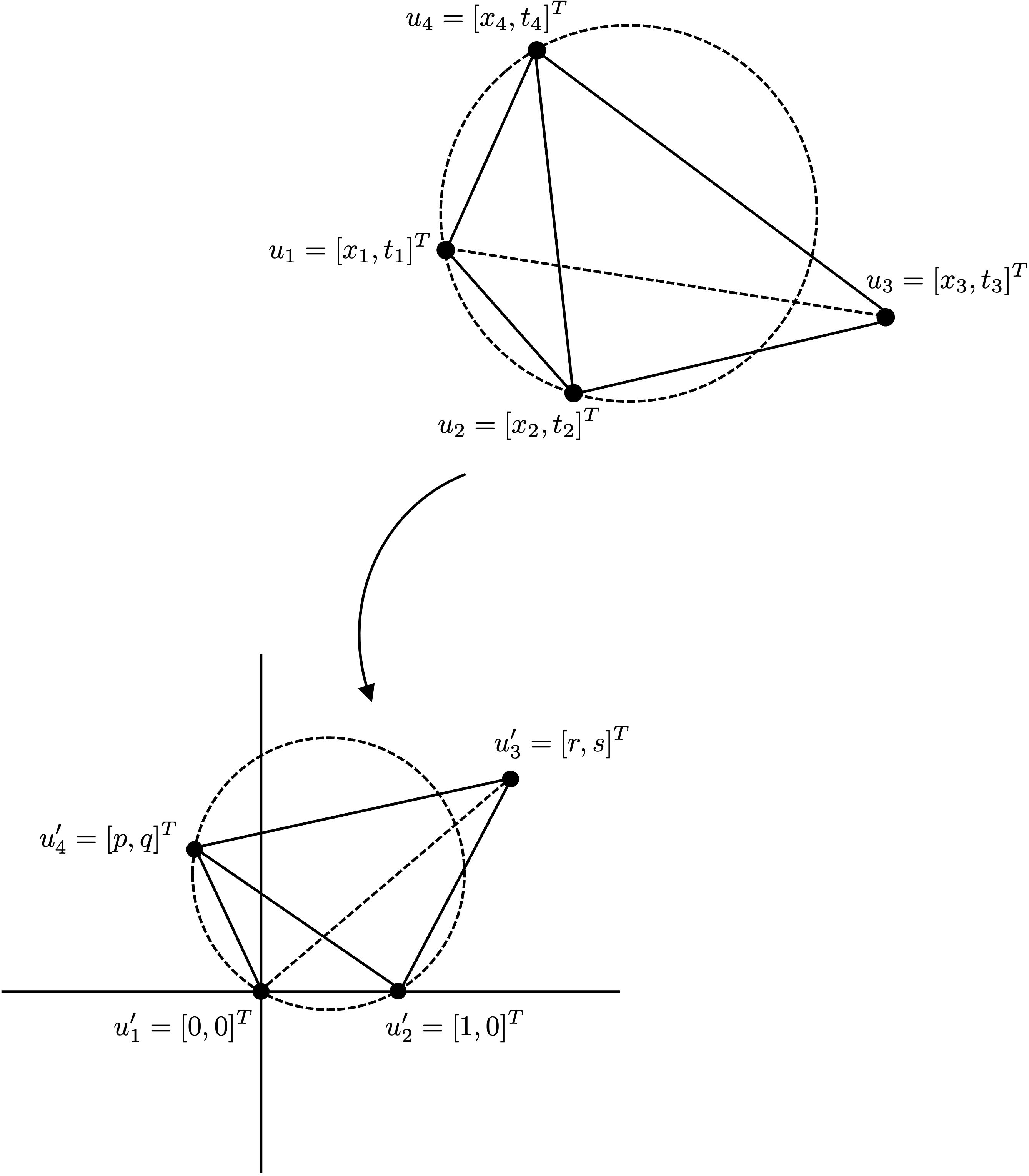}
\caption{A pair of triangles that form a strictly convex quadrilateral (top) and the mapped version of these triangles (bottom). This image is adapted from Figure 1 of~\cite{powar1992minimal}.}
\label{triangle_map_fig}
\end{figure}

These vertices can always be transformed by an affine transformation, such that we obtain mapped points $u_{1}' = [0,0]^{T}, u_{2}' = [1,0]^{T}, u_{3}' = [r,s]^{T}$ and $u_{4}' = [p,q]^{T}$. The necessary mapping is defined as follows
\begin{align*}
    &u_{i}' = Au_{i} + b, \qquad i = 1,2,3,4 \\[1.0ex]
    &A = \frac{1}{\left\| u_{2} - u_{1} \right\|_{M_v}} \begin{bmatrix}
        \cos \theta & -\sin\theta \\[1.0ex]
        \sin \theta & \cos \theta
    \end{bmatrix}, \qquad
    b = -u_1, \\[1.0ex]
    & \left\| u_{2} - u_{1} \right\|_{M_v} = \sqrt{(u_{2}-u_{1})^{T} M_{v} (u_{2}-u_{1})},
\end{align*}
where $\theta$ is the counter-clockwise angle required to rotate the edge $u_2u_1$ to align with the horizontal. The locations of the points $u_{1}',u_{2}',u_{3}'$, and $u_{4}'$ are shown in Figure~\ref{triangle_map_fig}.

We can now define two sets of triangles based on the points $u_{1}',u_{2}',u_{3}'$, and~$u_{4}'$
\begin{align*}
    T_{1,\mathbb{T}} &= \left(u_{1}', u_{2}', u_{4}'  \right), \qquad T_{2,\mathbb{T}} = \left(u_{2}', u_{3}', u_{4}'  \right), \\[1.0ex]
    T_{1,\mathbb{T}^{\ast}} &= \left(u_{1}', u_{2}', u_{3}'  \right), \qquad T_{2,\mathbb{T}^{\ast}} = \left(u_{1}', u_{3}', u_{4}'  \right).
\end{align*}
These triangles compose two triangulations: $\mathbb{T} = \left\{T_{1,\mathbb{T}}, T_{2,\mathbb{T}} \right\}$ and $\mathbb{T}^{\ast} = \left\{T_{1,\mathbb{T}^{\ast}}, T_{2,\mathbb{T}^{\ast}} \right\}$. 
The triangle areas are
\begin{align*}
    \left| T_{1,\mathbb{T}} \right| &= \frac{1}{2} q, \qquad \left| T_{2,\mathbb{T}} \right| = \frac{1}{2}\left(rq-ps +s -q \right), \\[1.0ex]
    \left| T_{1,\mathbb{T}^{\ast}} \right| &= \frac{1}{2} s, \qquad \left| T_{2,\mathbb{T}^{\ast}} \right| = \frac{1}{2}\left(rq-ps \right).
\end{align*}

\subsection{Roughness and Optimality}

We can define the `roughness' of a generic, weakly-differentiable function $g(x,t)$ on the triangulation $\mathbb{T}$ as follows
\begin{align}
    \left| g \right|_{\mathbb{T}}^{2} = \sum_{j=1}^{2} \left| g\right|^{2}_{T_{j,\mathbb{T}}}, \qquad \left| g\right|^{2}_{T_{j,\mathbb{T}}} = \int_{T_{j}} \left( \wavespeed_{v} \left( \frac{\partial g}{\partial x} \right)^2 + \frac{1}{\wavespeed_{v}} \left( \frac{\partial g}{\partial t} \right)^2 \right) dx dt. \label{rough_formula}
\end{align}
In accordance with the standard approach (see Powar~\cite{powar1992minimal}), we will only consider the case where $g(x,t)$ is a simple, piecewise linear function. In particular, we seek to understand the roughness of piecewise linear functions on our triangulations $\mathbb{T}$ and $\mathbb{T}^{\ast}$.

\begin{thm}
    An anisotropic Delaunay triangulation $\mathbb{T}$ yields a smaller value of roughness than any other triangulation $\mathbb{T}^{\ast}$. More precisely,
    \begin{align*}
        \left| H \right|_{\mathbb{T}^{\ast}}^{2} - \left| G \right|_{\mathbb{T}}^{2} \geq 0,
    \end{align*}
    where $H$ and $G$ are piecewise linear functions on the triangulations $\mathbb{T}^{\ast}$ and $\mathbb{T}$, respectively. 
\end{thm}

\begin{proof}
Consider the class of piecewise linear functions which interpolate the values $f_1, f_2, f_3,$ and $f_4$ at the points $u_{1}',u_{2}',u_{3}'$, and $u_{4}'$. There are only two ways in which we can tessellate the four points (as mentioned previously), and therefore, only two piecewise linear functions can be defined. Let us denote these two functions by $G$ and $H$. They can be defined as follows
\begin{align*}
    G_{T_1} &\equiv \frac{\widetilde{a}_{1,\mathbb{T}}}{\sqrt{\wavespeed_{v}}} x + \sqrt{\wavespeed_{v}} \, \widetilde{b}_{1,\mathbb{T}} t + \widetilde{c}_{1,\mathbb{T}}, \qquad
        G_{T_2}  \equiv \frac{\widetilde{a}_{2,\mathbb{T}}}{\sqrt{\wavespeed_{v}}} x + \sqrt{\wavespeed_{v}} \, \widetilde{b}_{2,\mathbb{T}} t + \widetilde{c}_{2,\mathbb{T}}, \\[1.0ex]
        H_{T_1} &\equiv \frac{\widetilde{a}_{1,\mathbb{T}^{\ast}}}{\sqrt{\wavespeed_{v}}} x + \sqrt{\wavespeed_{v}} \, \widetilde{b}_{1,\mathbb{T}^{\ast}} t + \widetilde{c}_{1,\mathbb{T}^{\ast}}, \qquad
        H_{T_2} \equiv \frac{\widetilde{a}_{2,\mathbb{T}^{\ast}}}{\sqrt{\wavespeed_{v}}} x + \sqrt{\wavespeed_{v}} \, \widetilde{b}_{2,\mathbb{T}^{\ast}} t + \widetilde{c}_{2,\mathbb{T}^{\ast}},
\end{align*}
where
\begin{align*}
    &\widetilde{a}_{1,\mathbb{T}} = \sqrt{\wavespeed_v} f_2, \qquad \widetilde{b}_{1,\mathbb{T}} = \frac{f_4 - p f_2}{\sqrt{\wavespeed_v}q}, \\[1.0ex]
     &\widetilde{a}_{2,\mathbb{T}} =  \frac{\sqrt{\wavespeed_v}\left(q(f_3-f_2)-s(f_4-f_2) \right)}{m}, \qquad \widetilde{b}_{2,\mathbb{T}} = \frac{(r-1)(f_4-f_2) - (p-1)(f_3-f_2)}{\sqrt{\wavespeed_v}m},\\[1.0ex]
    &\widetilde{a}_{1,\mathbb{T}^{\ast}} = \sqrt{\wavespeed_v} f_2, \qquad \widetilde{b}_{1,\mathbb{T}^{\ast}} = \frac{f_3 - r f_2}{\sqrt{\wavespeed_v}s}, \\[1.0ex]
    & \widetilde{a}_{2,\mathbb{T}^{\ast}} = \frac{\sqrt{\wavespeed_v}\left(q f_3 - s f_4\right)}{rq-ps}, \qquad \widetilde{b}_{2,\mathbb{T}^{\ast}} = \frac{rf_4-pf_3}{\sqrt{\wavespeed_v}\left(rq-ps\right)},
\end{align*}
and where
\begin{align*}
    m = rq-ps +s -q = 2 \left| T_{2,\mathbb{T}} \right|.
\end{align*}
The roughness of the piecewise linear function on $\mathbb{T}$ is given by $\left| G \right|_{\mathbb{T}}$, and similarly, the roughness on $\mathbb{T}^{\ast}$ is given by $\left| H \right|_{\mathbb{T}^{\ast}}$. As a result, the relative roughness of the piecewise linear functions can be computed as follows
\begin{align}
    \left| H \right|_{\mathbb{T}^{\ast}}^{2} - \left| G \right|_{\mathbb{T}}^{2} = \sum_{j=1}^{2} \left[ \int_{T_j, \mathbb{T}^{\ast}} \left( \left( \widetilde{a}_{j,\mathbb{T}^{\ast}} \right)^2 + \left( \widetilde{b}_{j,\mathbb{T}^{\ast}} \right)^2 \right) dx dt - \int_{T_j, \mathbb{T}} \left( \left( \widetilde{a}_{j,\mathbb{T}} \right)^2 + \left( \widetilde{b}_{j,\mathbb{T}} \right)^2 \right) dx dt \right]. \label{relative_rough}
\end{align}
By inspection, there is a direct relationship between minimizing roughness (Eq.~\eqref{rough_formula}), and minimizing the energy functional $\mathcal{J}(\widetilde{v}_h)$, (Eq.~\eqref{energy_function_new}). Therefore, we seek to find a triangulation $\mathbb{T}$ which has the smallest roughness, i.e., a triangulation for which the roughness is less than any other triangulation $\mathbb{T}^{\ast}$. In order for this to hold, the difference $\left| H \right|_{\mathbb{T}^{\ast}}^{2} - \left| G \right|_{\mathbb{T}}^{2}$ must always be non-negative. With this in mind, we can expand Eq.~\eqref{relative_rough} as follows
\begin{align*}
    \left| H \right|_{\mathbb{T}^{\ast}}^{2} - \left| G \right|_{\mathbb{T}}^{2} &= \left(  \frac{1}{2} s \right) \left( \left(\sqrt{\wavespeed_v} f_2 \right)^2 + \left( \frac{f_3 - r f_2}{\sqrt{\wavespeed_v}s}\right)^2 \right) \\[1.0ex] &+ \left(\frac{1}{2}\left(rq-ps \right)\right) \left( \left( \frac{\sqrt{\wavespeed_v}\left(q f_3 - s f_4\right)}{rq-ps}\right)^2 + \left(\frac{rf_4-pf_3}{\sqrt{\wavespeed_v}\left(rq-ps\right)} \right)^2 \right) \\
    &-\left(\frac{1}{2} q\right) \left( \left(\sqrt{\wavespeed_v} f_2 \right)^2 + \left(\frac{f_4 - p f_2}{\sqrt{\wavespeed_v}q}\right)^2 \right) \\
    &-\left( \frac{1}{2}m \right) \left( \left(\frac{\sqrt{\wavespeed_v}\left(q(f_3-f_2)-s(f_4-f_2) \right)}{m} \right)^2 + \left(\frac{(r-1)(f_4-f_2) - (p-1)(f_3-f_2)}{\sqrt{\wavespeed_v}m}\right)^2 \right) \\
    &=\frac{\left(q f_3 + (ps - rq)f_2 - s f_4  \right)^{2} \left(ps(1-p) - \wavespeed_{v}^{2} q^{2}s + q(\wavespeed_{v}^{2}s^2 + r^2 -r) \right)}{2\wavespeed_v q m s(rq-ps)} \\
    &= A \cdot B \cdot C,
\end{align*}
where 
\begin{align*}
    A &= \frac{1}{2 \wavespeed_v m s(rq-ps)}, \qquad B = \left(q f_3 + (ps - rq)f_2 - s f_4  \right)^{2}, \\ C &= \frac{ps(1-p) - \wavespeed_{v}^{2} q^{2}s + q(\wavespeed_{v}^{2}s^2 + r^2 -r)}{q}.
\end{align*}
Here, $A$ is positive, as it can be rewritten as follows
\begin{align*}
    A = \frac{1}{16 \wavespeed_v \left| T_{2,\mathbb{T}} \right| \left| T_{1,\mathbb{T}^{\ast}} \right| \left| T_{2,\mathbb{T}^{\ast}} \right|}
\end{align*}
In accordance with Powar~\cite{powar1992minimal}, $B$ is a coplanarity condition, and is always non-negative---only equaling zero if all the mapped points are coplanar. Finally, $C$ is the circumcircle criterion that governs triangle $T_{1,\mathbb{T}}$ and point $u_3'$. In particular, the triangulation $\mathbb{T}$ is an anisotropic Delaunay triangulation if $u_3'$ is outside of the metric-weighted circumcircle of $T_{1,\mathbb{T}}$. In order to see this, we first map our points into metric space
\begin{align*}
    u_1' \Rightarrow u_1'' = [0,0]^{T}, \qquad u_2' \Rightarrow u_2'' = [1,0]^{T}, \qquad u_3'  \Rightarrow u_3'' = [r, \wavespeed_v s]^{T}, \qquad u_4'  \Rightarrow u_4'' = [p,\wavespeed_v q]^{T}.
\end{align*}
Next, we compute the following expressions for the metric-weighted circumcenter $o_{M}$ and circumradius $r_{M}$ of $T_{1,\mathbb{T}}$
\begin{align*}
    o_{M} = [o_x, o_t]^{T} = \left[\frac{1}{2}, \frac{p^2 + \wavespeed_{v}^{2} q^2 -p}{2 \wavespeed_v q} \right]^{T}, \qquad r_{M} = \sqrt{\left(\frac{1}{2}\right)^2 + \left(\frac{p^2 + \wavespeed_{v}^{2} q^2 -p}{2 \wavespeed_v q} \right)^2}.
\end{align*}
The circumcircle criterion can be written as follows
\begin{align*}
    \left(o_{M} - u_3''\right)^{T} \left(o_{M} - u_3''\right) - r_{M}^{2} \geq 0,
\end{align*}
where $u_3''$ is outside of the circumcircle if this expression (above) yields a positive value. Upon performing some simple algebraic manipulations, we can show that
\begin{align*}
    &\left(o_{M} - u_3''\right)^{T} \left(o_{M} - u_3''\right) - r_{M}^{2}
    =  \frac{-rq + q r^2 - (p^2 + \wavespeed_{v}^{2} q^2 - p) s + \wavespeed_{v}^{2} s^2 q}{q} \\
    &= \frac{ps(1-p) - \wavespeed_{v}^{2} q^{2}s + q(\wavespeed_{v}^{2}s^2 + r^2 -r)}{q} = C \geq 0.
\end{align*}
In summary, $A$, $B$, and $C$ are all non-negative if (and only if) the triangulation $\mathbb{T}$ is an anisotropic Delaunay triangulation. Therefore, the relative roughness is guaranteed to be non-negative. Furthermore, it immediately follows that the anisotropic Delaunay triangulation has minimal roughness. 
\end{proof}

\pagebreak
\clearpage

{\scriptsize\bibliography{technical-refs}}

\end{document}